\documentclass[11pt,a4paper]{article}

\usepackage[utf8]{inputenc}

\usepackage[usenames,dvipsnames]{xcolor}
\usepackage[margin=1.0 in]{geometry}
\usepackage{graphicx}
\usepackage{amsmath, amsthm, amssymb}
\usepackage{appendix}
\usepackage{hyperref}
\usepackage{marginnote, enumitem}
\hypersetup{
    unicode=false,          
    pdftoolbar=true,        
    pdfmenubar=true,        
    pdffitwindow=false,     
    pdfstartview={FitH},    
    pdftitle={My title},    
    pdfauthor={Author},     
    pdfsubject={Subject},   
    pdfcreator={Creator},   
    pdfproducer={Producer}, 
    pdfkeywords={keyword1} {key2} {key3}, 
    pdfnewwindow=true,      
    colorlinks=true,       
    linkcolor=blue,          
    citecolor=blue,        
    filecolor=magenta,      
    urlcolor=cyan           
}

\usepackage[square, sort, numbers]{natbib}

\numberwithin{equation}{section}

\newtheorem{lem}{Lemma}[section]
\newtheorem{theo}{Theorem}[section]
\newtheorem{prop}{Proposition}[section]
\newtheorem{cor}{Corollary}[section]
\newtheorem{rmk}{Remark}[section]
\newtheorem{dfn}{Definition}[section]
\newtheorem{asmp}{Assumption}[section]

\newtheorem{ex}{Example}[section]

\newtheorem*{theo*}{Theorem}

\usepackage{amsmath, bbm, accents}
\usepackage{amsfonts}
\newcommand{\bbR}{\mathbb{R}}

\newcommand{\bbE}{\mathbb{E}}

\newcommand{\bbN}{\mathbb{N}}

\newcommand{\scrP}{\mathcal{P}}
\newcommand{\scrM}{\mathcal{M}}

\newcommand{\calB}{\mathcal{B}}

\newcommand{\calD}{\mathcal{D}}

\newcommand{\calS}{\mathcal{S}}
\newcommand{\calF}{\mathcal{F}}
\newcommand{\calL}{\mathcal{L}}

\newcommand{\calM}{\mathcal{M}}
\newcommand{\calN}{\mathcal{N}}

\newcommand{\bbP}{\mathbb{P}}
\newcommand{\bbQ}{\mathbb{Q}}
\newcommand{\E}[2][]{\mathbb{E}_{#1}\left[#2\right]}

\newcommand{\crochetb}[1]{\left\lbrack{#1}\right\rbrack}
\newcommand{\crochet}[1]{\left\langle{#1}\right\rangle}
\newcommand{\accro}[1]{\left\{{#1}\right\}}

\newcommand{\abs}[1]{\left| {#1} \right|}

\newcommand{\Kummer}{{}_1F_1}

\newcommand{\wt}[1]{\widetilde{#1}}
\newcommand{\ind}[1]{\mathbbm{1}_{\accro{#1}}}

\title{Asymptotic Lower Bounds for Optimal Tracking: \\a Linear Programming Approach }
\date{}

\author{Jiatu Cai$^{1}$, Mathieu Rosenbaum$^{2}$ and Peter Tankov$^{1}$
\\[0.4cm]
{\normalsize
$^1$ Laboratoire de Probabilit\'es et Mod\`eles Al\'eatoires,} \\ {\normalsize
    Universit\'e Paris Diderot (Paris 7)} \\
$~~$\\
{\normalsize $^2$ Laboratoire de Probabilit\'es et Mod\`eles Al\'eatoires,} \\ {\normalsize
    Universit\'e Marie et Pierre Curie (Paris 6)}
}

\begin{document}
\maketitle
\begin{abstract}
 \noindent We consider the problem of tracking a target whose dynamics is modeled by a continuous It\=o semi-martingale. The aim is to minimize both deviation from the target and tracking efforts. 
We establish the existence of  asymptotic lower bounds for this problem, depending on the cost structure. These lower bounds can be related to the time-average control of Brownian motion, which is characterized as a deterministic linear programming problem. A comprehensive list of examples with explicit expressions for the lower bounds is provided. 


\end{abstract}



\section{Introduction}

We consider the problem of tracking a target whose dynamics $(X^\circ_t)$ is modeled by a continuous It\=o semi-martingale defined on a filtered probability space $(\Omega, \calF, (\calF_t)_{t\geq 0}, \bbP)$  with values in $\bbR^d$ such that
\begin{equation*}
dX^\circ_t =  b_t dt+ \sqrt{ a}_t dW_t,\quad X^\circ_0= 0.
\end{equation*}
Here, $(W_t)$ is a $d$-dimensional Brownian motion and $(b_t)$, $(a_t)$ are  predictable processes with values in $\bbR^d$ and $\calS^d_+$ the set of $d\times d$ symmetric positive definite matrices respectively. An agent observes $X^\circ_t$  and adjusts her position in order to follow $X^\circ_t$.  However, she has to pay certain intervention costs for position adjustments. The objective is to stay close to the target  $X^\circ_t$ while minimizing the tracking efforts. This problem arises naturally in various situations such as discretization of option hedging strategies \cite{fukasawa2011asymptotically, rosenbaum2011asymptotically, gobet2012almost}, management of an index fund \cite{pliska2004optimal,  korn1999some}, control of exchange rate \cite{mundaca1998optimal, cadenillas2000classical},  portfolio selection under transaction costs \cite{kallsen2013general, Soner2012, possamai2012homogenization, altarovici2013asymptotics}, trading under market impact \cite{bank2015hedging,moreau2014trading, liu2014rebalancing} or illiquidity costs \cite{naujokat2011curve,Rogers2007}. \\

More precisely, {\color{black}let $(Y^\psi_t)$ be the position of the agent determined by the control $\psi$, with $Y^\psi_t\in \mathbb R^d$} and let $(X_t)$ be  the deviation of the agent from the target $(X^\circ_t)$, so that
\begin{equation}\label{eqn: controlled process}
X_t = -X_t^\circ + Y^\psi_t.
\end{equation}
Let $H_0(X)$ be the penalty functional for the deviation from the target and $H(\psi)$  the cost incurred by the control $\psi$ up to a finite horizon $T$.  Then the problem of tracking can be formulated as 
\begin{equation}\label{eqn: cost psi}
\inf_{\psi \in \mathcal{A}} J(\psi), \quad  J(\psi)= H_0(X) + H(\psi),
\end{equation}
{\color{black}where $\mathcal{A}$ is the set of admissible strategies.}
As is usually done in the literature (see for example \cite{karatzas1983class, menaldi1992singular}), we consider a penalty $H_0(X)$ for the deviation from the target of additive  form
\begin{equation*}
H_0(X) = \int_0^T r_tD( X_t)dt,
\end{equation*}
where $(r_t)$ is a random weight process and $D(x)$ a determinstic function. For example, we can take $D(x) = \langle x, \Sigma^D x \rangle $ where $\Sigma^D$ is positive definite and $\langle\cdot, \cdot\rangle$ is the inner product in $\bbR^d$.
On the other hand, depending on the nature of  the costs, the agent can either control her speed at all times or jump towards the target instantaneously.  The control $\psi$ and the cost functional $H(\psi)$ belong to one of the following classes:

\begin{enumerate}
\item
\emph{Impulse Control.} There is a fixed cost component for each action, so the agent has to intervene in a discrete way. {\color{black}The class $\mathcal A$ of admissible controls contains all sequences $\{(\tau_j,\xi_j),j\in \bbN^*\}$ where $\{\tau_j,j\in \bbN^*\}$ is a strictly increasing sequence of stopping times representing the jump times and satisfying $\lim_{j\to \infty} \tau_j = +\infty$, and for each $j\in \bbN^*$, $\xi_j\in \bbR^d$ is a $\calF_{\tau_j}$-measurable random vector representing the size of $j$-th jump. The position of the agent is given by}
\begin{equation*}
Y_t = \sum_{0< \tau_j \leq t} \xi_j
\end{equation*}
and the cumulated cost is then given by
\begin{equation*}
H(\psi) = \sum_{0< \tau_j \leq T} k_{\tau_j}F( \xi_j), 
\end{equation*}
where $(k_t)$ is a random weight process and $F(\xi)>0$ is the cost of a jump with size $\xi\neq 0$. 
If we take $k_t= 1$ and $F(\xi)= \sum_{i=1}^d\ind{\xi^i\neq 0}$ where $\xi^i$ is the i-th component of $\xi$, then $H(\psi)$ represents the total number of actions on each component over the time interval $[0, T]$, see \cite{fukasawa2011asymptotically, gobet2012almost}. If $F(\xi) =\sum_{i=1}^d\ind{\xi^i\neq 0} + \sum_{i=1}^dP_i |\xi^i|$ where $P_i \geq 0$, we say that the cost has a fixed component and a proportional component. 
\item

\emph{Singular Control.} If the cost is proportional to the size of the jump, then infinitesimal displacement is also allowed and it is natural to model $(Y_t)$ by a process with bounded variation. {\color{black}In this case, the class $\mathcal A$ of admissible controls contains all couples $(\gamma,\varphi)$ where $\varphi$ is a progressively measurable increasing process with $\varphi_{0-}=0$, which represents the cumulated amount of intervention and $\gamma$ is a progressively measurable process with $\gamma_t \in \Delta:=\{n\in \bbR^d |\sum_{i=1}^d |n^i| = 1\}$ for all $t\geq 0$, which represents the distribution of the control effort in each direction. In other words, $\varphi_t=\sum_{i=1}^d \|Y^i\|_t$ where $\|\cdot\|$ denotes the absolute variation of a process, and $\gamma^i_t$ is the Radon-Nikodym derivative of $Y^i_t$ with respect to $\varphi_t$.  The position of the agent is given by}
\begin{equation*}
Y_t = \int_0^t \gamma_s d\varphi_s,
\end{equation*}
and the corresponding cost is usually given as (see for example \cite{kallsen2013general, Soner2012})
\begin{equation*}
H(\psi) = \int_0^T h_t P(\gamma_t) d\varphi_t,
\end{equation*}
where $(h_t)$ is a random  weight process and we take (for example) $P(\gamma) =\crochet{P, |\gamma|}$ with $P\in \bbR_+^d$ and $|\gamma|=(|\gamma^1|, \cdots, |\gamma^d|)^T$. The vector $P = (P_1, \cdots, P_d)^T$ represents the  coefficients of proportional costs in each direction. 
\item 
\emph{Regular Control.} Most often,  the process $(Y_t)$ is required to be absolutely continuous with respect to time, see for example \cite{Rogers2007, moreau2014trading} among many others. {\color{black}In this case, the class $\mathcal A$ of admissible controls contains all progressively measurable integrable processes $u$ with values in $\bbR^d$,  representing the speed of the agent, the position of the agent is given by
\begin{equation*}
Y_t = \int_0^t u_s ds,
\end{equation*}
and the cost functional is}
\begin{equation*}
H(\psi) = \int_0^T l_t Q( u_t) dt, 
\end{equation*}
where $(l_t)$ is a random weight process and, for example,  $Q(u)=  \langle u,  \Sigma^Q u\rangle $ with $\Sigma^Q$ a positive definite matrix. Comparing to the case of singular control where the control variables are $(\gamma_t)$ and $(\varphi_t)$, here we optimize over $(u_t)$. 
\item 
\emph{Combined control.} It is possible that several types of control are available to the agent. {\color{black}In that case, $\psi = (\psi_1,\dots,\psi_n)$ where for each $i$, $\psi_i$ belongs to one of the classes introduced before.} For example, in the case of combined regular and impulse control (see \cite{mundaca1998optimal}), {\color{black}the position of the agent is given by} 
\begin{equation*}
Y_t = \sum_{0< \tau_j \leq t} \xi_j  + \int_0^t u_s ds, 
\end{equation*}
while the cost functional is given by
\begin{equation*}
H(\psi) =  \sum_{0< \tau_j \leq T} k_{\tau_j}F( \xi_j) + \int_0^T l_t Q(u_t) dt. 
\end{equation*}
Similarly, one can consider other combinations of controls. \\
\end{enumerate}

The problem \eqref{eqn: controlled process}-\eqref{eqn: cost psi} rarely admits an  explicit solution. In this paper, we propose an asymptotic framework where the tracking costs are small and derive a lower bound for \eqref{eqn: controlled process}-\eqref{eqn: cost psi} under this setting. {\color{black} More precisely
we introduce a parameter $\varepsilon$ tending to zero and consider a family of cost functionals $H^\varepsilon(\psi)$. For example, we can have $H^\varepsilon(\psi) = \varepsilon^{\beta_\psi}H(\psi)$ for some constant $\beta_\psi$, but different components of the cost functional may also scale with $\varepsilon$ at different rates. We define the control problem
\begin{equation}\label{eqn: controlled process eps}\tag{\ref{eqn: controlled process}-$\varepsilon$}
X^\varepsilon_t = -X^\circ_t + Y^{\psi^\varepsilon}_t,
\end{equation}
and objective function 
\begin{equation}\label{eqn: cost psi eps} \tag{\ref{eqn: cost psi}-$\varepsilon$}
\inf_{\psi^\epsilon\in \mathcal{A}} J^\varepsilon(\psi^\varepsilon), \quad J^\varepsilon(\psi^\varepsilon) = H_0(X^\varepsilon) + H^\varepsilon(\psi^\varepsilon).
\end{equation}
Moreover, we assume that the functions $D$, $Q$, $F$, $P$ possess a homogeneity property.}\\

The main result of this paper is a precise asymptotic relation between $J^\varepsilon$ and the time-average control problem of Brownian motion with constant parameters, {\color{black}in a variety of settings.} Let us give a flavor of the main result in the case of combined regular and impulse control (note that situations involving singular control are considered in Section \ref{sec: extensions part I}).
In this case, the dynamics of the controlled Brownian motion is given by 
\begin{equation}\label{eqn: local controlled process}\tag{\ref{eqn: controlled process}-local}
X_s = \sqrt{a} W_s + \int_0^s u_\nu d\nu+ \sum_{0 < \tau_j \leq s} \xi_j,
\end{equation}
and the time-average control problem can be formulated as
\begin{equation}\label{eqn: local cost}\tag{\ref{eqn: cost psi}-local}
\widetilde I(a,r,l,k) = \inf_{(\tau_j, \xi_j,u)} \underset{S\to \infty}{\overline{\text{lim}}} \frac{1}{S} \mathbb{E}\Big[\int_0^S  \big(rD(X_s) + lQ(u_s)\big) ds + k\!\!\!\sum_{0 < \tau_j\leq S } F(\xi_j)  
\Big].
\end{equation}

{\color{black}At the level of generality that we are interested in, we need to consider a relaxed formulation of the above control problem, as a linear programming problem on the space of measures.
Following \cite{kurtz2001stationary}, we introduce the occupation measures
\begin{align*}
\mu_t(H_1) &= \frac{1}{t}\mathbb E \int_0^t \mathbbm{1}_{H_1}(X_s, u_s) ds,\quad H_1 \in \mathcal B(\mathbb R^d \times \mathbb R^d),\\
\rho_t(H_2) &= \frac{1}{t} \mathbb E \sum_{0<\tau_j \leq t} \mathbbm{1}_{H_2}(X_{s-},\xi_j),\quad H_2 \in \mathcal B(\mathbb R^d \times \mathbb R^d).
\end{align*}
If the process $X$ and the controls are stationary, these measures do not depend on time and therefore
\begin{multline*}
\underset{S\to \infty}{\overline{\text{lim}}} \frac{1}{S} \mathbb{E}\Big[\int_0^S  \big(rD(X_s) + lQ(u_s)\big) ds + k\!\!\!\sum_{0 < \tau_j\leq S } F(\xi_j)  
\Big] \\ = \int_{\mathbb R^d \times \mathbb R^d} (rD(x) + lQ(u))  \mu(dx\times du) + \int_{\mathbb R^d \times \mathbb R^d} k F(\xi) \rho (dx \times d\xi). 
\end{multline*}
On the other hand, by It\=o's formula, for any $f\in C^2_0$, 
\begin{align}
f(X_t)  = f(X_0) + \sqrt{a}\int_0^t f'(X_s)  dW_s + \int_0^t Af (X_s,u_s) ds +  \sum_{0<\tau_j \leq t} B f(X_{\tau_j-},\xi_j) 
,\label{ito.eq}
\end{align}
where 
\begin{align*}
Af(x, u)& = \frac{1}{2}\sum_{i, j}a_{ij}\partial^2_{ij}f(x) + \langle u, \nabla f(x)\rangle, \\
B f(x,\xi) &= f(x+\xi) - f(x).
\end{align*}
Taking expectation in \eqref{ito.eq} and assuming once again the stationarity of controls, we  see that under adequate integrability conditions the measures $\mu$ and $\rho$ satisfy the constraint
\begin{align}
\int_{\mathbb R^d \times \mathbb R^d} A f(x,u) \mu(dx \times du) + \int_{\mathbb R^d \times \mathbb R^d} B f(x,\xi) \rho(dx\times d\xi)  
= 0.\label{cons.eq}
\end{align}
Therefore, the time-average control problem of Brownian motion \eqref{eqn: local controlled process}-\eqref{eqn: local cost} is closely related to the problem of computing 
\begin{align}
I(a,r,l,k) = \inf_{\mu,\rho} \int_{\mathbb R^d \times \mathbb R^d} (rD(x) + lQ(u))  \mu(dx\times du) + \int_{\mathbb R^d \times \mathbb R^d} k F(\xi) \rho (dx \times d\xi) 
,\label{lp.eq}
\end{align}
where $\mu$ is a probability measure and $\rho$ is a finite positive measure satisfying the constraint \eqref{cons.eq}. In Section \ref{sec: explicit I}  we shall see that this characterization is essentially equivalent to \eqref{eqn: local controlled process}-\eqref{eqn: local cost} if we formulate the optimal control problem for the Brownian motion as a controlled martingale problem. In the considered case of combined regular and impulse control, our main result is the following. \\

\noindent 
{\bf Main result, combined regular and impulse control. }{\it There exists $\beta^*>0$ explicitly determined by the cost structure $H_0$ and $(H^\varepsilon)_{\varepsilon> 0}$ such that for all $\delta >0$ and any sequence of admissible strategies $\{\psi^\varepsilon \in \mathcal{A}, \varepsilon >0\}$, we have
\begin{equation}\label{eqn: lower bound}
\lim_{\varepsilon\to 0+}\bbP\Big[\frac{1}{\varepsilon^{\beta^*}} J^\varepsilon(\psi^\varepsilon)\geq \int_0^T I_t dt-\delta\Big] = 1,
\end{equation}
where $I_t = I(a_t, r_t, k_t, l_t) $ is the optimal cost of the linear programming formulation \eqref{lp.eq} of the time-average control problem of Brownian motion \eqref{eqn: local controlled process}-\eqref{eqn: local cost} with parameters frozen at time $t$.}\\
%

\noindent
 The original problem \eqref{eqn: controlled process}-\eqref{eqn: cost psi} is therefore simplified in the sense that the local problem \eqref{eqn: local controlled process}-\eqref{eqn: local cost} is easier to analyze since the dynamics of the target is reduced to that of a Brownian motion and the cost parameters become constant.  In many practically important cases (see Examples \ref{ex: stochastic}-\ref{ex: combined singular}), we are able to solve explicitly \eqref{eqn: local controlled process}-\eqref{eqn: local cost} and show that the two formulations of the time-average control problem are equivalent and therefore $I = \widetilde I$.  Moreover, in a forthcoming paper, we show that for the examples where \eqref{eqn: local controlled process}-\eqref{eqn: local cost} admits an explicit solution, the lower bound is tight (see Remark \ref{upb}).} \\

Our result enables us to revisit the asymptotic lower bounds for the discretization of hedging strategies in \cite{fukasawa2011asymptotically, gobet2012almost}. In these papers, the lower bounds are deduced by using subtle inequalities. Here we show that these bounds can be interpreted in a simple manner through the time-average control problem of Brownian motion. \\

The local control problem \eqref{eqn: local controlled process}-\eqref{eqn: local cost} also arises in the study of utility maximization under transaction costs, see \cite{Soner2012, possamai2012homogenization, altarovici2013asymptotics, moreau2014trading}. This is not surprising since at first order, these problems and the tracking problem are essentially the same, see Section \ref{sec: util}. In the above references, the authors derive the PDE associated to the first order correction of the value function, which turns out to be the HJB equation associated to the time-average control of Brownian motion. Inspired by \cite{kurtz1999martingale} and \cite{kushner1993limit}, our approach, based on weak convergence of empirical occupation measures, is very different from the PDE-based method and enables us to treat more general situations. Contrary to \cite{kurtz1999martingale}, where the lower bound holds under expectation, we obtain pathwise lower bounds. Compared to \cite{kushner1993limit}, we are able to treat impulse control and general dynamics for the target.\\

The  paper is organized as follows. In Section \ref{sec: asymp framework}, we introduce our asymptotic framework and establish heuristically the lower bound for the case of combined regular and impulse control. Various extensions are then discussed in Section \ref{sec: extensions part I}. In Section \ref{sec: explicit I}, we provide an accurate definition for the time-average control of Brownian motion using a relaxed martingale formulation and collect a comprehensive list of explicit solutions in dimension one. The connection with utility maximization with small market frictions is made in Section \ref{sec: util} and the proofs are given in Sections \ref{sec: proof}, \ref{sec: proof singular} and \ref{sec: verification}.\\

\emph{Notation.} For a complete, separable, metric space $S$, we define $C(S)$  the
set of continuous functions on $S$, $C_b(S)$  the set of bounded, continuous functions on
$S$, $\scrM(S)$ the set of finite nonnegative Borel measures on $S$ and $\scrP(S)$  the set of probability measures on $S$. The sets $\scrM(S)$ and $\scrP(S)$ are equipped with the topology of weak convergence.
We define $\calL(S)$ the set of nonnegative Borel measures $\Gamma$ on $S \times [0, \infty)$ such that $\Gamma(S \times [0, t]) < \infty$. 
  Denote $\Gamma_t$  the restriction of $\Gamma$ to $S\times [0, t]$.  
 We use $\bbR^d_x$,  $\bbR^d_u$ and  $\bbR^d_\xi$ to indicate the state space corresponding to the variables $x$, $u$ and $\xi$. Finally, $C^2_0(\bbR^d)$ denotes the set of twice differentiable real functions on $\bbR^d$ with compact support, equipped with the norm
\begin{equation*}
\|f\|_{C^2_0}= \|f\|_{\infty} + \sum_{i=1}^d \|\partial_if\| _{\infty}+ \sum_{i,j=1}^d\|\partial^2_{ij}f\|_{\infty}.
\end{equation*}


\section{Tracking with combined regular and impulse control}\label{sec: asymp framework}
Instead of giving directly a general result, which would lead to a cumbersome presentation, we focus in this section on the tracking problem with combined regular and impulse control. This allows us to illustrate our key ideas. Other situations, such as singular control,  are discussed in Section \ref{sec: extensions part I}.  \\


In the case of combined regular and impulse control,  a tracking strategy $(u, \tau, \xi)$ is given by a progressively measurable process $u=(u_t)_{t\geq 0}$ with values in $\bbR^d$ and  $(\tau, \xi) =\{(\tau_j, \xi_j),  j\in \bbN^*\}$, with $(\tau_j)$ an increasing sequence  of stopping times and $(\xi_j)$  a sequence of $\calF_{\tau_j}$-measurable random variables with values in $\bbR^d$. The process $(u_t)$ represents the speed of the agent. The stopping time $\tau_j$ represents the timing of j-th jump towards the target and $\xi_j$ the size of the jump. The tracking error obtained  by following the strategy $(u, \tau, \xi)$  is given by 
\begin{equation*}
X_t = -X_t^\circ + \int_0^t u_s ds + \sum_{0<\tau_j \leq t}\xi_j.
\end{equation*}

\noindent At any time the agent is paying a cost for maintaining the speed $u_t$ and each jump $\xi_j$  incurs a  positive cost. We are interested in the following type of cost functional
\begin{equation*}
J(u, \tau, \xi) = \int_0^T\big( r_t D(X_t) + l^\circ_t Q(u_t)\big) dt + \sum_{j: 0<\tau_j\leq T}\big(k^\circ_{\tau_j}F(\xi_j) + h^\circ_{\tau_j}P(\xi_j)\big),	
\end{equation*}
where $T\in   \bbR^+$ and $(r_t)$, $(l^\circ_t)$, $(k^\circ_t)$ and $(h^\circ_t)$ are random weight processes. The cost functions $D$, $Q$, $F$, $P$ are deterministic functions which satisfy the following homogeneity property
\begin{equation}\label{eqn: homogeneous cost}
D(\varepsilon x) = \varepsilon^{\zeta_D}D(x), \quad Q(\varepsilon u) = \varepsilon^{\zeta_Q}Q(u), \quad F(\varepsilon \xi) = \varepsilon^{\zeta_F}F(\xi), \quad P(\varepsilon \xi) = \varepsilon^{\zeta_P} P(\xi),
\end{equation}
for any $\varepsilon > 0$ and 
\begin{equation*}
\zeta_D >0,\quad \zeta_Q > 1, \quad\zeta_F=0, \quad\zeta_P = 1.
\end{equation*}
Note that here we slightly extend the setting of the previous section by introducing two functions $F$ and $P$ which typically represent the fixed and the proportional costs respectively. \\

In this paper, we  essentially have in mind the case where
\begin{equation*}
D(x) = \langle x,  \Sigma^D x \rangle , \quad Q(u) = \langle u, \Sigma^Q u\rangle , \quad F(\xi) = \sum_{i=1}^d F_i \ind{\xi^i \neq 0}, \quad P(\xi) = \sum_{i=1}^d P_i |\xi^i|,
\end{equation*}
with $F_i, P_i \in  \mathbb{R}_+$ such that
\begin{equation}\label{eqn: F i pos}
\min_i F_i >0,
\end{equation}
 and $\Sigma^D, \Sigma^Q\in \calS^d_+$. Note that in this situation, we have $\zeta_D=\zeta_Q = 2$.  


\subsection{Asymptotic framework} \label{sec: asym framework}
We now explain our asymptotic setting where the costs are small and provide a heuristic proof of our main result. We assume that there exist  $\varepsilon > 0$ and  $\beta_Q, \beta_F, \beta_P>0$ such that
\begin{equation}\label{eqn: scaled param}
l^\circ_t = \varepsilon^{\beta_Q} l_t, \quad k^\circ_t = \varepsilon^{\beta_F} k_t, \quad h^\circ_t =\varepsilon^{\beta_P}h_t.
\end{equation}
Then the asymptotic framework consists in considering the sequence of optimization problems indexed by $\varepsilon\to0$ 
\begin{equation*}
\inf_{(u^\varepsilon, \tau^\varepsilon, \xi^\varepsilon)\in \mathcal{A}}J^\varepsilon(u^\varepsilon, \tau^\varepsilon, \xi^\varepsilon), 
\end{equation*}
with
\begin{equation*}
J^\varepsilon(u^\varepsilon, \tau^\varepsilon, \xi^\varepsilon)  = \int_0^T \big(r_t D(X_t^\varepsilon) + \varepsilon^{\beta_Q}l_t Q(u^\varepsilon_t)\big) dt + \sum_{j: 0 <\tau^\varepsilon_j \leq T}\big( \varepsilon^{\beta_F}k_{\tau^\varepsilon_j}F(\xi_j^\varepsilon) + \varepsilon^{\beta_P}h_{\tau^\varepsilon_j}P(\xi_j^\varepsilon)\big),	
\end{equation*}
and
\begin{equation*}
X_t^\varepsilon = -X_t^\circ +\int_0^t u^\varepsilon_s ds +  \sum_{j:0<\tau^\varepsilon_j \leq t}\xi_j^\varepsilon.
\end{equation*}

\noindent The key observation is that under such setting, the tracking problem  can be decomposed into a sequence of local problems. More precisely, 
let $\{ t^\varepsilon_k = k\delta^\varepsilon, k=0, 1, \cdots, K^\varepsilon\}$ be a partition of the interval $[0, T]$ with $\delta^\varepsilon\to 0$ as $\varepsilon \to 0$.
Then we can write
\begin{align*}
J^\varepsilon(u^\varepsilon, \tau^\varepsilon, \xi^\varepsilon) 
&= \sum_{k=0}^{K^\varepsilon-1}\Big( \int_{t^\varepsilon_k}^{t^\varepsilon_k+\delta^\varepsilon} \big(r_t D(X_t^\varepsilon) + \varepsilon^{\beta_Q}l_t Q(u^\varepsilon_t)\big) dt \\
&\quad\qquad\qquad + \sum_{j: t^\varepsilon_k <\tau^\varepsilon_j \leq t^\varepsilon_k+\delta^\varepsilon}\big( \varepsilon^{\beta_F}k_{\tau^\varepsilon_j}F(\xi_j^\varepsilon) + \varepsilon^{\beta_P}h_{\tau^\varepsilon_j}P(\xi_j^\varepsilon)\big)\Big)\\
&= \sum_{k=0}^{K^\varepsilon-1}\frac{1}{\delta^\varepsilon}\Big( \int_{t^\varepsilon_k}^{t^\varepsilon_k+\delta^\varepsilon} \big(r_t D(X_t^\varepsilon) + \varepsilon^{\beta_Q}l_t Q(u^\varepsilon_t)\big) dt\\
&\qquad\qquad\qquad  +\sum_{j: t^\varepsilon_k <\tau^\varepsilon_j \leq t^\varepsilon_k+\delta^\varepsilon}\big( \varepsilon^{\beta_F}k_{\tau^\varepsilon_j}F(\xi_j^\varepsilon) + \varepsilon^{\beta_P}h_{\tau^\varepsilon_j}P(\xi_j^\varepsilon)\big)\Big)(t^\varepsilon_{k+1} - t^\varepsilon_k)\\
&=\sum_{k=0}^{K^\varepsilon -1} j^\varepsilon_{t^\varepsilon_k}(t^\varepsilon_{k+1} -t^\varepsilon_k),
\end{align*}
with
\begin{equation*}
j^\varepsilon_{t^\varepsilon_k}=\frac{1}{\delta^\varepsilon}\Big( \int_{t^\varepsilon_k}^{t^\varepsilon_k+\delta^\varepsilon} \big(r_t D(X_t^\varepsilon) + \varepsilon^{\beta_Q}l_t Q(u^\varepsilon_t)\big) dt+\sum_{j: t^\varepsilon_k <\tau^\varepsilon_j \leq t^\varepsilon_k+\delta^\varepsilon}\big( \varepsilon^{\beta_F}k_{\tau^\varepsilon_j}F(\xi_j^\varepsilon) + \varepsilon^{\beta_P}h_{\tau^\varepsilon_j}P(\xi_j^\varepsilon)\big)\Big).
\end{equation*}
 As $\varepsilon$ tends to zero, we approximately have
\begin{equation*}
J^\varepsilon(u^\varepsilon, \tau^\varepsilon, \xi^\varepsilon)  \simeq \int_0^T j^\varepsilon_t dt.
\end{equation*}

\noindent We are hence led to study $j^\varepsilon_t$ as $\varepsilon\to 0$, which is closely related to the time-average control problem of Brownian motion. To see this, consider the following rescaling of $X^\varepsilon$ over the horizon $(t, t+\delta^\varepsilon]$:
\begin{equation*}
\wt{X}^{\varepsilon, t}_s = \frac{1}{\varepsilon^{\beta}}X^\varepsilon_{t+ \varepsilon^{\alpha\beta}s}, \quad s\in (0, T^\varepsilon],
\end{equation*}
where $T^\varepsilon= \varepsilon^{-\alpha\beta}\delta^{\varepsilon}$, $\alpha = 2$ and $\beta >0$ is to be determined (here $\alpha=2$ is related to the scaling property of Brownian motion).  We use the superscript $t$ to indicate that the scaled systems correspond to the horizon $(t, t+\delta^\varepsilon]$. Then the dynamics of $\wt{X}^{\varepsilon, t}$ is given by, see \cite[Proposition V.1.5]{Revuz1999},
\begin{equation*}\label{pb: local dynamics integral}
\wt{X}^{\varepsilon, t}_s =\wt{X}^{\varepsilon, t}_0+ \int_0^s \wt{b}^{\varepsilon, t}_\nu d\nu + \int_0^s\sqrt{\wt{a}^{\varepsilon, t}_\nu}d\wt{W}^{\varepsilon, t}_\nu + \int_0^s\wt{u}^{\varepsilon, t}_\nu d\nu+ \sum_{0< \wt{\tau}^{\varepsilon, t}_j \leq s}\wt{\xi}^{\varepsilon}_j,
\end{equation*}
with
$$
\wt{b}^{\varepsilon, t}_s = -\varepsilon^{(\alpha-1)\beta} b_{t+ \varepsilon^{\alpha\beta}s},\quad \wt{a}^{\varepsilon, t}_s = a_{t + \varepsilon^{\alpha\beta}s},\quad 
\wt{W}^{\varepsilon, t}_s = -\frac{1}{\varepsilon^\beta}(W_{t + \varepsilon^{\alpha \beta}s}-W_t),$$
and 
\begin{equation*} \wt{u}^{\varepsilon, t}_s = \varepsilon^{(\alpha-1)\beta} u^\varepsilon_{t+ \varepsilon^{\alpha \beta}s},  \quad \wt{\xi}^{\varepsilon}_j= \frac{1}{\varepsilon^{\beta}}\xi^\varepsilon_j, \quad \wt{\tau}^{\varepsilon, t}_j = \frac{1}{\varepsilon^{\alpha \beta}}(\tau^\varepsilon_j-t)\vee 0.
\end{equation*}
Note that $(\wt{W}^{\varepsilon, t}_s)$ a Brownian motion with respect to $\wt{\calF}^{\varepsilon, t}_s = \calF_{t + \varepsilon^{\alpha\beta}s}$.
Abusing notation slightly, we write
\begin{equation}\label{pb: local dynamics}
d\wt{X}^{\varepsilon, t}_s = \wt{b}^{\varepsilon, t}_s ds + \sqrt{\wt{a}^{\varepsilon, t}_s}d\wt{W}^{\varepsilon, t}_s + \wt{u}^{\varepsilon, t}_sds+ d(\sum_{0< \wt{\tau}^{\varepsilon, t}_j \leq s}\wt{\xi}^{\varepsilon}_j), \quad s\in (0, T^\varepsilon].
\end{equation}
Using the homogeneity properties \eqref{eqn: homogeneous cost} of the cost functions, we obtain 
\begin{align*}
j^\varepsilon_t 
&= \frac{1}{T^\varepsilon}\Big(
\int_0^{T^\varepsilon}\big( \varepsilon^{\beta \zeta_D}r_{t+\varepsilon^{\alpha\beta}s} D(\wt{X}^{\varepsilon,t}_s) + \varepsilon^{\beta_Q - (\alpha-1)\zeta_Q\beta}l_{t+\varepsilon^{\alpha\beta}s} Q(\wt{u}^{\varepsilon, t}_s)\big)ds\\
&\qquad\qquad + \sum_{0<\wt{\tau}^{\varepsilon,t}_j\leq T^\varepsilon }\big(\varepsilon^{\beta_F-(\alpha-\zeta_F)\beta} k_{t + \varepsilon^{\alpha\beta}\wt{\tau}^{\varepsilon,t}_j}F(\wt{\xi}^\varepsilon_j)+ \varepsilon^{\beta_P-(\alpha-\zeta_P)\beta} h_{t + \varepsilon^{\alpha\beta}\wt{\tau}^{\varepsilon,t}_j}P(\wt{\xi}^\varepsilon_j)\big)\Big)\\
&\simeq \frac{1}{T^\varepsilon}\Big(
\int_0^{T^\varepsilon}\big( \varepsilon^{\beta \zeta_D}r_{t} D(\wt{X}^{\varepsilon,t}_s) + \varepsilon^{\beta_Q - (\alpha-1)\zeta_Q\beta}l_t Q(\wt{u}^{\varepsilon, t}_s)\big)ds\\
&\qquad\qquad + \sum_{0<\wt{\tau}^{\varepsilon,t}_j\leq T^\varepsilon }\big(\varepsilon^{\beta_F-(\alpha-\zeta_F)\beta} k_{t}F(\wt{\xi}^\varepsilon_j)+ \varepsilon^{\beta_P-(\alpha-\zeta_P)\beta} h_{t}P(\wt{\xi}^\varepsilon_j)\big)\Big).
\end{align*}
The second approximation can be justified by the continuity of the cost coefficients $r_t$, $l_t$, 
$k_t$ and $h_t$.\\

Now, if there exists $\beta>0$ such that
\begin{equation*}
\beta \zeta_D = \beta_Q - (\alpha-1) \zeta_Q \beta = \beta_F - (\alpha-\zeta_F)\beta = \beta_P -  (\alpha-\zeta_P) \beta, 
\end{equation*}
that is,
\begin{equation}\label{eqn: order beta}
\beta = \frac{\beta_F}{\zeta_D + \alpha - \zeta_F} = \frac{\beta_P}{\zeta_D + \alpha - \zeta_P}=\frac{\beta_Q}{\zeta_D + (\alpha-1)\zeta_Q}, 
\end{equation}
where $ \alpha=2$, then we have
\begin{equation*}
j_t^\varepsilon \simeq \varepsilon^{\beta \zeta_D}I^\varepsilon_t,
\end{equation*}
with
\begin{equation}\label{eqn: bm cost eps}
{I}^\varepsilon_t =\frac{1}{T^\varepsilon}\Big(
\int_0^{T^\varepsilon}\big(r_{t} D(\wt{X}^{\varepsilon,t}_s) +l_t Q(\wt{u}^{\varepsilon, t}_s)\big)ds+ \sum_{0<\wt{\tau}^{\varepsilon,t}_j\leq T^\varepsilon }\big(k_{t}F(\wt{\xi}^\varepsilon_j)+  h_{t}P(\wt{\xi}^\varepsilon_j)\big)\Big).
\end{equation}
By suitably choosing $\delta^\varepsilon$, we have
\begin{equation*}
\delta^\varepsilon \to 0, \quad T^\varepsilon \to \infty. 
\end{equation*}
It follows that $\wt{b}^{\varepsilon, t}_s\simeq 0$ and $\wt{a}^{\varepsilon, t}_s\simeq a_t$ for $s\in (0, T^\varepsilon]$. Therefore, the dynamics of \eqref{pb: local dynamics} is approximately a controlled Brownian motion with diffusion matrix $a_t$. We deduce that 
\begin{equation}\label{eqn: lower bound I eps}
I^\varepsilon_t \gtrsim I(a_t, r_t, l_t, k_t, h_t),
\end{equation}
 where the term in the right-hand side is defined by 
 \begin{equation}\label{eqn: ta cost const pathwise}
I(a, r, l, k, h)=\inf_{(u, \tau, \xi)}{\limsup_{S\to \infty}}\frac{1}{S}\Big[\int_0^S \big(rD(X_s) + l Q( u_s)\big) ds + \sum_{0\leq \tau_j \leq S} \big(k F(\xi_j) + hP(\xi_j)\big)\Big],
\end{equation}
with
\begin{equation}\label{eqn: ta dynm const}
{\color{black}X_s = \sqrt{a} W_s  + \int_0^s u_rdr + \sum_{0 \leq \tau_j \leq s }\xi_j .}
\end{equation}
Therefore, we obtain that as $\varepsilon \to 0$:
\begin{equation*}
J^\varepsilon \simeq \int_0^T j^\varepsilon_t dt  \simeq  \varepsilon^{\beta \zeta_D} \int_0^T I^\varepsilon_t dt \gtrsim  \varepsilon^{\beta \zeta_D}\int_0^T I_t dt.
\end{equation*}

\noindent Then we may expect that \eqref{eqn: ta cost const pathwise} is equal to the following expected cost criterion
 \begin{equation}\label{eqn: ta cost const}
I(a, r, l, k, h)=\inf_{(u, \tau, \xi)}{\limsup_{S\to \infty}}\frac{1}{S}\bbE\Big[\int_0^S \big(rD(X_s) + l Q( u_s)\big) ds + \sum_{0\leq \tau_j \leq S} \big(k F(\xi_j) + hP(\xi_j)\big)\Big],
\end{equation}
see for example \cite{borkar1988ergodic, jack2006impulse, jack2006singular}. 
Therefore, we will use the latter version to characterize the lower bound since it is easier to manipulate.


\begin{rmk}\label{rmk: method of proof}
The approach of weak convergence is classical for proving inequalities similar to \eqref{eqn: lower bound I eps}, in particular in the study of heavy traffic networks (see \cite[section 9]{kushner2014partial} for an overview). The usual weak convergence theorems enable one to show that the perturbed system converges in the Skorohod topology to the controlled Brownian motion as $\varepsilon$ tends to zero. However, since the time horizon tends to infinity, this does not immediately imply the convergence of time-average cost functionals like $I_t^{\varepsilon}$.\\

In \cite{kushner1993limit}, the authors consider pathwise average cost problems for controlled queues in the heavy traffic limit, where the control term is absolutely continuous. They use the empirical ``{functional} occupation measure'' on the canonical path space and characterize the limit as a controlled Brownian motion. The same method has also been used in \cite{budhiraja2011ergodic} in the study of single class queueing networks.\\

However, this approach cannot be applied directly to singular/impulse controls for which the tightness of the occupation measures is difficult to establish. In fact, the usual Skorokhod topology is not suitable for the impulse control term 
$$\wt{Y}^\varepsilon_t:= \sum_{0< \wt{\tau}^\varepsilon_j\leq t} \wt{\xi}^\varepsilon_j .$$
Indeed, in the case of singular/impulse control, the component $\{\wt{Y}^\varepsilon\}$ is generally not tight under the Skorokhod topology. For example (see \cite[p.72]{kushner2001heavy}), consider the family $(Y_t^\varepsilon)$ where the function $Y_t^\varepsilon$  equals zero for $t<1$ and jumps upward by an amount $\sqrt{\varepsilon}$ at times $1 + i\varepsilon$,  $i=0, 1,\ldots,\varepsilon^{-1/2}$ until it reaches the value unity. The natural limit of $Y^\varepsilon$ is of course $\ind{t\geq 1}$ but this sequence is not tight in the Skorokhod topology. The nature of this convergence is discussed in \cite{kurtz1991random} and a corresponding topology is provided in \cite{jakubowski1997non}. \\

This difficulty could be avoided by introducing a random time change after which the {\color{black}(suitably interpolated)} control term becomes uniformly Lipschitz and hence converges under the Skorokhod topology. This technique is used in \cite{Budhiraja2006,budhiraja2012controlled,kushner2001heavy} to study the convergence of controlled queues with discounted costs.  This seems to be a possible alternative way to extend the approach of \cite{kushner1993limit} to singular/impulse controls. Nevertheless, the analysis would probably be quite involved.\\

Instead of proving the tightness of control terms $(\wt{Y}^{\varepsilon}_t)$ in weaker topologies, we shall use an alternative characterization of the time-average control problem of Brownian motion. In \cite{kurtz1999martingale}, the authors characterize the time-average control of a Jackson network in the heavy traffic limit as the solution of a linear program. The use of occupation measure on the state space instead of the path space turns out to be sufficient to describe the limiting stochastic control problem. However, the optimization criterion is not pathwise. \\

In this paper, we use a combination of the techniques in \cite{kurtz1999martingale} and \cite{kushner1993limit} to obtain pathwise lower bounds. 
\end{rmk}

\subsection{Lower bound}

In order to properly state our result for the case of combined regular and impulse control, we first introduce the solution $I= I(a, r,l, k, h) $ of the following linear programming problem:

\begin{equation}\label{LP: cost part I}
I(a, r, l, k, h) = \inf_{(\mu, \rho)} \int_{\bbR^d_x \times \bbR^d_u}\!\big(r D(x) + l Q(u) \big)\mu(dx \times du) + \int_{\bbR^d_x\times \bbR^d_\xi \setminus\{0_\xi\}}\!\big(kF(\xi) + h P(\xi)\big)\rho(dx\times d\xi),
\end{equation}
with $(\mu, \rho)\in\scrP(\bbR^d_x\times \bbR^d_u)\times \scrM(\bbR^d_x\times \bbR^d_\xi \setminus\{0_\xi\})$ satisfying the following constraint
\begin{equation}\label{LP: constraint part I}
\int_{\bbR^d_x\times \bbR^d_u}A^af(x, u)\mu(dx\times du) + \int_{\bbR^d_x\times \bbR^d_\xi\setminus\{0_\xi\}}Bf(x, \xi)\rho(dx\times d\xi) = 0, \quad \forall f\in C^2_0(\bbR^d_x),
\end{equation}
where
\begin{equation*}
A^{a}f(x, u) = \frac{1}{2}\sum_{i, j}a_{ij}\partial^2_{ij} f(x) + {\langle u, \nabla f(x)\rangle}, \quad Bf(x, \xi) = f(x+\xi) - f(x).
\end{equation*} 
We will see in  Section \ref{sec: ta control} that it is essentially an equivalent characterization of  the time-average control problem \eqref{eqn: ta dynm const}-\eqref{eqn: ta cost const}.
In Example \ref{ex: combined impulse}, we consider a particular case for which $I$ and the optimal solution $\mu^*, \rho^*$ can be explicitly determined.  From now on, we make the following assumptions.

\begin{asmp}[Regularity of linear programming]\label{asmp: regularity lp}
The function $I = I(a, r, l, k, h)$ defined by \eqref{LP: cost part I}-\eqref{LP: constraint part I} is measurable. 
\end{asmp}

\begin{asmp}[Model] \label{asmp: model} 
The predictable processes $(a_t)$ and $(b_t)$ are continuous and $(a_t)$ is positive definite on $[0, T]$. 
\end{asmp}

\begin{asmp}[Optimization criterion]\label{asmp: cost}
The parameters of the cost functional $(r_t)$, $(l_t)$, $(k_t)$ and $(h_t)$ are continuous and positive on $[0, T]$.
\end{asmp}

\begin{asmp}[Asymptotic framework]\label{asmp: beta}
{\color{black}The cost functionals satisfy the homogeneity property \eqref{eqn: homogeneous cost}  and the relation \eqref{eqn: order beta} holds for some $\beta>0$. }
\end{asmp}

\begin{rmk}
Let us comment briefly the above assumptions. Assumption \ref{asmp: regularity lp} is necessary to avoid pathological cases. In most examples, the function $I$ is continuous (see Examples \ref{ex: stochastic}-\ref{ex: combined singular}). Assumptions \ref{asmp: model}-\ref{asmp: cost} impose minimal regularity on the dynamics of $X^\circ$ and cost parameters.  Assumption \ref{asmp: beta}  ensures that all the costs have similar order of magnitude.
\end{rmk}

Second, we introduce the following notion.
\begin{dfn}
Let $\{Z, (Z^\varepsilon)_{\varepsilon}, \varepsilon > 0\}$ be random variables on the same probability space $(\Omega, \calF, \bbP)$. We say that $Z^\varepsilon$ is asymptotically bounded from below by $Z$ in probability if 
\begin{equation*}
\forall \delta > 0, \quad \lim_{\varepsilon \to 0} \bbP[Z^\varepsilon> Z -\delta] =1. 
\end{equation*}
We write $\liminf_{\varepsilon\to 0} Z^\varepsilon \geq_p Z$.
\end{dfn}

We now give the version of our main result for the case of combined regular and impulse control. 
\begin{theo}[Asymptotic lower bound for combined regular and impulse control]\label{theo: lower bound}
Under Assumptions \ref{asmp: regularity lp}, \ref{asmp: model}, \ref{asmp: cost} and \ref{asmp: beta},  we have
\begin{equation}\label{eqn: lower bound impulse}
\liminf_{\varepsilon\to 0} \frac{1}{\varepsilon^{\beta\zeta_D }} J^\varepsilon(u^\varepsilon, \tau^\varepsilon, \xi^\varepsilon) \geq_p \int_0^T  I(a_t, r_t, l_t, k_t, h_t) dt,
\end{equation}
for any sequence of admissible tracking strategies $\{(u^\varepsilon, \tau^\varepsilon, \xi^\varepsilon) \in \mathcal{A}, {\varepsilon >0}\}$.
\end{theo}

Thus, in Theorem \ref{theo: lower bound} we have expressed the lower bound for the tracking problem in terms of the integral of the solution $I$ of a linear program, which will be interpreted as time-average control of Brownian motion in Section \ref{sec: explicit I}. For any subsequence $\varepsilon'$, we can always pick a further subsequence $\varepsilon''$ such that 
\begin{equation*}
\liminf_{\varepsilon''\to 0} \frac{1}{(\varepsilon'')^{\beta \zeta_D}}J^{\varepsilon''}(u^{\varepsilon''}, \tau^{\varepsilon''}, \xi^{\varepsilon''}) \geq \int_0^T I(a_t, r_t, l_t, k_t, h_t) dt - \delta,
\end{equation*}
almost surely. Therefore, by Fatou's lemma, the following corollary holds. 
\begin{cor} We have
\begin{equation*}
\liminf_{\varepsilon\to 0} \frac{1}{\varepsilon^{\zeta_D \beta}} \mathbb{E}[{J^\varepsilon (u^\varepsilon, \tau^\varepsilon, \xi^\varepsilon)} \geq \bbE\Big[{\int_0^T  I(a_t, r_t, l_t, k_t, h_t) dt}\Big].
\end{equation*}
\end{cor}

\section{Extensions of Theorem \ref{theo: lower bound} to other types of control}\label{sec: extensions part I}

{\color{black}In this section, we consider the case of combined regular and singular control and those with only one type of control. In particular, we will see that in the presence of singular control, the operator $B$ is different. Formally we could give similar results for the combination of all three controls or even in the presence of several controls of the same type with different cost functions and scaling properties. To avoid cumbersome notation, we restrict ourselves to the cases meaningful in practice, which are illustrated by explicit examples in Section \ref{sec: explicit I}.}

\subsection{Combined regular and singular control}
When the fixed cost component is absent,  that is $F= 0$, impulse control and singular control can be merged. In that case, the natural way to formulate the tracking problem is to consider a strategy $(u^\varepsilon, \gamma^\varepsilon, \varphi^\varepsilon)$ with $u^\varepsilon$ a progressively measurable process as before, $\gamma^\varepsilon_t\in \Delta$ and $(\varphi^\varepsilon_t)$ a {\color{black}possibly discontinuous non-decreasing process} such that
\begin{equation*}
X^\varepsilon_t = -X^\circ_t + \int_0^tu^\varepsilon_s ds + \int_0^t \gamma^\varepsilon_{s}d\varphi^\varepsilon_s, 	\end{equation*}
and
\begin{equation*}
J^\varepsilon(u^\varepsilon, \gamma^\varepsilon, \varphi^\varepsilon) = \int_0^T\big(r_tD(X^\varepsilon_t) + \varepsilon^{\beta_Q}l_tQ(u^\varepsilon_t)\big)dt + \int_0^T \varepsilon^{\beta_P}h_tP(\gamma^\varepsilon_{t})d\varphi^\varepsilon_t.	
\end{equation*}
To avoid degeneracy, we assume that for any $\gamma\in\Delta$, 
\begin{equation}\label{eqn: P i pos}
P(\gamma)>0.
\end{equation}

Using similar heuristic arguments as those in the previous section, we are led to consider the time-average control of Brownian motion with combined regular and singular control 
 \begin{equation}\label{eqn: ta cost const singular}
I(a, r, l, h)=\inf_{(u, \gamma, \varphi)}{\limsup_{S\to \infty}}\frac{1}{S}\bbE\Big[\int_0^S \big(rD(X_s) + l Q( u_s)\big) ds + \int_0^S h_s P(\gamma_{s}) d\varphi_s\Big],
\end{equation}
where
\begin{equation}\label{eqn: ta dynm const singular}
{\color{black}X_s = \sqrt{a} W_s  + \int_0^s u_r dr + \int_0^s \gamma_r d\varphi_r.}
\end{equation}
The corresponding linear programming problem is given by 
\begin{equation}\label{LP: cost singular part I}
I(a, r, l, h) = \inf_{(\mu, \rho)} \int_{\bbR^d_x \times \bbR^d_u}\big(r D(x) + l Q(u)\big)\mu(dx \times du) + \int_{\bbR^d_x\times \Delta\times \bbR^+_\delta}h P(\gamma) \rho(dx\times d\gamma\times d\delta),
\end{equation}
with $(\mu, \rho)\in\scrP(\bbR^d_x\times \bbR^d_u)\times \scrM(\bbR^d_x\times \Delta \times \bbR^+_ \delta)$ satisfying the following constraint
\begin{equation}\label{LP: constraint singular part I}
\int_{\bbR^d_x\times \bbR^d_u}A^af(x, u)\mu(dx\times du) + \int_{\bbR^d_x\times \Delta \times \bbR^+_\delta}Bf(x, \gamma, \delta)\rho(dx\times d\gamma \times d\delta) = 0, \quad \forall f\in C^2_0(\bbR^d_x),
\end{equation}
where
\begin{equation*}
A^{a}f(x, u) = \frac{1}{2}\sum_{ij}a_{ij}\partial^2_{ij}f(x) + \langle u, \nabla f(x)\rangle, \quad Bf(x, \gamma, \delta) = \begin{cases}
\langle\gamma,\nabla f(x)\rangle, &\delta = 0,\\
\delta^{-1} \big(f(x+ \delta \gamma) -f(x)\big), &\delta>0.
\end{cases}
\end{equation*} 


{\color{black}We have the following theorem.
\begin{theo}[Asymptotic lower bound for combined regular and singular control]\label{theo: lower bound singular} Assume that $I(a, r, l, h)$ is measurable, that the parameters $(r_t)$, $(l_t)$ and $(h_t)$ are continuous and positive, that Assumption \ref{asmp: model} holds true and that Assumption \ref{asmp: beta} is satisfied for some $\beta>0$.
Then,
\begin{equation}\label{eqn: lower bound singular}
\liminf_{\varepsilon\to 0} \frac{1}{\varepsilon^{\beta\zeta_D }} J^\varepsilon (u^\varepsilon, \gamma^\varepsilon, \varphi^\varepsilon)\geq_p \int_0^T  I(a_t, r_t, l_t, h_t) dt, 
\end{equation}
 for any sequence of admissible  tracking strategies $\{(u^\varepsilon, \gamma^\varepsilon, \varphi^\varepsilon) \in\mathcal{A}, {\varepsilon >0}\}$.
\end{theo}}
Adapting the proofs of Theorem \ref{theo: lower bound} and Theorem \ref{theo: lower bound singular} in an obvious way, 
we easily obtain  the following bounds when only one control is present.

\subsection{Impulse control}
 Consider 
\begin{equation*}
\inf_{(\tau^\varepsilon, \xi^\varepsilon)\in \mathcal{A}}J^\varepsilon(u^\varepsilon, \tau^\varepsilon, \xi^\varepsilon), 
\end{equation*}
with
\begin{equation*}
 J^\varepsilon(u^\varepsilon, \tau^\varepsilon, \xi^\varepsilon) = \int_0^T r_t D(X_t^\varepsilon) dt + \sum_{0 < \tau^\varepsilon_j \leq T}\big( \varepsilon^{\beta_F}k_{\tau^\varepsilon_j}F(\xi_j^\varepsilon) + \varepsilon^{\beta_P}h_{\tau^\varepsilon_j}P(\xi_j^\varepsilon)\big),	
\end{equation*}
and
\begin{equation*}
X_t^\varepsilon = -X_t^\circ  +  \sum_{0< \tau^\varepsilon_j \leq t}\xi_j^\varepsilon.
\end{equation*}
We have the following theorem.{\color{black}
\begin{theo}[Asymptotic lower bound for impulse control]\label{theo: lower bound impulse} 
Let $I = I(a, r, k, h)$ be given by
\begin{equation*}\label{LP: cost}
I(a, r,  k, h) = \inf_{(\mu, \rho)} \int_{\bbR^d_x }\!r D(x)\mu(dx) + \int_{\bbR^d_x\times \bbR^d_\xi \setminus\{0_\xi\}}\!(kF(\xi) + h P(\xi))\rho(dx\times d\xi),
\end{equation*}
with $(\mu, \rho)\in\scrP(\bbR^d_x)\times \scrM(\bbR^d_x\times \bbR^d_\xi \setminus\{0_\xi\})$ satisfying the following constraint
\begin{equation*}\label{LP: constraint}
\int_{\bbR^d_x}A^af(x)\mu(dx) + \int_{\bbR^d_x\times \bbR^d_\xi\setminus\{0_\xi\}}Bf(x, \xi)\rho(dx\times d\xi) = 0, \quad \forall f\in C^2_0(\bbR^d_x),
\end{equation*}
where
\begin{equation*}
A^{a}f(x) = \frac{1}{2}\sum_{i, j}a_{ij}\partial^2_{ij} f(x), \quad Bf(x, \xi) = f(x+\xi) - f(x).
\end{equation*} 
Assume that $I(a,r,k,h)$ is measurable, that the parameters $(r_t)$, $(k_t)$ and $(h_t)$ are continuous and positive, that Assumption \ref{asmp: model} holds true, and that Assumption \ref{asmp: beta} is satisfied for some $\beta>0$.
Then, 
\begin{equation*}
\liminf_{\varepsilon\to 0} \frac{1}{\varepsilon^{\beta\zeta_D }} J^\varepsilon(\tau^\varepsilon, \xi^\varepsilon) \geq_p \int_0^T  I(a_t, r_t, k_t, h_t) dt.
\end{equation*}
\end{theo}}

See Example \ref{ex: impulse} for a closed form solution of $I$. 

\subsection{Singular control}
Consider
\begin{equation*}
\inf_{(\gamma^\varepsilon, \varphi^\varepsilon)\in\mathcal{A}} J^\varepsilon (\gamma^\varepsilon, \varphi^\varepsilon),
\end{equation*}
with
\begin{equation*}
J^\varepsilon( \gamma^\varepsilon, \varphi^\varepsilon) = \int_0^Tr_tD(X^\varepsilon_t) dt + \int_0^T \varepsilon^{\beta_P}h_tP(\gamma^\varepsilon_{t})d\varphi^\varepsilon_t.	
\end{equation*}
and
\begin{equation*}
X^\varepsilon_t = -X^\circ_t+ \int_0^t \gamma^\varepsilon_{s}d\varphi^\varepsilon_s. 	
\end{equation*}
We have the following theorem.{\color{black}
\begin{theo}[Asymptotic lower bound for singular control]\label{theo: lower bound singular only} 
Let $I = I(a, r, h)$ be given by
\begin{equation*}\label{LP: cost singular}
I(a, r,  h) = \inf_{(\mu, \rho)} \int_{\bbR^d_x }r D(x) \mu(dx ) + \int_{\bbR^d_x\times \Delta\times \bbR^+_\delta}h P(\gamma) \rho(dx\times d\gamma\times d\delta),
\end{equation*}
with $(\mu, \rho)\in\scrP(\bbR^d_x)\times \scrM(\bbR^d_x\times \Delta \times \bbR^+_ \delta)$ satisfying the following constraint
\begin{equation*}\label{LP: constraint singular}
\int_{\bbR^d_x}A^af(x)\mu(dx) + \int_{\bbR^d_x\times \Delta \times \bbR^+_\delta}Bf(x, \gamma, \delta)\rho(dx\times d\gamma \times d\delta) = 0, \quad \forall f\in C^2_0(\bbR^d_x),
\end{equation*}
where
\begin{equation*}
A^{a}f(x) = \frac{1}{2}\sum_{ij}a_{ij}\partial^2_{ij}f(x), \quad Bf(x, \gamma, \delta) = \begin{cases}
\langle\gamma, \nabla f(x) \rangle, &\delta = 0,\\
\delta^{-1} \big(f(x+ \delta \gamma) -f(x)\big), &\delta>0.
\end{cases}
\end{equation*} 
Assume that $I(a,r,h)$ is measurable, that the parameters $(r_t)$ and $(h_t)$ are continuous and positive, that Assumption \ref{asmp: model} holds true and that Assumption \ref{asmp: beta} is satisfied for some $\beta>0$. Then, 
\begin{equation*}
\liminf_{\varepsilon\to 0} \frac{1}{\varepsilon^{\beta\zeta_D }} J^\varepsilon(\gamma^\varepsilon, \varphi^\varepsilon) \geq_p \int_0^T  I(a_t, r_t,  h_t) dt.
\end{equation*}
\end{theo}}

See Example \ref{ex: singular}  for a closed form solution of $I$. 

\subsection{Regular control}
Consider
\begin{equation*}
\inf_{u^\varepsilon\in \mathcal{A}}J^\varepsilon(u^\varepsilon), 
\end{equation*}
with
\begin{equation*}
J^\varepsilon(u^\varepsilon) = \int_0^T \big(r_t D(X_t^\varepsilon) + \varepsilon^{\beta_Q}l_t Q(u^\varepsilon_t)\big) dt,
\end{equation*}
and
\begin{equation*}
X_t^\varepsilon = -X_t^\circ +\int_0^t u^\varepsilon_s ds.
\end{equation*}
We have the following theorem.{\color{black}
\begin{theo}[Asymptotic lower bound for regular control] \label{theo: lower bound stoch}
Let $I = I(a, r, l)$ be given by
\begin{equation*}
I(a, r, l) = \inf_{\mu} \int_{\bbR^d_x \times \bbR^d_u}\!\big(r D(x) + lQ(u)\big)\mu(dx, du) , 
\end{equation*}
with $\mu\in\scrP(\bbR^d_x\times \bbR^d_u)$ satisfying the following constraint
\begin{equation*}
\int_{\bbR^d_x\times \bbR^d_u}A^af(x, u)\mu(dx, du) = 0, \quad \forall f\in C^2_0(\bbR^d_x),
\end{equation*}
where
\begin{equation*}
A^{a}f(x, u) = \frac{1}{2}\sum_{i, j}a_{ij}\partial^2_{ij} f(x) + \langle u, \nabla f(x)\rangle.
\end{equation*} 

Assume that $I(a,r,l)$ is measurable, that the parameters $(r_t)$ and $(l_t)$ are continuous and positive on $[0,T]$, that Assumption \ref{asmp: model} holds true and that Assumption \ref{asmp: beta} is satisfied for some $\beta>0$.
Then, 
\begin{equation*}
\liminf_{\varepsilon\to 0} \frac{1}{\varepsilon^{\beta\zeta_D }} J^\varepsilon(u^\varepsilon) \geq_p \int_0^T  I(a_t, r_t, l_t) dt.
\end{equation*}
\end{theo}}

See Example \ref{ex: stochastic} for a closed form solution of $I$. 

%
%
\begin{rmk}[Upper bound]\label{upb}
It is natural to wonder whether the lower bounds in our theorems are tight and if it is the case, what are the strategies that attain them. In a forthcoming work,  we show that for the examples provided in Section \ref{sec: explicit I}, there are closed form strategies attaining asymptotically the lower bounds. For instance, in the case of combined regular and impulse control, it means that there exist $(u^{\varepsilon, *}, \tau^{\varepsilon, *}, \xi^{\varepsilon, *}) \in \mathcal{A}$ such that
\begin{equation*}
\lim_{\varepsilon\to 0} \frac{1}{\varepsilon^{\beta\zeta_D }} J^\varepsilon(u^{\varepsilon, *}, \tau^{\varepsilon, *}, \xi^{\varepsilon, *})  \to_p \int_0^T  I(a_t, r_t, l_t, k_t, h_t) dt.
\end{equation*}
These optimal strategies are essentially time-varying versions of the optimal strategies for the time-average control of Brownian motion.  
\end{rmk}

\section{Interpretation of lower bounds and examples}\label{sec: explicit I}

Our goal in this section is to provide a probabilistic interpretation of the lower bounds in Theorems \ref{theo: lower bound}, \ref{theo: lower bound singular}, \ref{theo: lower bound impulse}, \ref{theo: lower bound singular only} and \ref{theo: lower bound stoch}, which are expressed in terms of linear programming. In particular, we want to connect them with the time-average control problem of Brownian motion. 
To our knowledge, there is no  general result available for the equivalence between time-average control problem and linear programming. Partial results exist in \cite{borkar1988ergodic, kurtz1998existence, kurtz1999martingale, kurtz2013linear, Helmes2014} but do not cover all the cases we need. Here we provide a brief self-contained study enabling us to also treat the cases of singular/impulse controls and their combinations with regular control. 
We first introduce controlled martingale problems and show that they can be seen as a relaxed version of the controlled Brownian motion \eqref{eqn: local controlled process}. Then we formulate the time-average control problem in this martingale framework. We finally show that this problem has an equivalent description in terms of infinite dimensional linear program. 
While essential ingredients and arguments for obtaining these results  are borrowed from \cite{kurtz1998existence} and \cite{kurtz2001stationary}, we provide sharp conditions which guarantee the equivalence of these  two formulations.  


\subsection{Martingale problem associated to controlled Brownian motion}

In \cite{Soner2012, possamai2012homogenization, moreau2014trading, altarovici2013asymptotics}, the authors obtain a HJB equation in the first order expansion for the value function of the utility maximization problem under transaction costs, which essentially provides a lower bound for their control problems. They mention a connection between the HJB equation and the time-average control problem of Brownian motion, see also \cite{Hynd2012}. Here we wish to rigorously establish an equivalence between the linear programs in our lower bounds and the time-average control of Brownian motion. This leads us to introduce a relaxed version for the controlled Brownian motion. We shall see that the optimal costs for all these formulations coincide in the examples provided in the next section.\\

We place ourselves in the setting of \cite{kurtz2001stationary},  from which we borrow and rephrase several elements, and assume that the state space $E$ and control spaces $U$ and $V$ are complete, separable, metric spaces. Consider an operator $A: \calD \subset C_b(E)\to C(E\times U)$ and an operator $B: \calD \subset C_b(E)\to C(E\times V)$.

\begin{dfn}[Controlled martingale problem]
A triplet $(X, \Lambda, \Gamma)$ with $(X, \Lambda)$ an $E\times \scrP(U)$-valued process and $\Gamma$ an $\calL(E\times V)$-valued random variable is a solution of the controlled martingale problem for $(A, B)$ with initial distribution $\nu_0\in \scrP(E)$  if there exists a filtration $(\calF_t)$ such that the process $(X,\Lambda,  \Gamma_t)$ is $\calF_t$-progressive, $X_0$ has distribution $\nu_0$ and for every $f\in \calD$, 
\begin{equation}\label{eqn: mart}
f(X_t) - \int_0^t \int_U A f(X_s, u)\Lambda_s(du) ds - \int_{E\times V\times [0, t]} Bf(x, v)\Gamma(dx\times dv\times ds)
\end{equation}
is an $\calF_t$-martingale. 
\end{dfn}

We now consider two specific cases for the operators $A$ and $B$, which will be relevant in order to express our lower bounds. Furthermore, we explain why these specific choices of $A$ and $B$ are connected to combined regular and singular/impulse control of Brownian motion. 

\begin{ex}[Combined regular and impulse control of Brownian motion]\label{ex: A B impulse}
Let $  \calD = C^2_0(\bbR^d)\oplus \bbR$ and define $A: \calD \to C(E\times U)$ and $B:\calD\to C(E\times V)$ by 
\begin{align}
Af(x, u)& = \frac{1}{2}\sum_{i, j}a_{ij}\partial^2_{ij}f(x) + \langle u, \nabla f(x)\rangle, \label{eqn: operator A}\\
Bf(x, \xi) &=f(x+\xi) -f(x). \label{eqn: operator B impulse} 
\end{align}
Here $E =\bbR^d_x$, $U = \bbR^d_u$ and $V = \bbR^d_\xi\setminus \{0_\xi\}$. 
We call  any solution of this martingale problem the  combined regular and impulse control of Brownian motion. \\

Indeed, consider the following process
\begin{equation*}
X_t = X_0 + \sqrt{a}W_t +\int_0^tu_sds + \sum_{0<\tau_j\leq t}\xi_j,
\end{equation*}
with $(u_t)$ a progressively measurable process, $(\tau_j)$ a sequence of stopping times and $(\xi_j)$ a sequence of $\calF_{\tau_j}$-measurable random variables. Define
\begin{equation*}
N_t = \sum_{j} \ind{\tau_j \leq t},	\quad \xi_t = \xi_j, \quad t\in (\tau_{j-1}, \tau_j].
\end{equation*}
 Then for any $f\in \calD$, by It\=o's formula, 
\begin{align*}
f(X_t) -\int_0^t Af(X_s, u_s)ds &-\int_0^t Bf(X_{s-}, \xi_{s-}) dN_s\\
&=f(X_0) -\int_0^t \nabla f(X_s)^T \sqrt{a} dW_s,
\end{align*}
which is a martingale.
Let 
\begin{equation*}
\Lambda_t=\delta_{u_t}(du), \quad \Gamma (H\times [0, t]) = \int_0^t \mathbbm{1}_{H}(X_{s-}, \xi_{s-}) dN_s, \quad H\in\calB(\bbR^d_x\times\bbR^d_\xi\setminus \{0_\xi\}).
\end{equation*}
Then $(X, \Lambda, \Gamma)$ solves the martingale problem $(A, B)$ with initial distribution $\calL(X_0)$. 
\end{ex}

\begin{ex}[Combined regular and singular control of Brownian motion]\label{ex: A B singular}
Take $\calD = C^2_0(\bbR^d)\oplus \bbR$ and define 
\begin{align}
Af(x, u)& = \frac{1}{2}\sum_{i, j}a_{ij}\partial^2_{ij}f(x) + \langle u, \nabla f(x)\rangle, \label{eqn: operator A singular} \\
Bf(x, \gamma, \delta) &=
\begin{cases}
\langle\gamma, \nabla f(x)\rangle, \quad & \delta =0\\
\delta^{-1}\big(f(x+ \delta\gamma) - f(x)\big), \quad &\delta >0.
\end{cases}\label{eqn: operator B singular}
\end{align}
Here $E =\bbR^d_x$, $U = \bbR^d_u$ and $V = \Delta\times \bbR^+_{\delta}$. 
Any solution of this martingale problem is called combined regular and singular control of Brownian motion. \\

Indeed, let $X$ be given by
\begin{equation*}
X_t = X_0 + \sqrt{a}W_t +\int_0^tu_sds + \int_0^t \gamma_{s}d\varphi_s,
\end{equation*}
with $u$ a progressively measurable process, $\gamma_s\in \Delta$ and $\varphi_s$ non-decreasing. 
By It\=o's formula, we have
\begin{align*}
f(X_t) -\int_0^t Af(X_s, u_s)ds &-\int_0^t Bf(X_{s-}, \gamma_{s}, \delta \varphi_s) d\varphi_s\\
&=f(X_0) -\int_0^t \nabla f(X_s)^T \sqrt{a} dW_s,
\end{align*}
which is a martingale for any $f\in \calD$. 
Let 
\begin{equation*}
\Lambda_t=\delta_{u_t}(du), \quad \Gamma (H\times [0, t]) = \int_0^t \mathbbm{1}_{H}(X_{s-}, \gamma_{s}, \delta\varphi_s) d\varphi_s, \quad H\in\calB(\bbR^d_x\times \Delta\times \bbR^+_\delta).
\end{equation*}
Then $(X, \Lambda, \Gamma)$ solves the martingale problem $(A, B)$ with initial distribution $\calL(X_0)$. 
\end{ex}

\subsection{Time-average control of Brownian motion}\label{sec: ta control}

Now we formulate a relaxed version of the time-average control problem of Brownian motion in terms of a controlled martingale problem. This generalizes \cite{kurtz1998existence, Helmes2014} to combined regular and singular/impulse control of martingale problems, see also \cite{kurtz2013linear}. Recall that $A$ and $B$ are two operators where $A: \calD \subset C_b(E)\to C(E\times U)$ and $B: \calD \subset C_b(E)\to C(E\times V)$.
 Consider two cost functionals $C_A:E\times U \to \bbR_+$ and $ C_B:E\times V\to \bbR_+$. 

\begin{dfn}[Martingale formulation of time-average control problem] \label{def: ta MP}
The time-average control problem under the martingale formulation is given by
\begin{equation}\label{eqn: cost MP}
I^M = \inf_{(X,\Lambda, \Gamma)}\limsup_{t\to \infty}\frac{1}{t}\bbE\big[\int_0^t \int_U C_A(X_s, u)\Lambda_s(du)ds  + \int_{E\times V\times [0, t]} C_B(x, v) \Gamma(dx\times dv\times ds)\big],
\end{equation}
where the $\inf$ is taken over all solutions of the martingale problem $(A, B)$ with any initial distribution $\nu_0\in \scrP(E)$.
\end{dfn}
Now, let $(X, \Lambda, \Gamma)$ be any solution of the martingale problem with operators $A$ and $B$. Define $(\mu_t, \rho_t) \in \scrP(E\times U) \times \scrM(E\times V)$ as 
\begin{align}
\mu_t(H_1) &= \frac{1}{t}\bbE\Big[{\int_0^t\int_U \mathbbm{1}_{H_1}(X_s, u)\Lambda_s(du) ds}\Big], \label{eqn: occup mu}\\
\rho_t(H_2) &= \frac{1}{t}\bbE\big[\Gamma(H_2\times [0, t])\big],\label{eqn: occup rho}
\end{align}
for $H_1\in \calB(E\times U)$ and  $H_2\in \calB(E\times V)$. Then the average cost up to time $t$ in \eqref{eqn: cost MP}  can be expressed as
\begin{equation*}
\int_{E\times U} C_A(x, u)\mu_t(dx\times du) + \int_{E\times V} C_B(x, v)\rho_t (dx\times dv).
\end{equation*}
On the other hand, for $f\in \calD$, \eqref{eqn: mart} defines a martingale. Taking the expectation, we obtain
\begin{equation*}
\frac{1}{t}\bbE\Big[\int_0^t\int_UAf(X_s, u)\Lambda_s(du)ds  + \int_{E\times V\times [0, t]}Bf(x, v)\Gamma(dx\times dv\times ds)\Big] =\frac{1}{t} \bbE[f(X_t) - f(X_0)].
\end{equation*}
If $X_t$ is stationary, we have
\begin{equation*}
\int_{E\times U}Af(x, u)\mu_t(dx\times du) + \int_{E\times V}Bf(x, v)\rho_t(dx\times dv) = 0. 
\end{equation*}
Letting $t$ tend to infinity, this leads us to introduce the following linear programming  problem. 

\begin{dfn}[Linear programming (LP) formulation of time-average control]
The time-average control problem under the LP formulation is given by
\begin{equation}\label{eqn: cost LP}
I^P= \inf_{(\mu, \rho)} c(\mu, \rho),
\end{equation} 
with 
\begin{align}
c: \scrP(E\times U) \times \scrM(E\times V) &\to \bbR_+ \nonumber\\
(\mu, \rho)&\mapsto \int_{E\times U} C_A(x, u)\mu(dx\times du) + \int_{E\times V} C_B(x, v)\rho (dx\times dv),\label{eqn: cost functional c}
\end{align}
where the $\inf$ is computed over all $\mu\in \scrP(E\times U)$ and $ \rho\in \scrM(E\times V)$ satisfying the constraint
\begin{equation}\label{eqn: LP constraint}
 \int_{E\times U} Af(x, u)\mu(dx\times du) + \int_{E\times V} Bf(x, v)\rho(dx\times dv) = 0, \quad  \forall f\in \calD.
\end{equation}
\end{dfn}
We now present the theorem which connects linear programming and time-average control of Brownian motion. 

\begin{theo}[Equivalence between $I^M$ and $I^P$] \label{theo: equiv}
Assume that 
\begin{enumerate}
\item (Condition on the operators $A$ and $B$) The operators $A$ and $B$ satisfy Condition 1.2 in \cite{kurtz2001stationary}. In particular, there exist $\psi_A \in C(E\times U)$, $\psi_B\in C(E\times V)$,  and constants $a_f, b_f$ depending on $f\in \calD$ such that
\begin{equation} \label{eqn: Psi A B}
|Af(x, u)| \leq a_f \psi_A(x, u), \quad |Bf(x, v)|\leq b_f \psi_B(x, v),\quad \forall x\in E, u\in U, v\in V.
\end{equation}

\item 
(Condition on the cost function $C_A$) The cost function $C_A$ is non-negative and inf-compact, that is
$\{(x, u)\in E\times U \text{ }|\text{ } C_A(x, u) \leq c\}$  is a compact set for each $c\in \bbR_+$. In particular, $C_A$ is lower semi-continuous.

\item 
(Condition on cost function $C_B$) The cost function  $C_B$ is non-negative and lower semi-continuous. Moreover,  $C_B$ satisfies 
\begin{equation}\label{eqn: C B pos infty}
\inf_{(x, v) \in E\times V} C_B(x, v) > 0. 
\end{equation}

\item 
(Relation between operators and cost functions) There exist constants  $\theta$  and $0<\beta<1$ such that
\begin{equation}\label{eqn: bound psi c}
\psi_A(x, u)^{1/\beta} \leq \theta\big(1+ C_A(x, u)\big), \quad \psi_B(x, v)^{1/\beta} \leq  \theta C_B(x, v),
\end{equation}
for $\psi_A$ and $\psi_B$  given by \eqref{eqn: Psi A B}. 
\end{enumerate}
Then the two formulations above of the time-average control problem are equivalent in the sense that
\begin{equation*}
I^M = I^P.
\end{equation*}
\end{theo}

We therefore write $I$ for both $I^M$ and $I^P$. 

\begin{proof}
 The idea of the proof is the same as in \cite{kurtz1998existence}, with a key ingredient provided by \cite[Theorem 1.7]{kurtz2001stationary}. \\

 We first show that $I^P\leq I^M$.  Given any solution $(X, \Lambda, \Gamma)$ of the martingale problem, we define the occupation measures as \eqref{eqn: occup mu} and \eqref{eqn: occup rho}. Without loss of generality, we can assume that $\limsup_{t\to\infty} c(\mu_t, \rho_t) <\infty$ where  $c$ is defined in  \eqref{eqn: cost functional c} (otherwise we would have $I^M=\infty \geq I^P$). We will show that $I^P\leq \limsup_{t\to\infty}c(\mu_t, \rho_t)$.\\
 
We consider the one-point compacification $\overline{E\times V} = E\times V \cup \{\infty = (x_\infty, v_\infty)\}$ and extend $C_B$ to $\overline{E\times V}$ by
\begin{equation*}
C_B(x_\infty, v_\infty) = \liminf_{(x, v) \to (x_\infty, v_\infty)} C_B(x, v) > 0,
\end{equation*}
where the last inequality is guaranteed by \eqref{eqn: C B pos infty}.
Since $C_B$ is lower semi-continuous,  the level sets $\{(x, v) \in \overline{E\times V} \text{ }| \text{ } C_B(x, v) \leq c\}$ are compact.  By Lemma \ref{lem: tightness function}, we deduce that $c$ is a tightness function on $\scrP(E\times U) \times \calM(\overline{E\times V})$. So the family of occupation measures $(\mu_t, \rho_t)_{t\geq 0}$ is tight if $\rho_t$ is viewed as a measure on $\overline{E\times V}$.  It follows that the family of occupation measures indexed by $t$ is relatively compact. \\
 
 Let $(\mu, \overline{\rho})\in \scrP(E\times U) \times \calM(\overline{E\times V})$ be any limit point of $(\mu_t, \rho_t)$ with canonical decomposition $\bar{\rho} = \rho + \theta_{\bar{\rho}} \delta_\infty$.  We claim that $(\mu, \rho)$  satisfies the linear constraint \eqref{eqn: LP constraint}.  Indeed, by \eqref{eqn: Psi A B} and \eqref{eqn: bound psi c},  we have
\begin{equation*}
|Af| ^{1/\beta} \leq\theta_f (1 +C_A), \quad |Bf|^{1/\beta} \leq \theta_f C_B, 
\end{equation*}
where $\theta_f$ is a non-negative real number depending on $f$. Then $\sup_t c(\mu_t, \rho_t)<\infty$ implies that $Af$ and $Bf$ are uniformly integrable with respect to $\mu$ and $\bar{\rho}$. We therefore have
\begin{equation*}
\int_{E\times U} Af d\mu + \int_{\overline{E\times V}} \mathbbm{1}_{E\times V} Bf d\bar{\rho} = \lim_{t\to \infty} \int_{E\times U} Afd\mu_t + \int_{\overline{E\times V}} \mathbbm{1}_{E\times V} Bf d\rho_t.
\end{equation*}
The right hand side being equal to zero, we obtain
\begin{equation*}
\int_{E\times U} Af d\mu + \int_{{E\times V}}  Bf d{\rho} = 0.
\end{equation*}
Since $C_A$ and $C_B$ are lower semi-continuous, it follows that (see \cite[Theorem A.3.12]{dupuis2011weak})
$$I^P\leq c(\mu, \rho) \leq c(\mu, \bar{\rho})\leq \liminf_{t\to\infty} c(\mu_t, \rho_t)\leq \limsup_{t\to\infty} c(\mu_t, \rho_t).$$
As the choice of the solution of the controlled martingale problem $(X, \Lambda,\Gamma)$ is arbitrary,  we conclude that $I^P\leq I^M$.\\

We now show that $I^M\leq I^P$. Given any $(\mu, \rho)$ satisfying the linear constraint \eqref{eqn: LP constraint} such that $c(\mu, \rho) < \infty$, Theorem 1.7 in \cite{kurtz2001stationary}  together with the condition on the operators $A$ and $B$ provides the existence of a stationary solution  $(X, \Lambda, \Gamma)$ for the martingale problem $(A, B)$ with marginals $(\mu, \rho)$, hence $I^M \leq I^P$.
 \end{proof}

\begin{rmk} \label{rmk: relax C B}
Compared with the results in \cite{kurtz1998existence, borkar1988ergodic}, no near-monotone condition is necessary for the singular component. The is due to the fact that in the linear constraint \eqref{eqn: LP constraint}, $\rho$ belongs to $\calM(E\times V)$ instead of $\scrP(E\times V)$. 
\end{rmk}

We now give natural examples for which $I^P=I^M$. 
\begin{cor}[Time-average control of Brownian motion with quadratic costs]\label{cor: equiv}
Assume that $C_A$ is given by
\begin{equation}
C_A(x, u) = r \langle x, \Sigma^D x \rangle + l \langle u, \Sigma^Q u\rangle, \quad \forall (x, u)\in \bbR^d_x\times \bbR^d_u, 
\end{equation}
where $ \Sigma^D, \Sigma^Q\in \calS_d^+$. 
Then  for both following cases,  we have $I^M=I^P$:
\begin{enumerate}
\item
The operators $A$ and $B$ are  given by \eqref{eqn: operator A}-\eqref{eqn: operator B impulse} and the cost function $C_B$ is given by 
\begin{equation}\label{eqn: fixed prop cost}
C_B(x, \xi) =  k \sum_{i=1}^d F_i \mathbbm{1}_{\{\xi^i\neq 0\}} + h \langle P, |\xi|\rangle , \quad (x, \xi) \in \bbR^d_x\times \bbR^d_\xi\setminus\{0_\xi\},
\end{equation}
with $\min_i F_i > 0$.
\item 
The operators $A$ and $B$ are  given by \eqref{eqn: operator A singular}-\eqref{eqn: operator B singular} and the cost function $C_B$ is given by
\begin{equation}\label{eqn: explicit costs}
C_A(x, \gamma, \delta) = h \langle P,  |\gamma|\rangle, \quad (x, \gamma, \delta) \in \bbR^d_x\times \Delta \times \bbR^+_\delta,
\end{equation}
with  $\min_i P_i > 0$. 
\end{enumerate}

\end{cor}
\begin{proof}
We consider the impulse case.  First,  we show that $A$ and $B$ satisfy \cite[Condition 1.2]{kurtz2001stationary}.  (i) It is clear that $\mathbbm{1}_{\bbR^d}\in \calD$, and $A\mathbbm{1}_{\bbR^d}= 0, B\mathbbm{1}_{\bbR^d}= 0$. (ii) Define 
\begin{equation*}
\psi_A(x, u) =1 \vee \sum_{i=1}^d |u_i| , \quad \psi_B(x, \xi) =1, 
\end{equation*}
then \eqref{eqn: Psi A B} is satisfied. (iii) Since $C_0^2(\bbR^d_x)$ equipped with $\|\cdot\|_{C_0^2}$ is separable, the third condition is satisfied. (iv) $A_uf(x) = Af(x, u)$ and $B_\xi f(x) = Bf(x, \xi)$ satisfy the positive maximum principle, so they are dissipative. It is obvious that they verify \cite[(1.10)]{kurtz2001stationary}. Hence they are pre-generators.  (v) Obvious.  Second, since $C_A$ and $C_B$ are l.s.c. and $\min_i F_i >0$, the conditions on $C_A$ and $C_B$ are verified.   Third, \eqref{eqn: bound psi c} holds with $\beta = 1/2$. \\

The proof for the singular case is similar. Note that $\min_i P_i >0 $ is necessary for \eqref{eqn: C B pos infty} and the second bound in \eqref{eqn: bound psi c} to hold. 
\end{proof}


\subsection{Explicit examples in dimension one} 
We collect here a comprehensive list of examples in dimension one for which explicit solutions are available. 
Most of these results exist already in the literature under the classical SDE formulation (see for example \cite{dai2013browniana, jack2006impulse, jack2006singular, Jack2006a,  Soner2012, altarovici2013asymptotics, moreau2014trading, gobet2012almost, Helmes2014}), {\color{black}but Examples \ref{ex: combined impulse} and \ref{ex: combined singular} are presented here for the first time.} The basic idea is to solve explicitly the HJB equation corresponding to the time-average control problem and apply a verification theorem. Similar methods apply under the linear programming framework. However, for completeness we provide in Section \ref{sec: verification} detailed proofs tailored to the formulation in terms of linear programming. In fact, we prove only the case of combined regular and impulse control, that is Example \ref{ex: combined impulse}. The proofs for the other examples are similar and hence omitted.

\begin{ex}[Regular control of Brownian motion] \label{ex: stochastic}
  Let $r, l >0$ and consider the following linear programming problem
\begin{equation}\label{cost ex: stoch}
I(a, r, l) = \inf_{\mu} \int_{\bbR_x\times \bbR_u} (rx^2 + lu^2) \mu(dx, du),
\end{equation}\label{lp ex: stoch}
where $\mu \in \scrP(\bbR_x\times \bbR_u)$ satisfies
\begin{equation}\label{lp ex: stoch}
\int_{\bbR_x\times\bbR_u} \big(\frac{1}{2}af''(x)+ uf'(x)\big) \mu(dx, du) = 0, \quad \forall f\in C_0^2(\bbR_x). 
\end{equation}
By Corollary \ref{cor: equiv}, this is equivalent to the time-average control of Brownian motion with quadratic costs in the sense of Definition \ref{def: ta MP}.\\

Let us explain heuristically how to obtain the optimal solution {\color{black}(a rigorous verification argument for the linear programming formulation is provided in Section \ref{sec: verification}).}  Roughly speaking, Definition \ref{def: ta MP} describes the following dynamics
\begin{equation*}
dX^u_t = \sqrt{a} dW_t  + u_t dt.
\end{equation*}
The optimization objective is
\begin{equation*}
\inf_{(u_t)\in \mathcal{A}}\limsup_{T\to \infty} \frac{1}{T} \mathbb{E}\Big[\int_0^T (r (X^u_t)^2 + l u_t^2) dt \Big], 
\end{equation*}
{\color{black}where the set $\mathcal{A}$ of admissible controls contains all progressively measurable processes $(u_t)$ such that 
$$
\limsup_{T\to \infty} \mathbb E[(X^u_T)^2] <\infty
$$}
Consider the associated HJB equation
\begin{equation*}
\inf_{u\in \bbR_u} \frac{1}{2}a w''(x) + uw'(x) + l u^2 + rx^2 - I^V = 0,
\end{equation*}
{\color{black}where the constant $I^V$ must be found as part of the solution.}
It is easy to find the explicit solution (see also \cite[Equation (3.18)]{moreau2014trading}):
\begin{equation*}
w(x) = \sqrt{rl} x^2, \quad I^V = \sqrt{a^2 rl}.
\end{equation*}
{\color{black}Now, let $(u_t)$ be an admissible control and apply It\=o's formula to $w(X_T)$:}
\begin{equation*}
w(X_T) 
= w(X_0) + \int_0^T \Big(\frac{1}{2}a w''(X_t) + u_t w'(X_t)\Big) dt + \int_0^T w'(X_t) dW_t.
\end{equation*}
It follows that
\begin{align}
\int_0^T (rX_t^2 + lu_t^2)dt &\geq \int_0^T \Big(I^V - \frac{1}{2}a w''(X_t) - u_t w'(X_t)\Big) dt \label{eqn: geq hjb}\\
& =TI^V + w(X_0) - w(X_T) +  \int_0^T w'(X_t) dW_t.\nonumber
\end{align}
Taking expectation, dividing by $T$ on both sides, {\color{black}and using the admissibility conditions,} we obtain
\begin{equation*}
\limsup_{T\to \infty} \frac{1}{T} \mathbb{E}\Big[\int_0^T (r X_t^2 + l u_t^2) dt \Big]\geq I^V. 
\end{equation*}
{\color{black}To show that $I^V$ is indeed the optimal cost, it is enough to show that equality holds in \eqref{eqn: geq hjb}. for the optimal feedback control given by }
\begin{equation*}
u^*(x) = -\frac{w'(x)}{2l} = -\theta x,~\theta=\sqrt{\frac{r}{l}}.
\end{equation*}
Therefore, the optimally controlled process is an Ornstein-Uhlenbeck process 
\begin{equation*}
dX^* =\sqrt{a}dW_t -\theta X_t^* dt. 
\end{equation*}
Naturally, the stationary distribution of $(X^*, u^*(X_t^*))$ is the solution $\mu^*$ of the linear programming problem. We have the following result.

\begin{prop}\label{lem: stochastic}
 The solution of \eqref{cost ex: stoch}-\eqref{lp ex: stoch} is given by
\begin{equation*}
I(a, r, l) = \sqrt{a^2rl},
\end{equation*}
and the optimum is attained by
\begin{equation*}
\mu^*(dx, du)  = \frac{1}{\sqrt{2\pi} \sigma} e^{-\frac{x^2}{2 \sigma^2}} dx\otimes \delta_{\{-\theta x\}}(du), 
\end{equation*}
where
\begin{equation*}
\sigma^2 = \frac{a}{2\theta}, \quad \theta =\sqrt{\frac{r}{l}}.
\end{equation*}
\end{prop}
\end{ex}

 \begin{ex}[Singular control of Brownian motion]  \label{ex: singular}
 For any parameters $r, h> 0$, consider the following linear programming problem
\begin{equation}\label{cost ex: singular}
I(a, r, h)= \inf_{(\mu, \rho)} \int_{\bbR_x} rx^2 \mu(dx) + \int_{\bbR_x \times\{\pm 1\}\times \bbR^+_\delta} h |\gamma| \rho(dx, d \gamma, d\delta),
\end{equation}
where $\mu\in \scrP(\bbR_x)$ and $\rho\in \calM(\bbR_x\times\{\pm 1\}\times {\bbR}^+_\delta)$ satisfy
\begin{equation}\label{lp ex: singular}
\int_{\bbR_x} \frac{1}{2} af''(x)  \mu(dx) + \int_{\bbR_x \times\{\pm 1\}\times \bbR^+_\delta} \gamma f'(x) \rho(dx, d \gamma, d \delta) = 0, \quad \forall f\in C_0^2(\bbR_x).
\end{equation}
By Corollary \ref{cor: equiv}, this is equivalent to the time-average control of Brownian motion with quadratic deviation penalty  and proportional costs  in the sense of Definition \ref{def: ta MP}. \\

The dynamics of the solution of the controlled martingale problem is heuristically
\begin{equation*}
dX_t = \sqrt{a}dW_t + \gamma_t d\varphi_t, 
\end{equation*}
where $\gamma_t\in \{\pm 1\}$ and $\varphi_t$ is a non-decreasing process. The optimization objective is
\begin{equation*}
\inf_{(\gamma_t, \varphi_t)\in \mathcal{A}}\limsup_{T\to \infty} \frac{1}{T} \mathbb{E}\Big[\int_0^T r X_t^2  dt + h\varphi_T \Big]. 
\end{equation*}
The associated HJB equation is 
\begin{equation*}
(\frac{1}{2}aw''(x) + r x^2 - I^V) \wedge (\inf_{\gamma \in \{\pm 1\}} \gamma w'(x) +  h)=0.
\end{equation*}
An explicit solution for $w$ is provided in \cite{Soner2012} (see also \cite{jack2006singular, karatzas1983class, dai2013browniana}):
\begin{equation}
w(x) = 
\begin{cases}
Ax^4 + B x^2, &-U \leq x \leq U, \\
w(-U) + h(-U - x), & x< -U, \\
w(U) + h(x - U), & x > U, 
\end{cases}
\end{equation}
with
\begin{equation*}
A = -\frac{1}{6} \frac{r}{a}, \quad B = \frac{I}{a},
\end{equation*}
and
\begin{equation*}
I =  (\frac{3}{4}ar^{1/2}h)^{2/3}, \quad U =  (\frac{3}{4}ar^{-1}h)^{1/3}.
\end{equation*}
The optimally controlled process is 
\begin{equation*}
dX_t^* = \sqrt{a}dW_t  + d\varphi_t^- - d\varphi_t^+, 
\end{equation*}
where $\varphi^{\pm}$ are the local times keeping $X^*_t \in [-U, U]$  such that
\begin{equation*}
\int_0^t \mathbbm{1}_{\{X_s^*\neq-U\}} d\varphi^-_s + \int_0^t \mathbbm{1}_{\{X_s^*\neq U\}} d\varphi^+_s = 0. 
\end{equation*}
In other words,  this is a Brownian motion with reflection on the interval $[-U, U]$. The optimal solution $(\mu^*, \rho^*)$ is the stationary distribution of  $X_t^*$ and the limit of boundary measures
\begin{equation*}
\frac{1}{t}\int_0^t (\delta_{(U, -1, 0)} d\varphi^{+}_s +\delta_{(-U, 1, 0)} d\varphi^{-}_s) , 
\end{equation*}
as $t\to \infty$. We have the following result.
\begin{prop} \label{lem: singular}
The solution of \eqref{cost ex: singular}-\eqref{lp ex: singular} is given by 
\begin{equation*}
I (a, r, h)=  (\frac{3}{4}ar^{1/2}h)^{2/3}.
\end{equation*}
and the optimum is attained by
\begin{align}
\mu^*(dx) & = \frac{1}{2U}\mathbbm{1}_{[-U, U]}(x)dx, \label{eqn: pi reflection} \\
 \rho^*(dx, d \gamma, d \delta)& = \frac{a}{2U}\Big(\frac{1}{2}\delta_{(-U, 1, 0)} +\frac{1}{2} \delta_{(U, -1, 0)}\Big), \label{eqn: rho reflection}
\end{align}
where 
\begin{equation*}
U =  (\frac{3}{4}ar^{-1}h)^{1/3}.
\end{equation*}
\end{prop}
\end{ex}

\begin{ex}[Impulse control of Brownian motion] \label{ex: impulse}
For any parameters $r, k >0$ and $h \geq 0$,  consider  the following linear programming problem
\begin{equation}\label{cost ex: impulse}
I (a, r, k, h)= \inf_{(\mu, \rho)} \int_{\bbR_x} rx^2 \mu(dx) + \int_{\bbR_x\times\bbR_\xi\setminus\{0_{\xi}\}} (k+h|\xi|) \rho(dx, d\xi), 
\end{equation}
where $\mu\in \scrP(\bbR_x)$ and $\rho \in \calM(\bbR_x\times \bbR_\xi\setminus\{0_\xi\})$ satisfy
\begin{equation}\label{lp ex: impulse}
\int_{\bbR_x} \frac{1}{2}a f''(x)  \mu(dx) + \int_{\bbR_x\times \bbR_\xi\setminus\{0_\xi\}} \big(f(x+\xi) - f(x)\big) \rho(dx, d\xi)=0, \quad \forall f\in C_0^2(\bbR). 
\end{equation}
By Corollary \ref{cor: equiv}, this is equivalent to the time-average control problem of Brownian motion  in the sense of Definition \ref{def: ta MP}.\\ 

The associated HJB equation is (see also \cite{dai2013browniana, jack2006impulse, altarovici2013asymptotics, gobet2012almost})
\begin{equation*}
(\frac{1}{2}aw''(x) +  r x^2 - I^V) \wedge (\inf_{\xi\in \bbR_\xi \setminus \{0_\xi\}}( w(x+\xi) + k + h|\xi|) - w(x))=0.
\end{equation*} Let $\theta_1, \theta_2$ and  $0< \wt{x}^*< x^*$ be solutions of the following polynomial system
\begin{equation*}\label{eqn: explicit f+p}
\begin{cases}
6a\theta_1  + r = 0, \\
P(x^*) - P(\wt{x}^*) = k + h(x^*-\wt{x}^*), \\
P'(x^*) = h, \quad 
P'(\wt{x}^*)= h, 
\end{cases}
\end{equation*}
where $P(x) = \theta_1 x^4 + \theta _2 x^2$. Let $U=x^*$.
We can show that  the solution of the HJB equation is given by
\begin{equation*}
w(x) =
\begin{cases}
P(x), &|x|\leq U, \\
w(U)+ h(|x|-U), & |x|>U, 
\end{cases}
\end{equation*}
and
\begin{equation*}
I^V = a \theta_2.
\end{equation*}
Let  
$\xi^*=x^*- \wt{x}^*$. The optimally controlled process is a Brownian motion on the interval $[-U, U]$, which jumps to $\pm U \mp\xi^*$ when reaching the boundary point $\pm U$. Such processes have been studied in \cite{ben2009ergodic, grigorescu2002brownian}. We have the following result.

\begin{prop}\label{lem: impulse}
 The  solution of \eqref{cost ex: impulse}-\eqref{lp ex: impulse} is given by 
\begin{equation*}
I(a, r, k, h) = a \theta_2, 
\end{equation*}
and the optimum is attained by 
\begin{align}
\mu^* &= p(x)dx, \quad p(x) = \begin{cases}
\frac{1}{(x^*)^2 - (\wt{x}^*)^2}(x + x^*), & -x^* \leq x < -\wt{x}^*,\\
\frac{1}{x^*+\wt{x}^*}, & -\wt{x}^* \leq x\leq \wt{x}^*, \\
\frac{1}{(x^*)^2 - (\wt{x}^*)^2}( x^*-x), & \wt{x}^* < x\leq x^*,
\end{cases}\label{eqn: pi jump}\\
\rho^* &= \frac{a}{(x^*)^2 - (\wt{x}^*)^2}(\frac{1}{2}\delta_{(x^*, \wt{x}^* - x^*)}+\frac{1}{2}\delta_{(-x^*,- \wt{x}^* + x^*)}).\label{eqn: rho jump}
\end{align}
which correspond to the stationary distribution and boundary measure of Brownian motion with jumps from the boundary on the interval $[-U, U]$. 
\end{prop}

\end{ex}

\begin{ex}[Combined regular and impulse control of Brownian motion]\label{theo: sol combined impulse}\label{ex: combined impulse}
For any parameters $r, l, k >0$ and $h \geq 0$,  consider  the following linear programming problem
\begin{equation}\label{cost ex: impulse combined}
I (a, r, l, k, h)= \inf_{(\mu, \rho)} \int_{\bbR_x\times\bbR_u} (rx^2 + l u^2) \mu(dx, du) + \int_{\bbR_x\times\bbR_\xi\setminus\{0_{\xi}\}} (k+h|\xi|) \rho(dx, d\xi), 
\end{equation}
where $\mu\in \scrP(\bbR_x\times \bbR_u)$ and $\rho \in \calM(\bbR_x\times \bbR_\xi\setminus\{0_{\xi}\})$ satisfy
\begin{equation}\label{lp ex: impulse combined}
\int_{\bbR_x\times \bbR_u} \big(\frac{1}{2}a f''(x) + uf'(x)\big) \mu(dx, du) + \int_{\bbR_x\times \bbR_\xi\setminus\{0_{\xi}\}} \big(f(x+\xi) - f(x)\big) \rho(dx, d\xi)=0, \quad \forall f\in C_0^2(\bbR). 
\end{equation}
By Corollary \ref{cor: equiv}, this is equivalent to the time-average control problem of Brownian motion  in the sense of Definition \ref{def: ta MP}. \\

The corresponding HJB equation is 
\begin{equation*}
(\frac{1}{2}aw''(x) + \inf_u  (uw'(x) + lu^2 ) + r x^2 - I^V) \wedge \big(\inf_{\xi}( w(x+\xi) + k + h|\xi|) - w(x)\big)=0.
\end{equation*}
In Section \ref{sec: verification}, we show that this equation admits a classical solution
\begin{equation}\label{eqn: val combined ctrl}
w(x) =
\begin{cases}
  (rl)^{1/2} x^2 - 2 a l \ln \Kummer(\frac{1-\iota}{4}; \frac{1}{2};\Big( \frac{r}{a^2 l}\Big ) ^{1/2}x^2), & |x|\leq U,\\
  w(U)+ h(|x|-U) , &|x|> U,
  \end{cases}
\end{equation}
where $\Kummer$ is the Kummer confluent hypergeometric function (see Section \ref{sec: kummer})
and $\xi^*$ and $U$ are such that $0 <\xi^* < U$ and
\begin{equation*}
w(\pm U \mp \xi^*)  + k + h|\xi^*| - w(\pm U) = 0. 
\end{equation*}
Moreover,
\begin{equation*}
I^V = \iota a \sqrt{rl},
\end{equation*}
for some $\iota\in(0, 1)$.  \\

Let $u^*$ be defined as
\begin{equation*}\label{eqn: ctrl u}
u^*(x)=\underset{u}{\text{Argmin }} Aw(x, u) + C_A(x, u) = -\frac{w'(x)}{2l}.
\end{equation*}
The optimally controlled process is given by
\begin{equation*}
dX_t^* = \sqrt{a}dW_t + u^*(X^*_t)dt + d\Big(\sum_{\tau_j \leq t}\big(\mathbbm{1}_{\{X^*_{\tau_j-} = -U\}} \xi^* - \mathbbm{1}_{\{X^*_{\tau_j-} = U\}} \xi^*\big)\Big). 
\end{equation*}
We have the following result.

\begin{prop} \label{lem: combined impulse}
 The  solution of \eqref{cost ex: impulse combined}-\eqref{lp ex: impulse combined} is given by  
\begin{equation*}
I(a, r, l, k, h) = \iota a\sqrt{rl}, 
\end{equation*}
and  the optimum is attained by $(\mu^*, \rho^*)$ where 
\begin{align*}
\mu^*(dx, du) &= p^*(x)dx\otimes \delta_{u^*(x)}(du), \\
 \rho^*(dx, d\xi) & =( \rho^*_- \delta_{(-U, \xi^*)}+ \rho^*_+ \delta_{(U, -\xi^*)})(dx,  d\xi),
\end{align*}
with $p^*(x) \in C^0([L, U])\cap C^2([L, U]\setminus \{L+\xi^*(L), U + \xi^*(U)\})$ solution of 
\begin{equation*}
\begin{cases}
&\frac{1}{2}a p''(x) -(u^*(x)p(x))' =0, \quad x\in (-U, U)\setminus \{-U+ \xi^*,  U - \xi^*\},\\
&p(-U) = p(U) = 0, \\
&\frac{1}{2}a p'((-U)+) = p'((-U+ \xi^*
)-)) -\frac{1}{2}a (p'((-U+ \xi^*)+ ) , \\
&\frac{1}{2}a p'(U-) =    p'((U-\xi^*)+)) -\frac{1}{2}a(p'((U- \xi^*)-), \\
&\int_{-U}^U p(x) = 1,
\end{cases}
\end{equation*}
and $ \rho^*_-, \rho^*_+\in \bbR^+$  given by 
\begin{equation*}\label{eqn: boundary mes jump}
\rho^*_- = \frac{1}{2}a p'((-U)+) , \quad \rho^*_+ = -\frac{1}{2}a p'(U-) .
\end{equation*}
Note that $\mu^*$ and $\rho^*$ are the stationary distribution and boundary measure of the optimally controlled process $X_t^*$. 
\end{prop}

\end{ex}

\begin{ex}[Combined regular and singular control of Brownian motion]\label{ex: combined singular}
For any parameters $r, l, h> 0$, consider the following linear programming problem
\begin{equation}\label{cost ex: singular combined}
I = \inf_{(\mu, \rho)} \int_{\bbR_x\times \bbR_u} (rx^2 + lu^2) \mu(dx, du) + \int_{\bbR_x \times\{\pm 1\}\times \bbR^+_\delta} h |\gamma| \rho(dx, d \gamma, d\delta),
\end{equation}\label{lp ex: singular combined}
where $\mu\in \scrP(\bbR_x\times\bbR_u)$ and $\rho\in \calM(\bbR_x\times\{\pm 1\}\times {\bbR}^+_\delta)$ satisfy
\begin{equation} 
\int_{\bbR_x\times \bbR_u} \big(\frac{1}{2} af''(x) + uf'(x)\big) \mu(dx, du) + \int_{\bbR_x \times\{\pm 1\}\times \bbR^+_\delta} \gamma f'(x) (dx, d \gamma, d \delta) = 0, \quad \forall f\in C_0^2(\bbR_x).
\end{equation}
Again by Corollary \ref{cor: equiv}, this is equivalent to the time-average control problem of Brownian motion in the sense of Definition \ref{def: ta MP}. \\

The constant $I^V$ and $w:\bbR\to \bbR$ exist with the same expressions as in Example \ref{ex: combined impulse}. Similarly, we have the following result. 
\begin{prop}\label{lem: combined singular}
The solution of \eqref{cost ex: singular combined}-\eqref{lp ex: singular combined} is given by
$$I = \iota a \sqrt{rl},$$ 
and the optimum is attained by $(\mu^*, \rho^*)$ where 
\begin{align*}
\mu^*(dx, du)& = p^*(x)dx\otimes \delta_{u^*(x)}(du), \\ \rho^*(dx, d \gamma, d \delta) &=( \rho^*_- \delta_{(-U,1, 0)}+ \rho^*_U \delta_{(U, -1, 0)})(dx, d\gamma, d\delta),
\end{align*}
with $p^*(x) \in  C^2([-U, U])$ solution of 
\begin{equation}\label{eqn: adjoint pde reflection}
\begin{cases}
&\frac{1}{2}a p''(x) -(u^*(x)p(x))' =0, \quad x\in (-U, U),\\
&\frac{1}{2}a p'(x) + u^*(x) p(x) = 0, \quad x\in\{-U, U\},  \\
&\int_{-U}^U p(x) = 1,
\end{cases}
\end{equation}
and   $ \rho^*_-, \rho^*_+\in \bbR^+$ given by
\begin{equation}\label{eqn: boundary mes reflection}
 \rho_-^* =\frac{1}{2}a p(-U),\quad  
\rho_+^*= \frac{1}{2}a p(U).
\end{equation}
\end{prop}

\end{ex}

\begin{rmk}[Higher dimension examples] 
To our knowledge, examples with closed-form solutions for time-average control of Brownian motion in higher dimension are not available except \cite{gobet2012almost, moreau2014trading, altarovici2013asymptotics, Janecek2010, Guasoni2015}. 
\end{rmk}
\section{Relation with utility maximization under market frictions}\label{sec: util}

As we have already mentioned, the lower bound \eqref{eqn: lower bound} appears also in the study of impact of small market frictions in the framework of utility maximization, see \cite{kallsen2013general, kallsen2013portfolio, guasoni2012dynamic, Soner2012, possamai2012homogenization, altarovici2013asymptotics, moreau2014trading, liu2014rebalancing}. 
In this section, we explain heuristically how to relate utility maximization under small market frictions to the problem of tracking. It should be pointed out that we are just making connections between these two problems and no equivalence is rigorously established. \\

We follow the presentation in \cite{kallsen2013general} and consider the classical utility maximization problem
\begin{equation*}
u(t, w_t) = \sup_{\varphi}\bbE[U(w^{t, w_t}_T)],
\end{equation*}
with
\begin{equation*}
w^{t,w_t}_s= w_t+ \int_t^s \varphi_udS_u, 
\end{equation*}
where $\varphi$ is the trading strategy. The market dynamics is an It\=o semi-martingale 
\begin{equation*}
dS_t  = b^S_t dt+ \sqrt{a^S_t}dW_t.
\end{equation*}
In the frictionless market, we denote by $\varphi^*_t$ the optimal strategy and  by $w^*_t$ the corresponding wealth process. As mentioned in \cite{kallsen2013general}, the indirect marginal utility $u'(t, w^*_t)$ evaluated along the optimal wealth process is a martingale density, which we denote by $Z_t$:
\begin{equation*}
Z_t = u'(t, w^*_t).
\end{equation*}
Note that $S$ is a martingale under $\bbQ$ with
\begin{equation*}
\frac{d\bbQ}{d\bbP} =\frac{Z_T}{Z_0}. 
\end{equation*}
 One also defines the indirect risk tolerance process $R_t$ by  
\begin{equation*}
R_t = -\frac{u'(t, w^*_t)}{u''(t, w^*_t)}.
\end{equation*}
Consider the exponential utility function as in \cite{kallsen2013portfolio}, that is
\begin{equation*}
U(x) = -e^{-p x}, \quad p >0.
\end{equation*}
Then we have
\begin{equation*}
R_t =R=\frac{1}{p}.
\end{equation*}

\noindent In a market with proportional transaction costs, the portfolio dynamics is given by 
\begin{equation*}
w^{t, w_t, \varepsilon}_s= w_t^\varepsilon + \int_t^s\varphi^\varepsilon_udS_u - \int_t^s\varepsilon h_u d\|\varphi^\varepsilon\|_u,
\end{equation*}
where $h_t$ is a random weight process and $\varphi^\varepsilon_t$ a process with finite variation. The control problem is then
\begin{equation*}
u^{\varepsilon}(t, w_t) = \sup_{\varphi^\varepsilon}\bbE{[U(w^{t, w_t,\varepsilon}_T)]}.
\end{equation*}
When the cost $\varepsilon$ is small, we can expect that $\varphi^\varepsilon_t$ is close to $\varphi^*_t$ and set
\begin{equation*}
\Delta w^\varepsilon_T := w^{0, w_0,\varepsilon}_T - w^*_T = \int_0^T (\varphi^\varepsilon_t - \varphi^*_t) dS_t - \varepsilon\int_0^T h_t d\|\varphi^\varepsilon\|_t.
\end{equation*}

Then up to first order quantities, we have
\begin{align*}
u^\varepsilon - u 
&= \E{U(w^*_T + \Delta w^\varepsilon_T)} - \E{U(w^*_T)}\\
&\simeq \bbE[{U'(w^*_T) \Delta w^\varepsilon_T + \frac{1}{2}U''(w^*_T)(\Delta w^\varepsilon_T)^2}]\\
&= -u'(w_0) \bbE^\bbQ[{\Delta w^\varepsilon_T +\frac{1}{2R}(\Delta w^\varepsilon_T)^2}]\\
&\simeq -u'(w_0) \bbE^\bbQ[{\varepsilon\int_0^Th_td\|\varphi^\varepsilon\|_t + \frac{1}{2R}({\int_0^T(\varphi^\varepsilon_t - \varphi^*_t)dS_t})^2}]\\
&= -u'(w_0) \bbE^\bbQ[{\varepsilon\int_0^Th_td\|\varphi^\varepsilon\|_t + \int_0^T\frac{a^S_t}{2R}(\varphi^\varepsilon_t - \varphi^*_t )^2dt}].
\end{align*}

Similarly, in a market with fixed transaction costs $\varepsilon k_t$, see \cite{altarovici2013asymptotics}, the portfolio dynamics is given by 
\begin{equation*}
w^{t, w_t, \varepsilon}_s= w_t^\varepsilon + \int_t^s\varphi^\varepsilon_udS_u + \Psi_s - \varepsilon\sum_{t < \tau^\varepsilon_j \leq s} k_{\tau^\varepsilon_j} F(\xi^\varepsilon_j)  , \quad \varphi^\varepsilon_t =\sum_{0 <\tau^\varepsilon_j \leq t} \xi^\varepsilon_{j},
\end{equation*}
and we have
\begin{equation*}
u^\varepsilon - u \simeq -u'(w_0) \bbE^\bbQ[{\varepsilon \sum_{0 <\tau^\varepsilon_j \leq T} k_{\tau^\varepsilon_j}F(\xi^\varepsilon_j)+ \int_0^T\frac{a^S_t}{2R}(\varphi^\varepsilon_t - \varphi^*_t )^2dt}].
\end{equation*}

Finally, in a market with linear impact $\varepsilon l_t$ on price, see \cite{Rogers2007, moreau2014trading}, the portfolio dynamics is given by 
\begin{equation*}
w^{t, w_t, \varepsilon}_s= w_t^\varepsilon + \int_t^s\varphi^\varepsilon_udS_u - \varepsilon\int_t^s l_u (u^\varepsilon_u)^2du, \quad \varphi^\varepsilon_t = \int_0^t u^\varepsilon_tdt,
\end{equation*}
and we have
\begin{equation*}
u^\varepsilon - u \simeq -u'(w_0) \bbE^\bbQ[{\varepsilon\int_0^Tl_t(u^\varepsilon	_t)^2dt + \int_0^T\frac{a^S_t}{2R}(\varphi^\varepsilon_t - \varphi^*_t)^2dt}].
\end{equation*}
\

To sum up, utility maximization under small market frictions is heuristically equivalent to the tracking problem if the deviation penalty is set to be 
\begin{equation}\label{eqn: lambda}
 r_t D(x) = \frac{a_t^S}{2R} x^2.
\end{equation}
Defining the certainty equivalent wealth loss $\Delta^\varepsilon$ by 
\begin{equation*}
u^\varepsilon =: u(w_0 - \Delta^\varepsilon),
\end{equation*}
it follows that 
\begin{equation}\label{correspondence}
\frac{1}{\varepsilon^{\beta\zeta_D}} \Delta^\varepsilon \simeq 	 \bbE^\bbQ[\int_0^T I_t dt],
\end{equation}
see also \cite[Equation (3.4)]{kallsen2013general} and \cite[pp. 18]{moreau2014trading}.

\begin{rmk}[Higher dimension and general utility function]
For the case of higher dimension and general utility function,  one should set 
\begin{equation}\label{eqn: r R}
r_t D(x) = \frac{1}{2R_t} \langle x, a^S_t x\rangle.
\end{equation}
In other words,  utility maximization under market frictions can be approximated at first order by the problem of tracking with quadratic deviation cost \eqref{eqn: r R}. Thus one can establish a connection between the tracking problem and the utility maximization problems  in \cite{altarovici2013asymptotics, moreau2014trading, Guasoni2015, Guasoni2015a}.
\end{rmk}

\begin{rmk}[General cost structures]
When there are multiple market frictions with comparable impacts, the choice of deviation penalty is the same as \eqref{eqn: r R} and one only needs to adjust the cost structure. Our results apply directly in these cases, see \cite{liu2014rebalancing, Guasoni2015a}. For example, in the case of trading with proportional cost and linear market impact, see \cite{liu2014rebalancing}, the local problem is the time-average control of Brownian motion with cost structure
\begin{equation*}
C_A(x, u) = rx^2 + lu^2 +h|u|.
\end{equation*}
Indeed, Equations (4.3)-(4.5) in \cite{liu2014rebalancing} give rise to a verification theorem for the HJB equation of the time-average control problem of Brownian motion under this cost structure. 
\end{rmk}

\begin{rmk}[Non-zero interest rate]
In the case of non-zero interest rate, the correspondence should be written as
\begin{equation}\label{correspondence r}
\frac{1}{\varepsilon^{\beta\zeta_D}} \Delta^\varepsilon \simeq 	 \bbE^\bbQ[\int_0^T e^{-r}I_t dt],
\end{equation}
where $e^{-rt}S_t$ is a $\mathbb{Q}$-martingale and the tracking problem is defined by \eqref{eqn: r R}. For example, the right hand side of \eqref{correspondence r} is the probabilistic representation under Black-Scholes model of Equation (3.11) in \cite{Whalley1997}. 
\end{rmk}

\begin{rmk}[Optimal consumption over infinite horizon]
In \cite{Soner2012, possamai2012homogenization}, the authors consider the problem of optimal consumption over infinite horizon under small proportional costs. Their results can be related to the tracking problem in the same way, that is
\begin{equation*}
\frac{1}{\varepsilon^{\beta\zeta_D}} \Delta^\varepsilon \simeq 	 \bbE^\bbQ[\int_0^\infty e^{-r}I_t dt],
\end{equation*}
where $e^{-rt}S_t$ is a $\mathbb{Q}$-martingale and the tracking problem is defined by \eqref{eqn: r R}. 
\end{rmk}

\section{Proof of Theorem \ref{theo: lower bound}}\label{sec: proof}

This section is devoted to the proof of Theorem \ref{theo: lower bound}. In Section \ref{sec: reduction}, we  first rigorously establish the arguments outlined in Section \ref{sec: asym framework}, showing that it is enough to consider a small horizon $(t, t+\delta ^\varepsilon]$. Then we prove Theorem \ref{theo: lower bound} in Section \ref{sec: proof main}. Our proof is inspired by the approaches in \cite{kurtz1999martingale} and \cite{kushner1993limit}. An essential ingredient is Lemma \ref{lem: characterization}, whose proof is given in Section \ref{sec: proof key lem}.  
\subsection{Reduction to local time-average control problem}\label{sec: reduction}
We first show that,  to obtain \eqref{eqn: lower bound impulse}, it is enough to study the local  time-average control problem (note that the parameters $r_t, l_t, k_t, h_t$  are frozen at time $t$)

\begin{equation}\label{eqn: local cost t}
{I}^\varepsilon_t =\frac{1}{T^\varepsilon}\Big(
\int_0^{T^\varepsilon}\big(r_{t} D(\wt{X}^{\varepsilon,t}_s) +l_t Q(\wt{u}^{\varepsilon, t}_s)\big)ds+ \sum_{0<\wt{\tau}^{\varepsilon,t}_j\leq T^\varepsilon }\big(k_{t}F(\wt{\xi}^\varepsilon_j)+  h_{t}P(\wt{\xi}^\varepsilon_j)\big)\Big),
\end{equation}
where
\begin{equation}\label{eqn: local dynm t}
d\wt{X}^{\varepsilon, t}_s = \wt{b}^{\varepsilon, t}_s ds + \sqrt{\wt{a}^{\varepsilon, t}_s}d\wt{W}^{\varepsilon, t}_s + \wt{u}^{\varepsilon, t}_sds+ d(\sum_{0< \wt{\tau}^{\varepsilon, t}_j \leq s}\wt{\xi}^{\varepsilon}_j),\quad \wt{X}^{\varepsilon, t}_0 = 0
\end{equation}
with $T^\varepsilon = \varepsilon^{-\alpha \beta}\delta^\varepsilon $ and  $\delta^\varepsilon \in \bbR_+$ depending on $\varepsilon$ in such a way that
\begin{equation}
\delta^\varepsilon \to 0, \quad T^\varepsilon \to \infty,
\end{equation}
as $\varepsilon \to 0$.  Recall that $\alpha=2$, which is due to the scaling property of Brownian motion.   We can simply put $\delta^\varepsilon = \varepsilon^{\beta}$. 

\paragraph{Localization} Since we are interested in convergence results in probability, under Assumptions \ref{asmp: model} and \ref{asmp: cost}, it is enough to consider the situation where the following assumption holds.

\begin{asmp}\label{asmp: reduction}
There exists a positive constant $M\in \bbR_+^*$ such that 
\begin{equation*}
\sup_{(t, \omega)\in [0, T]\times \Omega} \abs{ a_t(\omega)} \vee r_t(\omega)^{\pm 1} \vee l_t(\omega)^{\pm 1}  \vee h_t(\omega)^{\pm 1} \vee k_t(\omega) ^{\pm 1} < M < \infty. 
\end{equation*}
Furthermore,  $X^\circ$ is a martingale ($b_t\equiv 0$). 
\end{asmp}

\noindent Indeed, set 
$T_m = \inf\{t > 0, \sup_{s\in [0, t]} \abs{b_s}\vee \abs{ a_s} \vee r_s^{\pm 1}  \vee l_s(\omega)^{\pm 1} \vee h_s^{\pm 1} \vee k_s ^{\pm 1} \leq m \}$. Then we have  $\lim_{m \to \infty} \bbP[T_m = T] =1$. By standard localization procedure, we can assume that all the parameters are bounded as in Assumption \ref{asmp: reduction}.  Let 
\begin{equation*}
\frac{d\bbQ}{d\bbP} = \exp \Big\{-\int_0^T {a^{-1}_t}{b_t}dW_t - \frac{1}{2}\int_0^T {b_t^T}{a_t^{-2}}b_tdt \Big\},
\end{equation*}
then by Girsanov theorem, $X^\circ$ is a martingale under $\bbQ$. Since $\bbQ$ is equivalent to $\bbP$, we only need to prove \eqref{eqn: lower bound impulse} under $\bbQ$. Consequently, we can assume that $X^\circ$ is a martingale without loss of generality.\\

\noindent From now on, we will suppose that Assumption \ref{asmp: reduction} holds.

\paragraph{Locally averaged cost}


\begin{lem} \label{lem: averaged cost}
Under Assumption \ref{asmp: reduction}, we have almost surely 
\begin{equation*}
\liminf_{\varepsilon\to 0}\frac{1}{\varepsilon^{\zeta_D \beta}} J^\varepsilon \geq
\liminf_{\varepsilon\to 0}\int_0^T I^\varepsilon_t dt. 
\end{equation*}
\end{lem}

\begin{proof} We introduce an auxiliary cost functional:
\begin{align*}
\bar{J}^\varepsilon = \int_0^T\Big( &\frac{1}{\delta^\varepsilon}\int_t^{(t+\delta^\varepsilon)\wedge T} \big(r_s D(X^\varepsilon_s) + l_s Q(u^\varepsilon_s)\big) ds \\
&\quad + \frac{1}{\delta^\varepsilon} \sum_{t < \tau^\varepsilon_j \leq (t+\delta^\varepsilon)\wedge T} \big({\varepsilon^{\beta_F}	k_{\tau^\varepsilon_j}F(\xi^\varepsilon_j) + \varepsilon^{\beta_P}h_{\tau^\varepsilon_j}P(\xi^\varepsilon_j)}\big)\Big) dt.
\end{align*}
Note that the parameters $r, l, k, h$ inside the integral are not frozen at $t$. Using Fubini theorem, we have
\begin{align*}
\bar{J}^\varepsilon& = \int_0^T \frac{s}{\delta^\varepsilon}\mathbbm{1}_{\{0<s<\delta^\varepsilon\}}\big(r_s D(X^\varepsilon_s) + l_s Q(u^\varepsilon_s)\big) ds \\
&\qquad +  \sum_{0< \tau^\varepsilon_j \leq  T} 	\frac{\tau^\varepsilon_j}{\delta^\varepsilon}\mathbbm{1}_{\{0<\tau^\varepsilon_j<\delta^\varepsilon\}}\big({\varepsilon^{\beta_F} k_{\tau^\varepsilon_j}F(\xi^\varepsilon_j) + \varepsilon^{\beta_P}h_{\tau^\varepsilon_j}P(\xi^\varepsilon_j)}\big).
\end{align*}
Hence
\begin{equation}\label{eqn: J Jbar}
0\leq \frac{1}{\varepsilon^{\zeta_D \beta}}(J^\varepsilon -\bar{J}^\varepsilon) \leq \delta^\varepsilon I^\varepsilon_0 + \mathbbm{1}_{\{\tau^\varepsilon_0=0\}}\big({\varepsilon^{\beta_F} k_{\tau^\varepsilon_0}F(\xi^\varepsilon_1) + \varepsilon^{\beta_P}h_{\tau^\varepsilon_0}P(\xi^\varepsilon_1)}\big),
\end{equation}
where $I^\varepsilon_0$ is given by \eqref{eqn: local cost t} with $t=0$.\\

On the other hand, we have 
\begin{align*}
\int_0^T I^\varepsilon_t dt =\frac{1}{\varepsilon^{\zeta_D\beta}} \int_0^T \Big(&\frac{1}{\delta^\varepsilon}\int_t^{(t+\delta^\varepsilon)\wedge T}\big(r_t D(X^\varepsilon_s)+ l_tQ(u^\varepsilon_s)\big) ds\\
&\quad + \frac{1}{\delta^\varepsilon} \sum_{t < \tau^\varepsilon_j \leq (t+\delta^\varepsilon)\wedge T} \big(\varepsilon^{\beta_F}k_t F(\xi^\varepsilon_j) + \varepsilon^{\beta_P}h_tP(\xi^\varepsilon_j)\big)\Big)dt,
\end{align*}
where the parameters are frozen at time $t$.
It follows that
\begin{equation}\label{eqn: Jbar Jtilde}
\Big|{\frac{1}{\varepsilon^{\zeta_D \beta}}\bar{J}^\varepsilon -\int_0^T I_t^\varepsilon dt }\Big| \leq  M \cdot{w(r, l, k, h; \delta^\varepsilon)}\cdot \Big( \frac{1}{\varepsilon^{\zeta_D \beta}}\bar{J}^\varepsilon \wedge \wt{J}^\varepsilon\Big),
\end{equation}
with $w$ the modulus of continuity and $M$ the constant in Assumption \ref{asmp: reduction}. Note that we have $w(r, l, k, h; \delta) \to 0+$ as $\delta\to 0+$ by the continuity of $r, l, k, h$. Combining \eqref{eqn: J Jbar} and \eqref{eqn: Jbar Jtilde}, the inequality follows.
\end{proof}

\paragraph{Reduction to local problems}  Using previous lemma, we can reduce the problem to the study of local problems as stated below. 
\begin{lem}[Reduction]
For the proof of Theorem \ref{theo: lower bound}, it is enough to show that 
\begin{equation}\label{eqn: loc lower bound}
\liminf_{\varepsilon\to 0}\E{I^\varepsilon_t} \geq \E{I_t}.	
\end{equation}
\end{lem}
\begin{proof}
In view of Lemma \ref{lem: averaged cost}, to obtain Theorem \ref{theo: lower bound}, we need to prove that
\begin{equation*}
 \liminf_{\varepsilon\to 0} \int_0^TI^\varepsilon_t dt \geq_p \int_0^T I_t dt.
\end{equation*}
By Assumption \ref{asmp: regularity lp} and  Lemma \ref{lem: convergence}, it is enough to show that
\begin{equation*}
\liminf_{\varepsilon\to 0}\E{YI_t^\varepsilon} \geq \E{YI_t},
\end{equation*}
for any bounded random variable $Y$.  Up to a change of notation {\color{black}$r_t \to Yr_t$, $l_t \to Yl_t$, $h_t \to Yh_t$ and $k_t \to Yk_t$ (note that this is allowed since we do not require $r_t$, $l_t$, $h_t$ and $k_t$ to be adapted)}, it suffices to show \eqref{eqn: loc lower bound}.
\end{proof}

\subsection{Proof of Theorem \ref{theo: lower bound}}\label{sec: proof main}

After Section \ref{sec: reduction}, it suffices to prove \eqref{eqn: loc lower bound} where $I^\varepsilon_t$ is given by \eqref{eqn: local cost t}-\eqref{eqn: local dynm t}. In particular, we can assume that
\begin{equation} \label{eqn: bounded cost}
\sup_{\varepsilon> 0} \E{I^\varepsilon_t} < \infty. 
\end{equation}

Combining ideas from \cite{kushner1993limit, kurtz1999martingale}, we first consider the empirical occupation measures of $(\wt{X}^{\varepsilon, t}_s)$. Define the following {random} occupation measures with natural inclusion
\begin{align}
\mu^\varepsilon_t&=\frac{1}{T^\varepsilon}\int_0^{T^\varepsilon}\delta_{\{(\wt{X}^{\varepsilon, t}_s,\wt{u}^{\varepsilon,t}_s)\}}ds \in \scrP(\bbR^d_x\times \bbR^d_u), \nonumber \\
\rho^\varepsilon_t& = \frac{1}{T^\varepsilon} \sum_{0 < \wt{\tau}^{\varepsilon, t}_j \leq T^\varepsilon} \delta_{\{(\wt{X}^{\varepsilon, t}_{\wt{\tau}^{\varepsilon, t}_j-}, \wt{\xi}^{\varepsilon}_j)\}} \in \scrM(\bbR^d_x\times\bbR^d_\xi\setminus\{0_\xi\})  \hookrightarrow \scrM(\overline{\bbR^d_x\times\bbR^d_\xi\setminus\{0_\xi\}}),\nonumber
\end{align}
where $\overline{E}= E \cup \{\infty\}$ is the one-point compactification of $E$. Such compacification of state space appears in \cite{borkar1988ergodic}. See also the proof of Corollary \ref{cor: equiv} where the compactification of state space is used. Note that for $\bar{m}\in \scrM(\overline{E})$ we have the canonical decomposition 
$$
\bar{m}(de) = m(de)+ \theta \delta_{\infty}, 
$$ with $m \in \scrM(E)$ and $\theta \in \bbR^+$.
Second, we define $c_t: \Omega \times \scrP(\bbR^d_x\times \bbR^d_u)\times \scrM(\overline{\bbR^d_x\times \bbR^d_\xi\setminus\{0_\xi\}}) \to \bbR$,
\begin{multline*}
(\omega, \mu, \bar{\rho}) \mapsto \int_{\bbR^d_x\times \bbR^d_u}\big(r_t(\omega) D(x)
+ l_t(\omega) Q(u)\big) \mu(dx\times du) \\+ \int_{\overline{\bbR^d_x\times \bbR^d_\xi\setminus\{0_\xi\}}} \big(k_t(\omega) F(\xi) + h_t(\omega) P(\xi)\big)\bar{\rho}(dx\times d\xi),
\end{multline*}
where the cost functions $F$ and $P$ are extended to $\overline{\bbR^d_x\times \bbR^d_\xi\setminus\{0_\xi\}}$ by setting
 \begin{equation}\label{eqn: extend to infty}
 F(\infty_{(x,\xi)}) = \inf_{\xi \in \bbR^d_\xi\setminus\{0_\xi\}} F(\xi)>0, \quad P(\infty_{(x, \xi)}) =0, 
\end{equation}
 with $\infty_{(x, \xi)}$ the point of compactification for $\bbR^d_x\times \bbR^d_\xi\setminus\{0_\xi\}$. Note that the functions $F$ and $P$ remain l.s.c. on the compactified space, which is an important property we will need in the following. 
Moreover, for any $\bar{\rho} = \rho + \theta_{\bar{\rho}} \delta_{\infty_{(x, \xi)}}$, we have
  \begin{equation}\label{eqn: c infty}
c_t(\omega, \mu, \bar{\rho}) \geq c_t(\omega, \mu, \rho). 
\end{equation}
 Now we have
 \begin{equation}\label{eqn: I  = c}
I_t^\varepsilon = c_t(\mu^\varepsilon_t, \rho^\varepsilon_t),
\end{equation}
  and we can write  \eqref{eqn: bounded cost} as 
\begin{equation}\label{eqn: finite c}
\sup_{\varepsilon > 0}\E{{c_t}(\mu^\varepsilon_t, \rho^\varepsilon_t)}  < \infty.
\end{equation} 
    
\noindent 
The following lemma,  proved in Section \ref{sec: proof key lem},  is the key of the demonstration for Theorem \ref{theo: lower bound}. 
\begin{lem}[Characterization of limits]\label{lem: characterization} Assume that \eqref{eqn: finite c} holds, then 
\begin{enumerate}
\item
The sequence $\{(\mu^\varepsilon_t, \rho^\varepsilon_t)\}$ is tight as a sequence of random variables with values in $\scrP({\bbR^d_x\times \bbR^d_u}) \times \scrM(\overline{\bbR^d_x\times \bbR^d_\xi\setminus\{0_\xi\}})$, equipped with the topology of weak convergence. In particular,  $\{(\mu^\varepsilon_t, \rho^\varepsilon_t)\}$ is relatively compact in $\scrP^\bbP\big(\Omega \times \scrP({\bbR^d_x\times \bbR^d_u}) \times \scrM(\overline{\bbR^d_x\times \bbR^d_\xi\setminus\{0_\xi\}})\big)$, equipped with the topology of stable convergence (see Appendix \ref{sec: convergence notion}). 
\item
Let $\bbQ_t \in \scrP^\bbP\big(\Omega\times \scrP({\bbR^d_x\times \bbR^d_u}) \times \scrM(\overline{\bbR^d_x\times \bbR^d_\xi\setminus\{0_\xi\}})\big)$ be any $\calF$-stable limit of $\{(\mu^\varepsilon_t, \rho^\varepsilon_t)\}$ with disintegration form
\begin{equation*}
\bbQ_t(d\omega, d{\mu}, d\bar{\rho}) = \bbP(d\omega) \bbQ_t^\omega( d{\mu}, d\bar{\rho}). 
\end{equation*}
and
\begin{align*}
S(a)= \Big\{ (\mu, \bar{\rho})&\in \scrP({\bbR^d_x\times \bbR^d_u}) \times \scrM(\overline{\bbR^d_x\times \bbR^d_\xi\setminus\{0_\xi\}}), \\
&\bar{\rho}  = \rho + \theta_{\bar{\rho}} \delta_{\infty} \text{ with } \rho\in \scrM(\bbR^d_x\times \bbR^d_\xi\setminus\{0_\xi\}),   \\
& \int_{\bbR^d_x\times \bbR^d_u} A^af (x, u)\mu(dx, du)  + \int_{\bbR^d_x\times \bbR^d_\xi\setminus\{0_\xi\}} Bf(x, \xi)\rho(dx, d\xi) = 0, \forall f\in C^2_0(\bbR^d_x)\Big.\Big\}, 
\end{align*}
where $A^a$ and $B$ are given by \eqref{eqn: operator A} and \eqref{eqn: operator B impulse}.
Then we have $\bbP$-almost surely, 
$$(\mu, \bar{\rho}) \in S(a_t(\omega)),\quad \bbQ_t^\omega\text{-almost surely}. $$
\end{enumerate}
\end{lem}
  
\noindent By Lemma \ref{lem: characterization}, we have, up to a subsequence, 
\begin{equation*}
(\mu^\varepsilon_t, \rho^\varepsilon_t) \to_\calF { \bbQ}_t \in \scrP^\bbP\big( \scrP(\bbR^d_x\times \bbR^d_u) \times \scrM(\overline{\bbR^d_x\times\bbR^d_\xi\setminus\{0_\xi\}})\big).
\end{equation*}
Write ${\bbQ}_t$ in disintegration form, we have
\begin{equation}\label{eqn: disintegration Q}
{\bbQ}_t(d\omega, d\mu, d\bar{\rho}) = \bbP(d\omega) {\bbQ}_t^\omega(d\mu, d\bar{\rho}).
\end{equation} and $\bbQ_t^\omega$-almost surely,
\begin{equation}\label{eqn: char mu rho}
 (\mu, \rho) \in S(a_t(\omega)).
\end{equation}
 Since the cost functional $c_t$ is lower semi-continuous, we have
\begin{align*}
\quad \liminf_{\varepsilon\to 0}\E{I^\varepsilon_t}
& \overset{\eqref{eqn: I = c}}{=} \liminf_{\varepsilon \to 0} \bbE[c_t(\mu^\varepsilon_t, \rho^\varepsilon_t)]\\
&\overset{\eqref{eqn: stable lsc}}{\geq} \bbE^{\bbQ_t}\crochetb{c_t({\omega, \mu}, \bar{\rho})}\\
&\overset{\eqref{eqn: c infty}}{\geq}\bbE^{\bbQ_t}\crochetb{c_t(\omega, \mu, \rho)}\\
&\overset{\eqref{eqn: disintegration Q}}{=}\int_\Omega \bbP(d\omega) \int_{ \scrP(\bbR_x\times \bbR_u) \times \scrM(\overline{\bbR_x\times\bbR_\xi\setminus\{0_\xi\}})} c_t(\omega, \mu, \rho)\bbQ_t^\omega(d\mu, d\bar{\rho})\\
&\overset{\eqref{eqn: char mu rho}}{=}\int_\Omega \bbP(d\omega) \int_{S(a_t(\omega))} c_t(\omega, \mu, \rho)\bbQ_t^\omega(d\mu, d\bar{\rho})\\
&\overset{\qquad}{\geq} \int_\Omega \bbP(d\omega)\inf_{(\mu, \rho)\in S(a_t(\omega))}c_t(\omega, \mu, \rho).
\end{align*}
Finally, by definition of $I$, we have
\begin{equation*}
\liminf_{\varepsilon\to 0}\E{I^\varepsilon_t} \geq\bbE[{I(a_t, r_t, l_t, k_t, h_t)}].
\end{equation*}

\subsection{Proof of Lemma \ref{lem: characterization}}\label{sec: proof key lem}
First we show the tightness of $\{(\mu^\varepsilon, \rho^\varepsilon)\}$ in $ \scrP(\bbR^d_x\times \bbR^d_u) \times \calM(\overline{\bbR^d_x\times \bbR^d_\xi\setminus\{0_\xi\}})$. A common method is to use tightness functions (see Section \ref{sec: tightness function}). \\

Recall that the cost functions $F$ and $P$ are extended to $\overline{\bbR^d_x\times \bbR^d_\xi\setminus\{0_\xi\}}$ by \eqref{eqn: extend to infty} such that  $c$ is lower semi-continuous. Moreover, $c$ is a tightness function under Assumption \ref{asmp: reduction}, see Section \ref	{sec: tightness function} or \cite[pp. 309]{dupuis2011weak}.  \\

Consequently, if \eqref{eqn: finite c} holds, the family of random measures $\{(\mu^\varepsilon, \rho^\varepsilon)\}$ is tight. Furthermore, by Proposition \ref{prop: stable}, we have
\begin{equation*}
(\mu^\varepsilon_t, \rho^\varepsilon_t) \to_\calF { \bbQ}_t \in \scrP^\bbP( \scrP(\bbR^d_x\times \bbR^d_u) \times \scrM(\overline{\bbR^d_x\times\bbR^d_\xi\setminus\{0_\xi\}})),
\end{equation*}
up to a subsequence,  with
\begin{equation*}
{\bbQ}_t(d\omega, d\mu, d\bar{\rho}) = \bbP(d\omega) {\bbQ}_t^\omega(d\mu, d\bar{\rho}).
\end{equation*}

For the rest of the lemma, we use a combination of the arguments in \cite{kurtz1999martingale} and \cite{kushner1993limit}. Recall that 
\begin{equation*}
A^{a}f(x, u) = \frac{1}{2}\sum_{i, j}a_{ij}\partial^2_{ij} f(x) + {u^T \nabla f(x)}, \quad Bf(x, \xi) = f(x+\xi) - f(x). 
\end{equation*} 
For $f\in C^2_0(\bbR^d_x)$, define 
\begin{align*}
\Psi_t^f(\omega, \mu, \bar{\rho})
&:= \int_{\bbR^d_x\times \bbR^d_u}A^{a_t(\omega)}f(x, u)\mu(dx\times dx) + \int_{{\bbR^d_x\times \bbR^d_\xi\setminus\{0_\xi\}}}Bf(x, \xi){\rho}(dx, d\xi), \quad \bar{\rho} = \rho + \delta_{\bar{\rho}} \delta_\infty.
\end{align*}
Note that $\Psi^f_t$ is well-defined since $\rho\in \calM(\bbR^d_x\times \bbR^d_\xi\setminus\{0_\xi\})$. 
Then we claim that
\begin{equation}\label{eqn: convergence char}
\bbE^\bbQ[{|{\Psi^f_t(\omega, \mu, \bar{\rho})}|}]=\lim_{\varepsilon\to 0}\bbE[{|{\Psi^f_t(\omega, \mu^\varepsilon_t(\omega), \rho^\varepsilon_t(\omega))}|}]=0. 
\end{equation}

Although $\Psi^f(\omega, \cdot, \cdot) \in C(\scrP(\bbR^d_x\times \bbR^d_u)\times \scrM(\overline{\bbR^d_x\times \bbR^d_\xi\setminus\{0_\xi\}}))$, it is not bounded.  The first equality in \eqref{eqn: convergence char} does not follow directly from the definition of stable convergence. However,  by Corollary \ref{cor: equiv}, Condition \eqref{eqn: bound psi c} holds, that is there exists $\beta \in(0, 1)$ and $\theta_f $ non-negative real number depending on $f$ such that
\begin{equation*}
(Af) ^{1/\beta} \leq \theta_f (1 +C_A), \quad (Bf)^{1/\beta} \leq \theta_f C_B. 
\end{equation*}
 By \eqref{eqn: finite c}, we deduce that $\{\Psi^f_t(\omega, \mu^\varepsilon_t, \rho^\varepsilon_t)\}$ is uniformly integrable and  obtain the first equality in \eqref{eqn: convergence char} by \cite[Theorem 2.16]{jacod1981type}.  \\

For the second equality in \eqref{eqn: convergence char}, we apply It\=o's formula to $f(\wt{X}^{\varepsilon, t}_{T^\varepsilon})$ (recall that the dynamics of $\wt{X}^{\varepsilon, t}$ is given by \eqref{pb: local dynamics}) and obtain that
\begin{align*}
f(\wt{X}^{\varepsilon, t}_{T^\varepsilon})
&= f(\wt{X}^{\varepsilon, t}_{0+}) + \int_0^{T^\varepsilon}  f'(\wt{X}^{\varepsilon, t}_s) \sqrt{\wt{a}^{\varepsilon, t}_s} d\wt{W}^{\varepsilon, t}_s \\
&\quad + {\int_0^{T^\varepsilon} \frac{1}{2}\sum_{ij} \wt{a}^{\varepsilon, t}_{ij,s}\partial^2_{ij}f(\wt{X}^{\varepsilon, t}_s)ds}+\int_0^{T^\varepsilon}\sum_{i} \wt{u}^{\varepsilon, t}_{i,s}\partial_{i}f(\wt{X}^{\varepsilon, t}_s)ds\\
&\quad+ \sum_{0 < \wt{\tau}^{\varepsilon, t}_j \leq T^\varepsilon}\big( f(\wt{X}^{\varepsilon, t}_{ \wt{\tau}^{\varepsilon, t}_j-} +  \wt{\xi}^{\varepsilon}_j) -  f(\wt{X}^{\varepsilon, t}_{ \wt{\tau}^{\varepsilon, t}_j-}) \big).
\end{align*}
Combining the definitions of $\mu^\varepsilon$, $\rho^\varepsilon$ and $\Psi^f_t$, we have
\begin{align*}
\bbE[{|{\Psi^f_t(\omega, \mu^\varepsilon_t(\omega), \rho^\varepsilon_t(\omega))}|}]
&\leq \frac{1}{T^\varepsilon} \bbE[|{ f(\wt{X}^{\varepsilon, t}_{T^\varepsilon})-f(\wt{X}^{\varepsilon, t}_{0+}) |] + \frac{1}{T^\varepsilon}\bbE\Big[ \Big|\int_0^{T^\varepsilon}  f'(\wt{X}^{\varepsilon, t}_s) \sqrt{\wt{a}^{\varepsilon, t}_s} d\wt{W}^{\varepsilon, t}_s}\Big|\Big]\\
&\quad +\frac{1}{T^\varepsilon}\bbE[{\int_0^{T^\varepsilon} \frac{1}{2}\sum_{ij} |\wt{a}^{\varepsilon, t}_{ij,s} - \wt{a}^{\varepsilon, t}_{ij,0}|\partial^2_{ij}f(\wt{X}^{\varepsilon, t}_s)ds}].
\end{align*}
By Assumptions \ref{asmp: model} and \ref{asmp: reduction} and dominated convergence, the term on the right hand side converges to zero. Therefore \eqref{eqn: convergence char} holds. \\

By definition of ${\bbQ}^\omega$ and Fubini theorem, we have
\begin{equation*}
 \bbE^\bbQ[{|{\Psi_t^f(\omega, \mu,\bar{ \rho})}|} ] = \bbE^\bbP[\bbE^{{\bbQ}^\omega}[{|{\Psi_t^f(\omega, \mu, \bar{\rho})}|}]]=0.
\end{equation*}
Hence we have $\bbP$-a.e.-$\omega$, $\bbE^{{\bbQ}^\omega}[{|{\Psi_t^f(\omega, \mu, \bar{\rho})}|}]=0$. Let $D$ be a countable dense subset of $C^2_0$.
Since $D$ is countable we have $\bbP$-a.e.-$\omega$,  $\bbE^{{\bbQ}^\omega}[{|{\Psi_t^f(\omega, \mu, \bar{\rho})}|}]=0$ for all $f\in D$.  Fix $\omega\in \Omega\setminus \calN$ for which the property holds. Again by the same argument, we have ${\bbQ}^\omega$-a.e-$(\mu, \bar{\rho})$., $\Psi_t^f(\omega, \mu, \bar{\rho})=0$ for all $f\in D$. Since $D$ is dense in $C^2_0$, $\Psi_t^f(\omega, \mu, \bar{\rho})=0$ holds for $f\in C^2_0$. 

\section{Proof of Theorem \ref{theo: lower bound singular}}\label{sec: proof singular}

The rescaled process $(\wt{X}^{\varepsilon, t}_s)$ is given by 
\begin{equation*}
d\wt{X}^{\varepsilon, t}_s = \wt{b}^{\varepsilon, t}_s ds  + \sqrt{\wt{a}^{\varepsilon, t}_s} d\wt{W}^{\varepsilon, t}_s +\wt{u}^{\varepsilon, t}_sds+ \wt{\gamma}^{\varepsilon,  t}_{s}d\wt{\varphi}^{\varepsilon, t}_s, 
\end{equation*}
with 
\begin{equation*}
\wt{\gamma}^{\varepsilon, t}_s = \gamma^\varepsilon_{t+ \varepsilon^{2\beta} s}, \quad \wt{\varphi}^{\varepsilon, t}_s = \frac{1}{\varepsilon^\beta}(\varphi^\varepsilon_{t+\varepsilon^{2\beta} s} -\varphi^{\varepsilon}_t).
\end{equation*}
The empirical occupation measure of singular control $\nu^\varepsilon_t$ is defined by
\begin{equation*}
\nu^\varepsilon_t = \frac{1}{T^\varepsilon}\int_0^{T^\varepsilon} \delta_{\{(\wt{X}^{\varepsilon, t}_s, \wt{\gamma}^{\varepsilon, t}_{s}, \Delta \wt{\varphi}^{\varepsilon, t}_s)\}}d\wt{\varphi}^{\varepsilon, t}_s.
\end{equation*}
while $\mu^\varepsilon_t$ is defined in the same way as previously.\\

Define the cost functional 
$c_t: \Omega \times \scrP(\bbR^d_x\times \bbR^d_u)\times \scrM(\overline{\bbR^d_x\times \Delta\times \bbR_\delta^+})\to \bbR$:
\begin{multline*}
(\omega, \mu, \bar{\nu}) \mapsto \int_{\bbR^d_x\times \bbR^d_u}(r_t(\omega) D(x)
+ l_t(\omega) Q(u)) \mu(dx\times du) \\\qquad+ \int_{\overline{\bbR^d_x\times\Delta \times  \bbR_\delta^+}} h_t(\omega) P(\gamma)\bar{\nu}(dx\times d\gamma\times d\delta).
\end{multline*}
Then we can show a similar version of Lemma \ref{lem: characterization} and prove Theorem \ref{theo: lower bound singular} with the operator $B$ replaced by \eqref{eqn: operator B singular}. The key ingredients are 
\begin{itemize}
\item
The functional $c$ is a tightness function. 
\item 
The condition \eqref{eqn: bound psi c} holds for  $A$ and $B$ given by \eqref{eqn: operator A singular}-\eqref{eqn: operator B singular}. 
\end{itemize}
To satisfy these two properties is the main reason why we need to use a different operator $B$.

\section{Proof of Propositions \ref{lem: stochastic}-\ref{lem: combined singular}}\label{sec: verification}

In this section, we prove Propositions \ref{lem: stochastic}-\ref{lem: combined singular}.  First, we provide a verification argument tailored to the linear programming formulation in  $\bbR^d$.  Second, we give full details for the proof of Proposition \ref{lem: combined impulse}. The proofs in the remaining cases are exactly the same hence omitted. 

\subsection{Verification theorem in $\bbR^d$}

Consider $A: \calD \to C(\bbR^d_x\times \bbR^d_u)$ and $B: \calD \to C(\bbR^d_x\times V)$  with $\calD = C_0^2(\bbR^d_x)$. The operator $A$ is given by
\begin{equation*}
Af(x, u) = \frac{1}{2} \sum_{ij} a_{ij} \partial^2_{ij}f(x) + \sum_i u_i \partial_i f(x), \quad f\in C_0^2(\bbR^d_2), 
\end{equation*}
The operator $B$ is given by 
\begin{equation*}
Bf(x , \xi ) = f(x+ \xi) - f(x),  
\end{equation*}
if $V= \bbR^d_\xi \setminus \{0_\xi\}$, and by
\begin{equation*}
Bf(x, \gamma, \delta) = \begin{cases}
\langle\gamma,  \nabla f(x)\rangle , &\delta=0, \\
\delta^{-1}\big(f(x+ \delta \gamma) - f(x)\big), &\delta >0,
\end{cases}
\end{equation*}
if $V= \Delta\times \bbR^+_\delta$. \\

Let $C_A: \bbR^d_x\times \bbR^d_u \to \bbR^+$ and $C_B:\bbR^d_x\times V \to \bbR^+$ be two cost functions. We consider the following optimization problem:
\begin{equation}\label{app: LP cost Rd}
I = \inf_{(\mu, \rho)} c(\mu, \rho):= \int_{\bbR^d_x\times \bbR^d_u} C_A(x, u) \mu(dx, du) +\int_{\bbR^d_x\times V} C_B(x, v) \rho(dx, dv), 
\end{equation}
where $(\mu, \rho) \in \scrP(\bbR^d_x\times \bbR^d_u) \times \calM(\bbR^d_x\times V)$ satisfies 
\begin{equation}\label{app: LP constraint Rd}
\int_{\bbR^d_x\times \bbR^d_u}Af(x, u) \mu(dx, du) + \int_{\bbR^d_x\times V}Bf(x, v)\rho(dx, dv) = 0, \quad \forall f\in C_0^2(\bbR^d_x). 
\end{equation}

\begin{lem}[Verification]\label{lem: verification}
Let $w \in C^1(\bbR^d_x)\cap C^2(\bbR^d_x\setminus N) $ so that $Aw$ is well-defined point-wisely for $x\notin N$ and $Bw$ is well-defined for $x\in \bbR^d_x$. 
Assume that
\begin{enumerate}
\item
For each $(\mu, \rho)\in\scrP(\bbR^d_x\times \bbR^d_u)\times \scrM(\bbR^d_x\times V)$ satisfying \eqref{app: LP constraint Rd}  and $c(\mu, \rho)<\infty$,
 we have $\mu(N\times\bbR^d_u ) =0$.
\item 
There exists $w_n\in C_0^2(\bbR^d_x)$ 
such that 
\begin{align*}
&Aw_n(x, u) \to Aw(x, u), \quad \forall (x, u)\in \bbR_x^d\setminus N\times \bbR^d_u, \\
&Bw_n(x, v) \to Bw(x, v), \quad \forall (x, v)\in \bbR^d_x\times V,
\end{align*}
and there exist $\theta\in \bbR^+$ such that
\begin{align*}
&|Aw_n(x, u)| \leq \theta(1+ C_A(x, u)), \quad \forall (x, u)\in (\bbR^d_x\setminus N )\times \bbR^d_u\\
& |Bw_n (x, v)| \leq \theta C_B(x, v), \quad \forall (x, v) \in \bbR^d_x\times V.  
\end{align*}
\item 
There exists a constant $I^V\in \bbR$ such that
\begin{align}
&\inf_{u\in \bbR^d_u} Aw(x, u) + C_A(x, u) \geq I^V , \quad x\in \bbR^d_x\setminus N, \label{eqn: continuation}\\
 &\inf_{v \in V} Bw(x, v) + C_B(x, v)\geq 0, \quad x\in \bbR^d_x. \label{eqn: intervention}
\end{align}
\end{enumerate}
Then we have $I\geq I^V$. \\

If there exists $(\mu^*, \rho^*)$ satisfying the LP constraint and 
\begin{align}
&Aw(x, u) + C_A(x, u) = I^V, \quad \mu^*-a.e. \label{eqn: continuation a e}\\
&Bw(x, v) + C_B(x, v) = 0, \quad \rho^*-a.e. \label{eqn: intervention a e}
\end{align}
then we have
$
I = I^V$. Moreover,  the optimum is attained by $(\mu^*, \rho^*)$ and we call $(w, I^V)$ the value function of the linear programming problem. 
\end{lem}

\begin{proof}[Proof of Lemma \ref{lem: verification}]
Let $(\mu, \rho)$ be any pair satisfying \eqref{app: LP constraint Rd} and $c(\mu, \rho)<\infty$. 
We have
\begin{align*}
&\quad \int_{\bbR^d_x\times\bbR^d_u} Aw(x, u)\mu(dx, du) + \int_{\bbR^d_x\times V} Bw(x, v) \rho(dx, dv) \\
&= \int_{\bbR^d_x\times\bbR^d_u} Aw_n(x, u)\mu(dx, du) + \int_{\bbR^d_x\times V} Bw_n(x, v) \rho(dx, dv) \\
&=0.
\end{align*}
The first term is well-defined since $Aw$ is defined $\mu$-everywhere. The second equality follows from the second condition and dominated convergence theorem. Hence
$$
c(\mu, \rho)=\int_{\bbR^d_x\times \bbR^d_u} (C_A(x, u) + Aw(x, u)) \mu(dx, du) + \int_{\bbR^d_x\times V} (C_B(x, v) + Bw(x, v))\rho(dx, dv)\geq I^V,$$ 
where the last inequality is due to \eqref{eqn: continuation}-\eqref{eqn: intervention}, and the equality holds if and only if \eqref{eqn: continuation a e}-\eqref{eqn: intervention a e} are satisfied. 
\end{proof}

Therefore, finding an explicit solution of a linear programming is possible  if we can determine a suitable value function from \eqref{eqn: continuation a e}-\eqref{eqn: intervention a e}.

\subsection{Verification of Proposition \ref{lem: combined impulse}}

In this section, we provide an explicit solution of  the following linear programming problem:
\begin{equation}\label{app: LP cost}
I(a, r, l, k, h) = \inf_{(\mu, \rho)} \int_{\bbR_x\times\bbR_u} (rx^2 + l u^2) \mu(dx, du) + \int_{\bbR_x\times\bbR_\xi\setminus\{0_{\xi}\}} (k+h|\xi|) \rho(dx, d\xi), 
\end{equation}
where $\mu\in \scrP(\bbR_x\times \bbR_u)$ and $\rho \in \calM(\bbR_x\times \bbR_\xi\setminus\{0_{\xi}\})$ satisfy
\begin{equation}\label{app: LP constraint}
\int_{\bbR_x\times \bbR_u} \big(\frac{1}{2}a f''(x) + uf'(x)\big) \mu(dx, du) + \int_{\bbR_x\times \bbR_\xi\setminus\{0_{\xi}\}} \big(f(x+\xi) - f(x)\big) \rho(dx, d\xi)=0, 
\end{equation}
for any $f\in C_0^2(\bbR)$. 
 The following lemma, whose proof is given in the Appendix, and Theorem \ref{theo: sol combined impulse} establish the existence of the value function $(w, I^V)$ for \eqref{app: LP cost}-\eqref{app: LP constraint}. 
\begin{figure}
\begin{center}
\includegraphics[width =10cm]{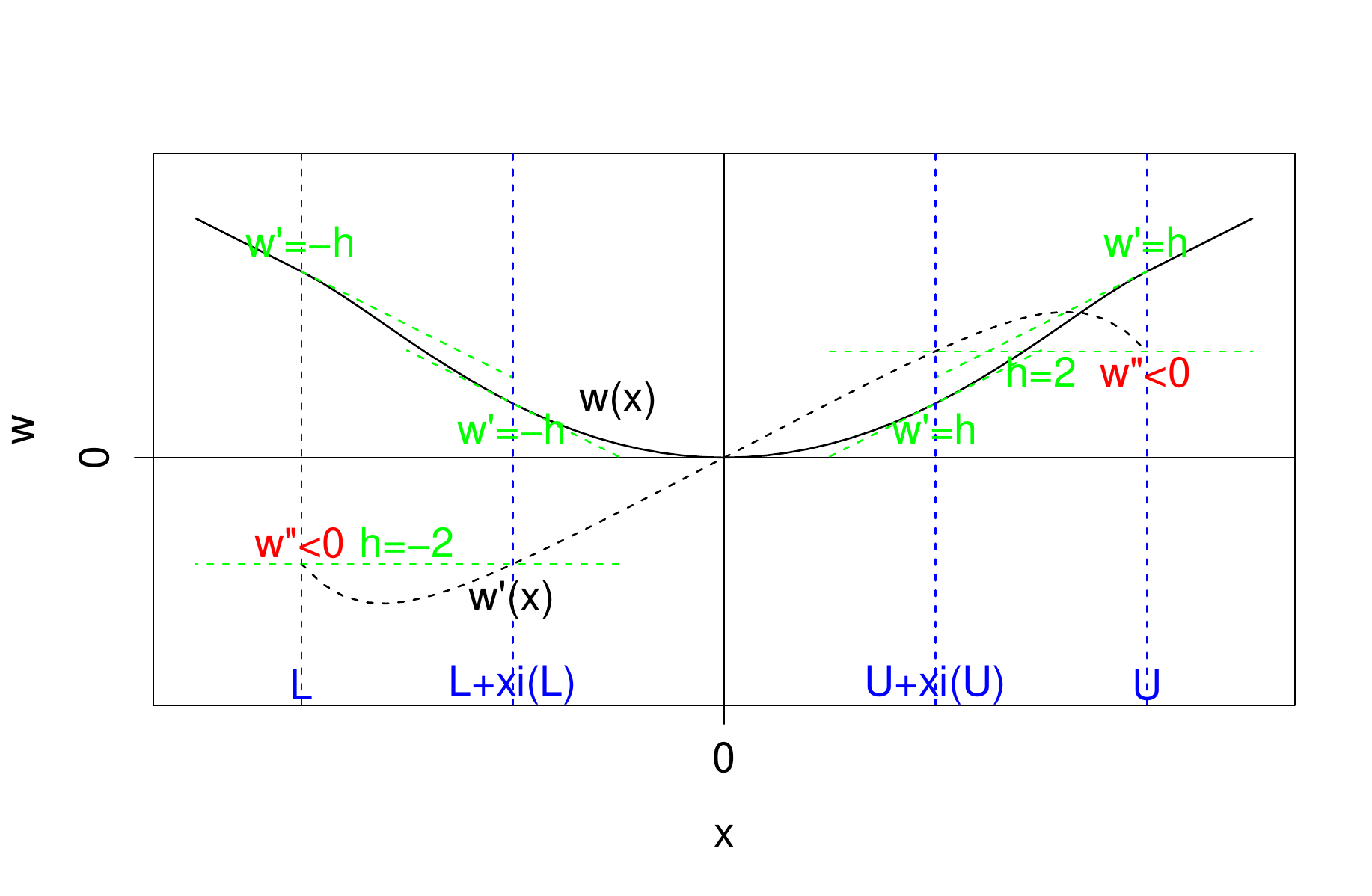}
\end{center}
\caption{Value function for combined regular and impulse control}\label{fig: sol_stoch_impulse}
\end{figure}

\begin{lem}[Value function for combined regular and impulse control]\label{lem: value function} 
There exist $U>\xi^* > 0$,  $I^V >0$, and $w\in  C^1(\bbR)\cap C^2(\bbR\setminus\{U, U\})$ such that
\begin{align}
&Aw(x, u^*(x)) + C_A(x, u^*(x)) = I^V, \quad x\in (-U, U), \label{eqn: pde}\\
&Bw(x, -\text{sgn}(x)\xi^*) + C_B(x, -\text{sgn}(x)\xi^*) = 0, \quad x\in \{-U, U\}, \label{eqn: C0 fit}
\end{align}
where $u^*(x)$ is defined by 
 \begin{equation}\label{eqn: u star}
u^*(x):=\underset{u\in \bbR_u}{\text{Argmin }} Aw(x, u) + C_A(x, u)= -\frac{w'(x)}{2l}.
\end{equation}
More precisely,  we have (c.f. Figure \ref{fig: sol_stoch_impulse})
\begin{equation}\label{eqn: val combined ctrl}
w(x) =
\begin{cases}
  (rl)^{1/2} x^2 - 2 a l \ln \Kummer(\frac{1-\iota}{4}; \frac{1}{2};\Big( \frac{r}{a^2 l}\Big ) ^{1/2}x^2), & |x|\leq U,\\
  w(U)+ h(|x|-U) , &|x|> U,
  \end{cases}
\end{equation}
where $\Kummer$ is the Kummer confluent hypergeometric function (see Section \ref{sec: kummer}) and 
\begin{equation*}
I^V = \iota \sqrt{a^2rl} ,
\end{equation*}
for some $\iota\in(0, 1)$.
Moreover, $w$ satisfies the following conditions 
\begin{align}
&w'(x) = w'(x - \text{sgn}(x) \xi^*)= \text{sgn}(x)h ,\quad x\in \{-U, U\}, \label{eqn: C1 fit} \\
&w''(x) < 0, \quad x\in \{-U, U\}, \label{eqn: cond A}\\
&w'(x) \in \begin{cases} (-\infty,  -h),& -U<x < -U +\xi^*,\\  (-h, h),& -U + \xi^*<x<U -  \xi^*,\\ (h, \infty), &U-\xi^*<x< U, \end{cases}\label{eqn: cond B}
\end{align}
and $w$, $\xi^*$, $U$ and $I^V$ depend continuously on the parameters $(a, r, k, h)$. 
\end{lem}
\begin{rmk}
Equations \eqref{eqn: pde} and \eqref{eqn: C0 fit} correspond essentially to \eqref{eqn: continuation a e} and \eqref{eqn: intervention a e}. The interval $(-U, U)$ is called continuation region.  Equation \eqref{eqn: C1 fit} is the so called ``smooth-fit'' condition and guarantees that $w$ is a $C^1$ function. Equations \eqref{eqn: cond A} and \eqref{eqn: cond B} characterize the growth of the derivatives of $w$ and will be useful in the proof of Theorem \ref{theo: sol combined impulse}.
\end{rmk}

Proposition \ref{lem: combined impulse} is a direct consequence of the following theorem. 

\begin{theo}[Combined regular and impulse control.]\label{theo: sol combined impulse}
For any parameters $a, r, l, k >0$ and $h \geq 0$, we have 
\begin{enumerate}
\item
The pair $(w, I^V)$ in Lemma \ref{lem: value function} is the value function of \eqref{app: LP cost}-\eqref{app: LP constraint} in the sense of Lemma \ref{lem: verification}. 
In particular, the optimal cost of \eqref{app: LP cost}-\eqref{app: LP constraint} is given by  $I = I^V$.

\item  Let  $p^*(x) \in C^0([-U, U])\cap C^2((-U, U)\setminus \{-U+\xi^*, U - \xi^*\})$ be a solution of
\begin{equation}\label{eqn: adjoint pde jump}
\begin{cases}
&\frac{1}{2}a p''(x) -(u^*(x)p(x))' =0, \quad x\in (-U, U)\setminus \{-U+ \xi^*,  U - \xi^*\},\\
&p(-U) = p(U) = 0, \\
&\frac{1}{2}a p'((-U)+) = p'((-U+ \xi^*
)-)) -\frac{1}{2}a (p'((-U+ \xi^*)+ ) , \\
&\frac{1}{2}a p'(U-) =   p'((U-\xi^*)+))-\frac{1}{2}a(p'((U- \xi^*)-) , \\
&\int_{-U}^U p(x) = 1,
\end{cases}
\end{equation}
write $ \rho^*_-, \rho^*_+\in \bbR^+$  for
\begin{equation}\label{eqn: boundary mes jump}
\rho_-^* = \frac{1}{2}a p'((-U)+) , \quad \rho_+^* = -\frac{1}{2}a p'(U-),
\end{equation}
and recall that $u^*$ is  given by \eqref{eqn: u star}.
Then  the optimum of \eqref{app: LP cost}-\eqref{app: LP constraint}  is attained by 
\begin{equation}\label{eqn: optimum jump}
\mu^*(dx, du) = p^*(x)dx\otimes \delta_{u^*(x)}(du), \quad
 \rho^*(dx, d\xi) = \rho^*_- \delta_{(-U, \xi^*)}+ \rho^*_+ \delta_{(U, -\xi^*)}.
\end{equation}
\end{enumerate}
\end{theo}

\begin{proof}
$~$\\
\begin{enumerate}
\item
Consider the function $w$ defined in Lemma \ref{lem: value function}.
First we show that the three conditions in Lemma \ref{lem: verification} are satisfied by $w$.

 i) Note that $N = \{- U, U\}$.  For any $(\mu, \rho)$ satisfying the LP constraint, we show that $\mu(\{x\}\times \bbR_u) =0, \forall x\in \bbR_x$. In particular, $\mu(N\times \bbR_u) = 0$.  Indeed, let $f_n\in C_0^2(\bbR_x)$ be a sequence of test functions such that $f''_n\to \ind{x}$, $\| f_n\|_\infty \vee \|f_n'\|_\infty \to 0$ and there exists $\theta\in \bbR^+$ such that
{\color{black}\begin{equation*}
|Af_n(x,u)| \leq \theta(1+ C_A(x,u)), \quad |Bf_n(x,\xi)| \leq \theta C_B(x,\xi)\quad \forall x\in \mathbb R_x, u\in \mathbb R_u, \xi\in \mathbb R_\xi \setminus\{0_\xi\}. 
\end{equation*}}
For example, let $\varphi\in C^2_0$ with $\varphi''$ being a piece-wise linear function such that $\varphi''(\pm\infty) = \varphi''(-1)=\varphi''(1) = \varphi''(3) = \varphi''(5) = 0$, $\varphi''(0) = \varphi''(4)=1$ and $\varphi''(2)= -2$ and take {\color{black}$f_n(z) = \frac{1}{n^2}\varphi(n(z-x))$.}  Since $c(\mu, \rho) < \infty$, we have by dominated convergence theorem 
\begin{equation*}
\mu(\{x\}\times \bbR_u) = \lim_n \int Af_n(z, u)\mu(dz\times du) + \int Bf_n(z, \xi)\rho(dz\times d\xi) = 0.
\end{equation*}

ii) Let $\varphi_n \in C^2_0$ be a sequence of indicator functions such that $\varphi_n(x) = 1$ for $|x| \leq n$ and $$\sup_n \|\varphi_n\|_{C^2_0}:=\|\varphi_n\|_\infty\vee\|\varphi_n'\|_\infty\vee\|\varphi_n''\|_\infty <\infty.$$
Let $w_n=w\varphi_n$. Then $w_n$ is $C^2$ except at $\{-U, U\}$ and is of compact support. For each $n$, $w_n$ satisfies also the LP constraint
\begin{equation*}
\int Aw_n d\mu + \int Bw_n d\rho = 0.
\end{equation*}
Indeed, let $\varphi_\delta$ be any convolution kernel and  $w_{n,\delta}:=w_n*\varphi_\delta$. So $w_{n, \delta}$ satisfies the LP constraint. Moreover,  $A w_{n, \delta} \to Aw_n$ for $x\notin \{-U, U\}$,  $Bw_{n, \delta}\to Bw_n$ for any $ (x, u)$ and $\sup_\delta \|w_{n, \delta}\|_{C^2_0} < \|w_n\|_{C^2_0} < \infty$. By dominated convergence, $w_n$ satisfies the LP constraint. Finally, a direct computation shows that for some constants $\theta$ and $\theta'$,
\begin{equation*}
|Aw_n| \leq  \theta' \| \varphi_n\|_{C^2_0}(|w|+ |w'| + |w''|)  \leq \theta (1 +C_A), \quad |Bw_n|\leq 2\|\varphi_n\|_{\infty}|w_n| \leq \theta C_B.
\end{equation*}
So the second condition is satisfied. 

iii) By \eqref{eqn: pde} and \eqref{eqn: cond A}, we have $Aw + C_A \geq 0$ for $x\notin \{-U, U\}$. By \eqref{eqn: C0 fit}-\eqref{eqn: cond B} and definition of $w$ outside $[U, U]$, we have $Bw + C_B\geq 0 $.\\

By Lemma \ref{lem: verification}, we then conclude that $I = I^V$. \\

\item We need to show that $\mu^*$ and $\rho^*$ satisfy the LP constraint. Assume that  $\mu^*$ and $\rho^*$ are given by \eqref{eqn: optimum jump}, then 
by integration by parts, the LP constraint holds if 
$p^*(x)$ is solution of \eqref{eqn: adjoint pde jump}. It is easy the see that the latter admits a unique solution. 
\end{enumerate}
\end{proof}

\appendix

\section{Kummer confluent hypergeometric function $\Kummer$}\label{sec: value function}\label{sec: kummer}
We collect here some properties of the Kummer confluent hypergeometric function $\Kummer$ which are useful to establish the existence of value functions of combined control problems in dimension one. Recall that $\Kummer$ is defined as
\begin{equation*}
\Kummer(a, b, z) = \sum_{k=0}^\infty \frac{(a)_k}{(b)_k}\frac{z^k}{k!},
\end{equation*}
with $(a)_k$ the Pochhammer symbol.
\begin{lem}We have the following properties.  \label{lem: kummer}
\begin{enumerate}
\item
The function $\Kummer$ admits the following integral representation
\begin{equation*}
\Kummer(a, b, z) =\frac{\Gamma(b)}{\Gamma(b-a)\Gamma(a)}\int_0^1e^{zt}t^{a-1}(1-t)^{b-a-1}dt.
\end{equation*}
It is an entire function of $a$ and $z$ and a meromorphic function of $b$. 
\item 
We have
\begin{align*}
\frac{\partial }{ \partial z}\Kummer(a, b, z)& = \frac{a}{b}\Kummer(a+1, b+1, z),\\
\frac{\partial }{\partial a} \Kummer(a, b, z)&=\sum_{k=0}^\infty \frac{(a)_k}{(b)_k}\frac{z^k}{k!}\sum_{p=0}^{k-1}\frac{1}{p+a}. 
\end{align*}
\item
We have
\begin{equation*}
(a+1) z \Kummer(a+2, b+2, z) + (b+1)(b-z)\Kummer(a+1, b+1, z) - b(b+1)\Kummer(a, b, z) =0.
\end{equation*}
\item 
We have
\begin{equation*}
\Kummer(a, b, z) = \frac{\Gamma(a)}{\Gamma(b-a)}e^{i\pi a}z^{-a}(1+O(\frac{1}{|z|})) + \frac{\Gamma(b)}{\Gamma(a)}e^z z^{a-b}(1 + O(\frac{1}{|z|})),
\end{equation*}
as $z\to \infty$. 
\item
Consider the Weber differential equation
\begin{equation}\label{eqn: weber}
{w}''(x) - (\frac{1}{4}x^2 +\theta){w}(x)=0.
\end{equation}
The even and odd solutions of this equation are given, respectively, by 
\begin{align}
\wt{w}(x; \theta)& = ~ e^{-\frac{1}{4}x^2}\Kummer(\frac{1}{2}\theta + \frac{1}{4}, \frac{1}{2}, \frac{1}{2}x^2), \label{eqn: weber odd}\\
 \bar{w}(x; \theta) &= x  e^{-\frac{1}{4}x^2}\Kummer(\frac{1}{2}\theta + \frac{3}{4}, \frac{3}{2}, \frac{1}{2}x^2).
\end{align}
\end{enumerate}
\end{lem}
\begin{proof} See \cite{abramowitz1972handbook, ancarani2008derivatives}. \end{proof}
\section{Proof of Lemma \ref{lem: value function}}



    We first look for $w$  in the continuation region $(-U, U)$. Define (the change of variable comes from \cite[pp.260]{oksendal2005})
\begin{equation*}
w(x):= -2 al \ln \wt{w}(\frac{x}{\alpha}; -\frac{\iota}{2}), \quad \alpha^2 = \frac{1}{2}a(r^{-1}l)^{1/2},\quad \iota = \frac{I^V}{(a^2rl)^{1/2}},
\end{equation*}
where $\wt{w}$ is the odd solution \eqref{eqn: weber odd} of the Weber differential equation \eqref{eqn: weber}. 
Then ${w}$ satisfies the following ODE
\begin{equation*}
\frac{1}{2}aw''(x) - \frac{1}{4l}(w'(x))^2 +  rx^2 = I^V,
\end{equation*}
which is exactly \eqref{eqn: pde}. 
Hence we conjecture that the solution in the continuation region $(-U, U)$ is given by 
\begin{align*}
w(x) &= -2 al \ln\Big( e^{-\frac{1}{4}x^2/\alpha^2}\Kummer( \frac{1-\iota}{4}, \frac{1}{2}, \frac{1}{2 \alpha^2}x^2)\Big)\\
&=(rl)^{1/2}x^2 - 2al \ln\Big(\Kummer(\frac{1-\iota}{4}, \frac{1}{2}, \frac{1}{2 \alpha^2}x^2)\Big).
\end{align*}
Now we show that there exist suitable values $U$, $\xi(U)$ and $\iota$  such that  $0 \leq U +\xi(U) \leq U$,  $\iota \in (0, 1)$  and Conditions  \eqref{eqn: C0 fit}-\eqref{eqn: cond B} are satisfied.  Let 
$$
h(x; \iota):= \frac{\partial w}{\partial x}= 2(rl)^{1/2}(1 - (1-\iota)g(\frac{1}{2\alpha^2}x^2; \iota))x,$$
with 
\begin{equation*}
g(z; \iota) = \frac{\Kummer(\frac{1-\iota}{4} +1, \frac{1}{2}+1, z)}{\Kummer(\frac{1-\iota}{4}, \frac{1}{2}, z)}.
\end{equation*}
We have the following lemma.
\begin{lem}\label{lem: g}
The function $g(z;\iota)$ satisfies 
\begin{equation*}
g(z; \iota) \to\begin{cases} 1, & z\to 0^+,\\ \frac{2}{1-\iota}, &z\to +\infty,\end{cases}	
\end{equation*}
and
\begin{equation*}
g'(z; \iota) >0, \quad \forall z\in [0, +\infty). 
\end{equation*}
\end{lem}
\begin{proof}
The limits of $g(z; \iota)$ follow from the asymptotic behaviour of $\Kummer$. To show that $g$ is increasing, we use Properties 2 and 3 in Lemma \ref{lem: kummer} and obtain
\begin{equation*}
g'(z) = g(z)\big(1-\frac{1-\iota}{2}g(z)\big) + \frac{1}{2z}\big(1 -g(z)\big).
\end{equation*}
Note that $g'(0) >0$, so $g>1$ near $x=0$. Since $g'(z) >0$ for $g(z)=1$ and $g'(z) <0$ for $g(z)=\frac{2}{1-\iota}$,  $g(z)$ cannot leave the band $[1, \frac{2}{1-\iota}]$. 
\end{proof}

We now state a second lemma.
\begin{lem}
The function $h(x;  \iota)$ satisfies the following properties. 
\begin{enumerate}
\item
For $\iota \in (0, 1)$, we have
\begin{equation*}
\frac{h(x; \iota)}{x}\to \begin{cases} 2(rl)^{1/2}\iota, & x\to 0+, \\ -2(rl)^{1/2}, &x\to +\infty.\end{cases}
\end{equation*}
Let $0< \bar{x}_\iota <\infty$ be the first zero of $h(x; \iota)$. We have
\begin{equation*}
h'(x; \iota) = \begin{cases} 2(rl)^{1/2} \iota>0, & x=0, \\ -2(rl)^{1/2} 2(1-\iota) \bar{x}_\iota g'(\bar{z}_\iota) <0, & x=\bar{x}_\iota, \end{cases}
\end{equation*}
where $\bar{z}_\iota = \frac{1}{2\alpha^2}\bar{x}^2_\iota$ and 
\begin{equation*}
h''(x; \iota) < 0,\quad x\in [0, \bar{x}_\iota].
\end{equation*}
\item
For $x\in (0, \infty)$, we have
\begin{equation*}
\frac{\partial}{\partial \iota}h(x; \iota) >0
\end{equation*}
and
\begin{equation*}
h(x; \iota) \to 2(rl)^{1/2}x, \quad \iota \to 1-. 
\end{equation*}
We have
\begin{equation*}
\bar{x}_\iota =O(\iota^{1/2}),  \quad \iota \to 0+
\end{equation*}
and hence
\begin{equation*}
\max_{x\in [0, \bar{x}_\iota]}h(x; \iota) \to 0, \quad \iota \to 0+. 
\end{equation*}
\end{enumerate}
\end{lem}
\begin{figure}
\begin{center}
\includegraphics[width =10cm]{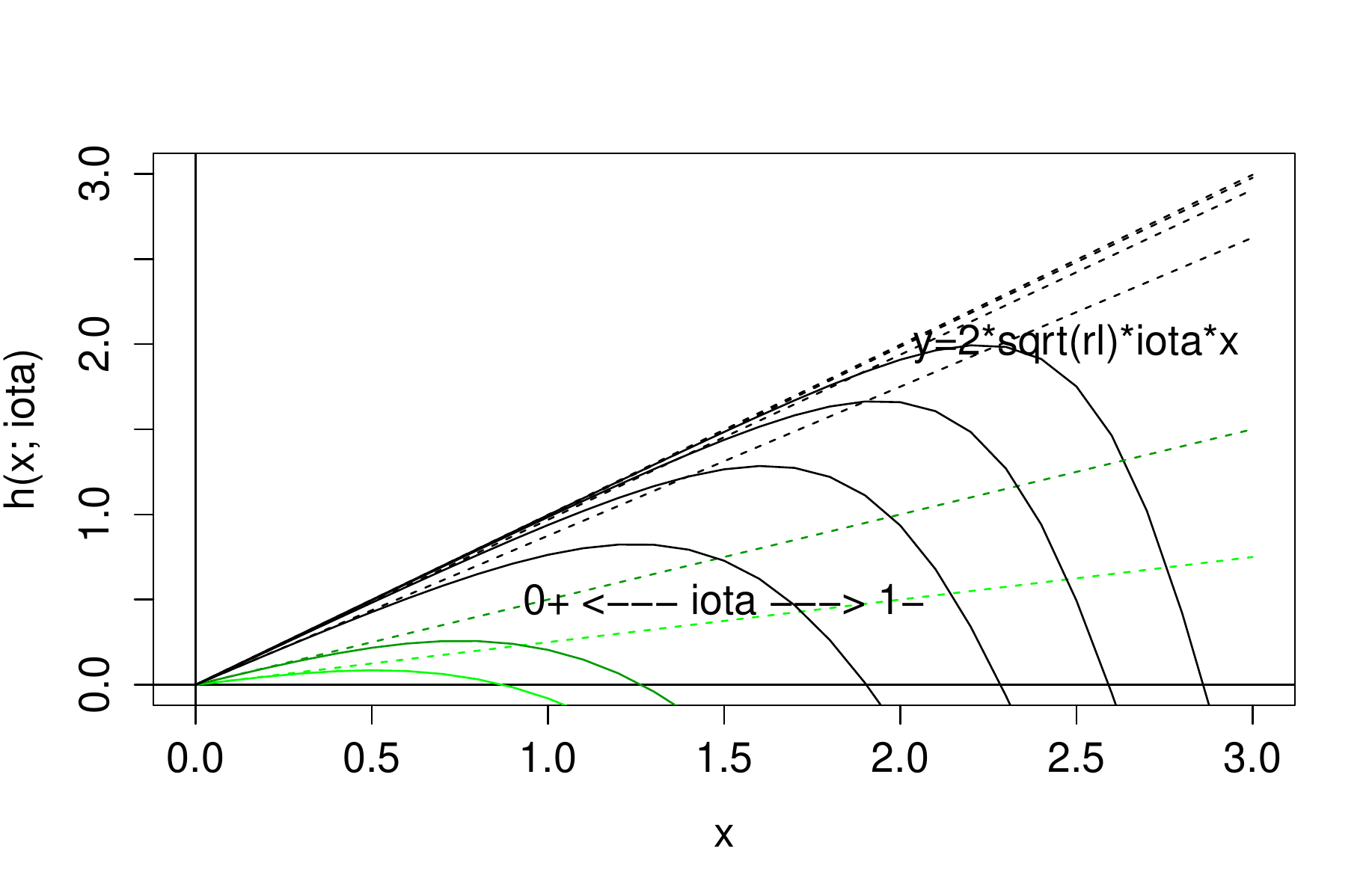}
\end{center}
\caption{Qualitative behaviour of $h(x; \iota)$}
\end{figure}
\begin{proof}
\emph{For fixed $\iota$.} The asymptotic behaviour of $h$ follows from Lemma \ref{lem: g}.  The second property is clear since
\begin{align*}
h'(x; \iota)& =2(rl)^{1/2}\Big(-2(1-\iota)g'(z)z + \big(1- (1-\iota)g(z)\big)\Big)\\
&=2(rl)^{1/2}( -2(1-\iota)g'(z)z) +\frac{ h(x)}{x},
\end{align*}
with $z= \frac{1}{2\alpha^2}x^2$. Finally, we get
$$
h''(x; \iota) = \frac{d}{dz}h'(z; \iota) z'(x)=- 2(rl)^{1/2}(1-\iota)(3g'(z)+ 2g''(z)z) z'(x).
$$Furthermore, we have
\begin{align*}
3g'(z)+ 2g''(z)z&=
3g'(z) + 2z\Big(g'(z)\big(1-(1-\iota)g(z)\big) + \frac{1}{2z^2}\big(g(z) -1-zg'(z)\big)\Big)\\
&=
2zg'(z)\big(1-(1-\iota)g(z)\big) + \frac{1}{z}(g(z) - 1) + 2 g'(z)
\end{align*}
which is strictly positive term by term for $x\in [0, \bar{x}_\iota]$.\\

\emph{For fixed $x$.} The limit of $h$ as $\iota \to 1-$ follows from the fact that $\Kummer(a, b, z)$ is entire in $a$. Now we show that $h$ is monotone in $\iota$. Let $G:= \partial_\iota \Kummer$, we have
\begin{align*}
\partial_\iota h(x; \iota)
&= -2al \frac{\partial}{\partial \iota}\frac{\partial }{\partial x}\ln \Kummer(\frac{1-\iota}{4}; \frac{1}{2}; \frac{1}{2\alpha^2}x^2)\\
& =-2al \frac{\partial}{\partial x}\frac{\partial }{\partial \iota}\ln \Kummer(\frac{1-\iota}{4}; \frac{1}{2}; \frac{1}{2\alpha^2}x^2)\\
&=\frac{1}{2} al \frac{\partial}{\partial x} \frac{G}{\Kummer}(\frac{1-\iota}{4}; \frac{1}{2}; \frac{1}{2\alpha^2}x^2)\\
&= \frac{1}{2}al z'(x)\frac{\Kummer \frac{\partial}{\partial z} G - G \frac{\partial}{\partial z} \Kummer}{\Kummer^2}(\frac{1-\iota}{4}; \frac{1}{2}; z) ,
\end{align*}
with $z(x) = \frac{1}{2 \alpha^2 }x^2$.
It is enough to show that the last term is positive. Using the series representation of $\Kummer$ and $G$ (see Lemma \ref{lem: kummer}), write 
\begin{equation*}
\Kummer = \sum_{k=0}^\infty f_k z^k, \quad G=\sum_{k=0}^\infty \beta_k f_k z^k, 
\end{equation*}
with 
\begin{equation*}
\beta_k = \sum_{p=0}^{k-1}\frac{1}{p+a}, \quad a=\frac{1-\iota}{4}.
\end{equation*}
We get
\begin{align*}
\Kummer \frac{\partial}{\partial x} G - G \frac{\partial}{\partial x} \Kummer
&= \Big(\sum_{i=0}^\infty f_i z^i\Big)\Big(\sum_{j=0}^\infty (j+1)\beta_{j+1} f_{j+1} z^{j}\Big) -\Big (\sum_{i=0}^\infty (i+1)f_{i+1} z^i\Big)\Big(\sum_{j=0}^\infty \beta_j f_j z^j\Big)\\
&= \sum_{k=0}^\infty \Big( \sum_{i+j=k} f_i(j+1)\beta_{j+1} f_{j+1} \Big) z^k - \sum_{k=0}^\infty \Big( \sum_{i+j=k} (i+1)f_{i+1}\beta_{j} f_{j} \Big) z^k.
\end{align*}
Then the coefficient of $z^k$ is given by 
\begin{align*}
&\quad \sum_{i+j=k} f_i(j+1)\beta_{j+1} f_{j+1}  - \sum_{i+j=k} (i+1)f_{i+1}\beta_{j} f_{j} \\
&=
\sum_{i+j=k}(j+1)(\beta_{j+1}-\beta_i) f_i f_{j+1}\\
&=\sum_{1\leq i < j+1 \leq k}\Big( (j+1)(\beta_{j+1}-\beta_i) f_i f_{j+1} - i(\beta_{i}-\beta_{j+1}) f_{j+1} f_{i}\Big)+ \sum_{j+1=k+1}(\cdots) + \sum_{i=j+1}(\cdots)\\
&=\sum_{1\leq i < j+1 \leq k} (j+1-i)(\beta_{j+1}-\beta_i) f_i f_{j+1}+ \sum_{j+1=k+1}(\cdots) + \sum_{i=j+1}(\cdots).
\end{align*}
This term is positive since $\beta_k$ is increasing in $k$.  Hence $\partial_\iota h > 0$. Thus $h(x; \iota)$ is increasing in $\iota$ for fixed $x\in \bbR^+$. \\

From the relation between $g$ and $h$,  $\bar{z}_\iota$ is the first solution of 
\begin{equation*}
1 - (1-\iota) g(z)= 0, \quad z >0. 
\end{equation*}
Moreover, we have
\begin{equation*}
g(z) = 1 + (1 - \frac{1-\iota}{2}) z + o(z), \quad z \to 0+.
\end{equation*}
Then uniformly on $\iota$, $g(z)$ is bounded from below by $1+ \frac{1}{3}z$ on $[0, z_0]$, hence 
\begin{equation*}
\bar{z}_\iota \leq \frac{3\iota}{1- \iota}= O(\iota), \quad \iota \to 0+. 
\end{equation*}
Finally, we have
\begin{equation*}
\max_{[0, \bar{x}_\iota]} h(x; \iota) \leq 2(rl)^{1/2} \bar{x}_\iota \to 0, \quad \iota \to 0+. 
\end{equation*}
\end{proof}

\begin{prop}\label{theo: h}
For any parameters $r, l, h >0$ and $k\geq 0$, there exist $\iota\in (0, 1)$ and $0 \leq U + \xi \leq U$ such that
\begin{align*}
&\int_{U+\xi}^U h(x; \iota) dx = k - h \xi, \\
&h(U; \iota) = h,\\
&h(U+\xi; \iota) = h, \\
&h(x; \iota) \in \begin{cases} (0, h), & 0 \leq x\leq U + \xi, \\
(h, \infty), &U + \xi \leq x\leq U, \end{cases}\\
& h''(x; \iota) < 0,\quad 0\leq x \leq U. 
\end{align*}
Moreover, $(\iota, \xi, U)$ depends continuously on $(r, l, k, h)$. 
\end{prop}
\begin{proof}
\emph{Existence. } Let $k>0$. Since $h(x;\iota)$ is monotone in $\iota$ and $h(x; \iota)\to 2 (rl)^{1/2} x $ as $\iota\to 1-$,  there exists $\underline{\iota}= \underline{\iota}(h)\geq 0$ such that
\begin{equation*}
(\underline{\iota}, 1) = \{\iota \in(0, 1),\text{ there exists exactly two solutions $U_\iota+\xi_\iota$ and $U_\iota$ on $[0, \bar{x}_\iota]$.}\}
\end{equation*}
 We have
\begin{equation*}
h'(U_\iota+\xi_\iota; \iota) >0, \quad h'(U_\iota) < 0,
\end{equation*}
so by the implicit function theorem, $U_\iota$ and $U_\iota+\xi_\iota$ depend continuously on $\iota$. Define 
\begin{equation*}
K(\iota) = \int_{U_\iota+\xi_\iota}^{U_\iota} h(x; \iota) dx.
\end{equation*}
Then $K$ is continuous in $\iota$ and 
\begin{equation*}
\lim_{\iota \to \underline{\iota}} K(\iota) =0, \quad \lim_{\iota\to 1-}K(\iota) = \infty. 
\end{equation*}
Hence there exists $\iota(h, k) \in (\underline{\iota}(h), 1)$ such that
$K(\iota(h, k)) = k$. The remaining property of $h$ is easily verified.\\

If $k=0$, then there exists exactly one $\iota(h)\in (0, 1) $ such that the maximum of $h(x; \iota)$ is $h$ and is attained by $U_\iota$ such that 
\begin{equation*}
h'(U_\iota; \iota) =0. 
\end{equation*}
Since $h''(U_\iota; \iota) <0$, $U_\iota$ depends continuously on $\iota$ by the implicit function theorem. \\

\emph{Continuous dependence.} Since $\xi$ and $U$ depend continuously on $\iota$, it suffices to show that $\iota$ depends continuously on the parameters $a, r, h, k, l$. To see this, note that $\iota=\iota(a, l, r, h, k)$ is determined by
\begin{equation*}
K(\iota; a, r, l, k, h)=k. 
\end{equation*}
But, we have 
\begin{equation*}
\frac{\partial}{\partial \iota} K(\iota; a, r, l, k, h) >0.
\end{equation*}
Thus $\iota$ depends continuously on the parameters by the implicit function theorem. 
\end{proof}

\begin{proof}[Proof of Lemma \ref{lem: value function}]
Extend the function $w$ in Proposition \ref{theo: h} to $\bbR$ by 
\begin{equation*}
w(x) = \begin{cases}
w(|x|), &|x|\leq U, \\
w(U) + h(|x|-U), & |x| > U.
\end{cases}
\end{equation*}
Then \eqref{eqn: pde}-\eqref{eqn: cond A} hold. By \eqref{eqn: C0 fit} and \eqref{eqn: C1 fit}, we have $w\in C^1(\bbR)\cap C^2(\bbR\setminus\{U, U\})$. 
\end{proof}


%
%

\section{Tightness function} \label{sec: tightness function}
For more details on the following results, see \cite[Appendix A.3]{dupuis2011weak} and \cite{bogachev2007measure}. 
\begin{dfn}
A measurable real-valued function $g$ on a metric space $g: (E, d) \to \bbR\cup \{\infty\}$ is a tightness function if 
\begin{enumerate}
\item
$\inf_{x\in E}g(x) > -\infty$. 
\item 
$\forall M < \infty$, the level set $\{x\in E | g(x)\leq M\}$ is a relatively compact subset of $(E, d)$.
\end{enumerate}
\end{dfn}

\begin{lem}\label{lem: tightness function}
If $g$ is a tightness function on a Polish space $E$, then 
\begin{enumerate}
\item
The function $G(\mu) = \int_E g(x)\mu(dx)$ is a tightness function on $\scrP(E)$.
\item 
If in addition $g\geq \delta$ where $\delta$ is a positive constant and $E$ is compact, then 
$G(\mu) = \int_E g(x)\mu(dx)$ is a tightness function on $\calM(E)$.
\end{enumerate}
\end{lem}
\begin{proof}
Note that $\calM(E)$ is a metric space, thus sequential compactness is equivalent to relative compactness (see \cite[pp. 303]{dupuis2011weak} for the metric). 
For the first property, see \cite[pp. 309]{dupuis2011weak}. For the second property,  we consider the level set $\{\mu \in \calM(E)| G(\mu) \leq M\}$ and let $\{\mu_n\}$ be any sequence in the level set. By \cite[Theorem 8.6.2]{bogachev2007measure}, it is enough to show that 
\begin{enumerate}
\item
The sequence of nonnegative real numbers $\mu_n(E)$ is bounded. 
\item 
The family $\{\mu_n\}$ is tight. 
\end{enumerate}
Since $g\geq \delta$, we have $\mu(E) \leq G(\mu)/\delta\leq M/\delta$. Hence the first condition is true. On the other hand, for any $\varepsilon>0$, we consider $\mu_n/\mu_n(E)\in \scrP(E)$. Then $G(\mu_n /\mu_n(E)) \leq M/\mu_n(E) \leq M/\varepsilon$, if $\mu_n(E) >\varepsilon$. Since $G$ is a tightness function, we deduce that $\{\mu_n \text{ }| \text{ } \mu_n(E) >\varepsilon\} $ is tight. Therefore $\{\mu_n\}$ is tight and the second condition follows.
\end{proof}

\section{Convergence in probability, stable convergence}\label{sec: convergence notion}

Let $(\Omega, \calF)$ be a measurable space and $(E, \mathcal{E})$ a Polish space where $\mathcal{E}$ is the Borel algebra of $E$. Define 
\begin{equation*}
\overline{\Omega} = \Omega\times E, \quad \overline{\calF} = \calF \otimes \mathcal{E}.
\end{equation*}
Let $B_{mc}(\overline{\Omega})$ be the set of  bounded mesurable functions $g$ such that $z\mapsto g(\omega, z)$ is a continuous application for any $\omega\in\Omega$. 
Let $M_{mc}(\overline{\Omega})$ be the set of finite positive measures on $(\overline{\Omega}, \overline{\calF})$, equiped with the weakest topology such that
\begin{equation*}
\mu \mapsto \int_{\overline{\Omega}} g(\omega, z) \mu(d\omega, dz),
\end{equation*}
is continuous for any $g\in B_{mc}(\overline{\Omega})$. \\

We fix a probability measure $\bbP$ on $(\Omega, \calF)$.  Let $\scrP^\bbP(\Omega\times E, \calF\otimes \mathcal{E})\subset M_{mc}(\overline{\Omega})$ be the set of probability measures on $\overline{\Omega}$ with marginal $\bbP$ on $\Omega$, equiped with the induced topology from $M_{mc}(\overline{\Omega})$. Note that $\scrP^\bbP(\Omega\times E, \calF\otimes \mathcal{E})$ is a closed subset of $M_{mc}(\overline{\Omega})$. For any random variable $Z$ defined on the probability space $(\Omega, \calF, \bbP)$, we  define
\begin{equation*}
\bbQ^{Z}(d\omega, dz) := \bbP(d \omega)\otimes \delta_{Z(\omega)}(dz) \in \scrP^\bbP(\Omega\times E, \calF\otimes \mathcal{E}).
\end{equation*}
\begin{dfn}[Stable convergence]
Let $\{Z^\varepsilon, \varepsilon > 0\}$ be random variables defined on the same probability space $(\Omega, \calF, \bbP)$ with values in the Polish space $(E, \mathcal{E})$.  
 We say that $Z^\varepsilon$ converges stably in  law to $\bbQ \in \scrP^\bbP(\Omega\times E, \calF\otimes \mathcal{E}) $, written $Z^\varepsilon \to_{\text{stable}} \bbQ$, if $\bbQ^{Z^\varepsilon} \to \bbQ$ in $M_{mc}(\overline{\Omega})$. 
\end{dfn}

We use the following properties in our proofs. 

\begin{prop}\label{prop: stable}
Let $\{Z^\varepsilon, \varepsilon > 0\}$ be random variables on the probability space $(\Omega, \calF, \bbP)$ with values in $(E, \mathcal{E})$. 
\begin{enumerate}
\item 
We have $Z^\varepsilon \to_{\text{stable}}\bbQ\in \scrP^\bbP(\Omega\times E, \calF\otimes \mathcal{E})$ if and only if
\begin{equation*}
\E{Yf(Z^\varepsilon)}\to \bbE^\bbQ\crochetb{Y f(z)},
\end{equation*}
for all bounded random variables  $Y$ on $(\Omega, \calF)$ and all bounded continuous functions $f\in C_b(E, \bbR)$.

\item 
Assume that $Z^\varepsilon \to_{\text{stable}} \bbQ \in \scrP^\bbP(\Omega\times E, \calF\otimes \mathcal{E})$. Then 
\begin{equation}\label{eqn: stable lsc}
\liminf_{\varepsilon \to 0}\bbE[g(Z^\varepsilon)] \geq \bbE^{\bbQ}[g(\omega, z)],
\end{equation}
for any $g$ bounded from below with lower semi-continuous section $g(\omega, \cdot)$ on $E$. 

\item
Let $Z$ be a random variable defined on $(\Omega, \calF, \bbP)$. We have
\begin{equation*}
Z^\varepsilon\to_p Z \Leftrightarrow Z^\varepsilon \to_{\text{stable}} \bbQ^Z. 
\end{equation*}

\item 
The sequence $\{\bbQ^{Z^\varepsilon}, \varepsilon >0\}$ is relatively compact in $\scrP^\bbP(\Omega\times E, \calF\otimes \mathcal{E})$ if and only if $\{Z^\varepsilon, \varepsilon >0\}$ is relatively compact as subset of $\scrP(E)$.  In particular, if $E$ is compact, then $\scrP^\bbP(\Omega\times E, \calF\otimes \mathcal{E})$ is compact. 
\end{enumerate}

 \end{prop}
\begin{proof}
$~$\\
1.  This is a direct consequence of \cite[Proposition 2.4]{jacod1981type}.

2. This is generalization of the Portmanteau theorem, see \cite[Proposition 2.11]{jacod1981type}. 

3. The $\Rightarrow$ implication is obvious. Let us prove the other. Consider $F(\omega, z) = |Z(\omega)- z|\wedge 1$. On the one hand, we have $\bbE[F(\omega, Z^\varepsilon)]\to\bbE^{\bbQ^Z}[F(\omega, z)] = 0$ by definition. On  the other hand, for any $\delta\in (0, 1)$ we have 
\begin{equation*}
 \bbP[|Z^\varepsilon - Z|>\delta]  \leq \bbE[F(\omega, Z^\varepsilon) > \delta]\leq \frac{\bbE[F(\omega, Z^\varepsilon)] }{\delta},
\end{equation*}
by Markov inequality. We deduce that $Z^\varepsilon \to _p Z$. 

4. See \cite[Theorem 3.8 and Corollary 3.9]{jacod1981type}.
\end{proof}

\begin{lem}\label{lem: convergence}
Let $\{Z, Z^\varepsilon, \varepsilon > 0\}$ be \emph{positive} random variables on the probability space $(\Omega, \calF, \bbP)$. 
If, for any random variable $Y$ defined on $(\Omega, \calF, \bbP)$ with $c_Y\leq Y\leq C_Y$ where $c_Y, C_Y$ are positive constants depending on $Y$,
$$\liminf_{\varepsilon\to 0} \E{YZ^\varepsilon} \geq \E{YZ},$$
then
$$\liminf_{\varepsilon\to 0} Z^\varepsilon \geq_p Z.
$$
\end{lem}
\begin{proof}
Let $\delta>0$ be any real number and, without loss of generality,  let $\{Z^\varepsilon\}$  be a minimizing sequence of $\bbP[{Z^\varepsilon > Z - \delta}]$ as $\varepsilon \to 0$. Considering the one-point compactification $\bbR_+ \cup\{\infty \}$, we can assume that $Z^\varepsilon$ converge stably to $\bbQ\in \scrP(\Omega \times (\bbR_+ \cup\{\infty \}))$ with canonical realization $\bar{Z}$. Then we have
\begin{equation*}
\E{Y\bar{Z}} \geq \limsup_{\varepsilon\to 0} \E{YZ^\varepsilon} \geq \E{YZ}, 
\end{equation*}
where the first inequality comes from the fact that $z\mapsto z$ is u.s.c. on $\bbR_+\cup\{\infty\}$ and \cite[Prop 2.11]{jacod1981type}. Since $Y$ is arbitrary, we conclude that $\bar{Z}\geq Z$. Then by stable convergence of $Z^\varepsilon$ to $\bar{Z}$, we have $
\bbP[Z^\varepsilon > Z -\delta ] \to  \bbP[\bar{Z} > Z - \delta ]=1. $
\end{proof}




\begin{thebibliography}{10}

\bibitem{abramowitz1972handbook}
{\sc M.~Abramowitz and I.~A. Stegun}, {\em Handbook of mathematical functions},
  Dover New York, 1972.

\bibitem{altarovici2013asymptotics}
{\sc A.~Altarovici, J.~Muhle-Karbe, and H.~M. Soner}, {\em Asymptotics for
  fixed transaction costs}, Finance and Stochastics,  (2013).

\bibitem{ancarani2008derivatives}
{\sc L.~U. Ancarani and G.~Gasaneo}, {\em Derivatives of any order of the
  confluent hypergeometric function ${}_1{F}_1( a , b , z )$ with respect to
  the parameters $a$ or $b$}, Journal of Mathematical Physics,  (2008).

\bibitem{bank2015hedging}
{\sc P.~Bank, M.~H. Soner, and M.~Vo\ss}, {\em Hedging with transient price
  impact}.
\newblock Preprint, 2015.

\bibitem{ben2009ergodic}
{\sc I.~Ben-Ari and R.~G. Pinsky}, {\em Ergodic behavior of diffusions with
  random jumps from the boundary}, Stochastic processes and their applications,
  119 (2009), pp.~864--881.

\bibitem{bogachev2007measure}
{\sc V.~I. Bogachev}, {\em Measure theory}, vol.~2, Springer Science \&
  Business Media, 2007.

\bibitem{borkar1988ergodic}
{\sc V.~S. Borkar and M.~K. Ghosh}, {\em Ergodic control of multidimensional
  diffusions {I}: The existence results}, SIAM Journal on Control and
  Optimization, 26 (1988), pp.~112--126.

\bibitem{Budhiraja2006}
{\sc A.~Budhiraja and A.~P. Ghosh}, {\em Diffusion approximations for
  controlled stochastic networks: An asymptotic bound for the value function},
  Annals of Applied Probability, 16 (2006), pp.~1962--2006.

\bibitem{budhiraja2012controlled}
\leavevmode\vrule height 2pt depth -1.6pt width 23pt, {\em Controlled
  stochastic networks in heavy traffic: Convergence of value functions}, The
  Annals of Applied Probability, 22 (2012), pp.~734--791.

\bibitem{budhiraja2011ergodic}
{\sc A.~Budhiraja, A.~P. Ghosh, and C.~Lee}, {\em Ergodic rate control problem
  for single class queueing networks}, SIAM Journal on Control and
  Optimization, 49 (2011), pp.~1570--1606.

\bibitem{cadenillas2000classical}
{\sc A.~Cadenillas and F.~Zapatero}, {\em Classical and impulse stochastic
  control of the exchange rate using interest rates and reserves}, Mathematical
  Finance, 10 (2000), pp.~141--156.

\bibitem{dai2013browniana}
{\sc J.~G. Dai and D.~Yao}, {\em Brownian inventory models with convex holding
  cost, part 1: Average-optimal controls}, Stochastic Systems, 3 (2013),
  pp.~442--499.

\bibitem{dupuis2011weak}
{\sc P.~Dupuis and R.~S. Ellis}, {\em A weak convergence approach to the theory
  of large deviations}, vol.~902, John Wiley \& Sons, 2011.

\bibitem{fukasawa2011asymptotically}
{\sc M.~Fukasawa}, {\em Asymptotically efficient discrete hedging}, Stochastic
  Analysis with Financial Applications,  (2011), pp.~331--346.

\bibitem{gobet2012almost}
{\sc E.~Gobet and N.~Landon}, {\em Almost sure optimal hedging strategy}, The
  Annals of Applied Probability, 24 (2014), pp.~1652--1690.

\bibitem{grigorescu2002brownian}
{\sc I.~Grigorescu and M.~Kang}, {\em Brownian motion on the figure eight},
  Journal of Theoretical Probability, 15 (2002), pp.~817--844.

\bibitem{guasoni2012dynamic}
{\sc P.~Guasoni and M.~Weber}, {\em Dynamic trading volume}.
\newblock Available from \texttt{www.ssrn.com}, 2012.

\bibitem{Guasoni2015a}
\leavevmode\vrule height 2pt depth -1.6pt width 23pt, {\em Nonlinear price
  impact and portfolio choice}.
\newblock Available from \texttt{www.ssrn.com}, 2015.

\bibitem{Guasoni2015}
\leavevmode\vrule height 2pt depth -1.6pt width 23pt, {\em Rebalancing multiple
  assets with mutual price impact}.
\newblock Available from \texttt{www.ssrn.com}, 2015.

\bibitem{Helmes2014}
{\sc K.~Helmes, R.~H. Stockbridge, and C.~Zhu}, {\em Impulse control of
  standard {B}rownian motion: long-term average criterion}, System Modeling and
  Optimization,  (2014).

\bibitem{Hynd2012}
{\sc R.~Hynd}, {\em The eigenvalue problem of singular ergodic control},
  Communications on pure and applied mathematics, LXV (2012), pp.~649--682.

\bibitem{Jack2006a}
{\sc A.~Jack and M.~Zervos}, {\em {Impulse and absolutely continuous ergodic
  control of one-dimensional It\^{o} diffusions}}, in From Stochastic Calculus
  to Mathematical Finance: the Shiryaev Festschrift, Y.~Kabanov, R.~Lipster,
  and J.~Stoyanov, eds., Springer, Berlin, 2006, pp.~295--314.

\bibitem{jack2006impulse}
{\sc A.~Jack and M.~Zervos}, {\em Impulse control of one-dimensional
  {I}t{\^{o}} diffusions with an expected and a pathwise ergodic criterion},
  Applied mathematics and optimization, 54 (2006), pp.~71--93.

\bibitem{jack2006singular}
\leavevmode\vrule height 2pt depth -1.6pt width 23pt, {\em A singular control
  problem with an expected and a pathwise ergodic performance criterion},
  International Journal of Stochastic Analysis, 2006 (2006).

\bibitem{jacod1981type}
{\sc J.~Jacod and J.~M\'{e}min}, {\em Sur un type de convergence
  interm\'{e}diaire entre la convergence en loi et la convergence en
  probabilit\'{e}}, S\'{e}minaire de Probabilit\'{e}s XV 1979/80,  (1981),
  pp.~529--546.

\bibitem{jakubowski1997non}
{\sc A.~Jakubowski}, {\em A non-{S}korohod topology on the {S}korohod space},
  Electronic Journal of Probability, 2 (1997), pp.~1--21.

\bibitem{Janecek2010}
{\sc K.~Jane\v{c}ek and S.~Shreve}, {\em Futures trading with transaction
  costs}, Illinois Journal of Mathematics, 54 (2010), pp.~1239--1284.

\bibitem{kallsen2013portfolio}
{\sc J.~Kallsen and S.~Li}, {\em Portfolio optimization under small transaction
  costs: a convex duality approach}, arXiv preprint arXiv:1309.3479,  (2013).

\bibitem{kallsen2013general}
{\sc J.~Kallsen and J.~Muhle-Karbe}, {\em {The general structure of optimal
  investment and consumption with small transaction costs}}, Mathematical
  Finance,  (2015), pp.~1--38.

\bibitem{karatzas1983class}
{\sc I.~Karatzas}, {\em A class of singular stochastic control problems},
  Advances in Applied Probability, 15 (1983), pp.~225--254.

\bibitem{korn1999some}
{\sc R.~Korn}, {\em Some applications of impulse control in mathematical
  finance}, Mathematical Methods of Operations Research, 50 (1999),
  pp.~493--518.

\bibitem{kurtz1991random}
{\sc T.~G. Kurtz}, {\em Random time changes and convergence in distribution
  under the {M}eyer-{Z}heng conditions}, The Annals of probability,  (1991),
  pp.~1010--1034.

\bibitem{kurtz1998existence}
{\sc T.~G. Kurtz and R.~H. Stockbridge}, {\em Existence of {M}arkov controls
  and characterization of optimal {M}arkov controls}, SIAM Journal on Control
  and Optimization, 36 (1998), pp.~609--653.

\bibitem{kurtz1999martingale}
\leavevmode\vrule height 2pt depth -1.6pt width 23pt, {\em Martingale problems
  and linear programs for singular control}, in 37th annual Allerton Conference
  on Communication Control and Computing, 1999.

\bibitem{kurtz2001stationary}
\leavevmode\vrule height 2pt depth -1.6pt width 23pt, {\em Stationary solutions
  and forward equations for controlled and singular martingale problems},
  Electronic Journal of Probability, 6 (2001), pp.~1--52.

\bibitem{kurtz2013linear}
\leavevmode\vrule height 2pt depth -1.6pt width 23pt, {\em Linear programming
  formulations of singular stochastic control problems}.
\newblock Personal communication, 2015.

\bibitem{kushner2001heavy}
{\sc H.~J. Kushner}, {\em Heavy traffic analysis of controlled queueing and
  communication networks}, vol.~47, Springer, 2001.

\bibitem{kushner2014partial}
\leavevmode\vrule height 2pt depth -1.6pt width 23pt, {\em A partial history of
  the early development of continuous-time nonlinear stochastic systems
  theory}, Automatica, 50 (2014), pp.~303--334.

\bibitem{kushner1993limit}
{\sc H.~J. Kushner and L.~F. Martins}, {\em Limit theorems for pathwise average
  cost per unit time problems for controlled queues in heavy traffic},
  Stochastics: An International Journal of Probability and Stochastic
  Processes, 42 (1993), pp.~25--51.

\bibitem{liu2014rebalancing}
{\sc R.~Liu, J.~Muhle-Karbe, and M.~Weber}, {\em Rebalancing with linear and
  quadratic costs}, arXiv preprint arXiv:1402.5306,  (2014).

\bibitem{menaldi1992singular}
{\sc J.~L. Menaldi, M.~Robin, and M.~I. Taksar}, {\em Singular ergodic control
  for multidimensional {G}aussian processes}, Mathematics of Control, Signals
  and Systems, 5 (1992), pp.~93--114.

\bibitem{moreau2014trading}
{\sc L.~Moreau, J.~Muhle-Karbe, and H.~M. Soner}, {\em Trading with small price
  impact}, arXiv preprint arXiv:1402.5304,  (2014).

\bibitem{mundaca1998optimal}
{\sc G.~Mundaca and B.~Oksendal}, {\em Optimal stochastic intervention control
  with application to the exchange rate}, Journal of Mathematical Economics, 29
  (1997), pp.~225--243.

\bibitem{naujokat2011curve}
{\sc F.~Naujokat and N.~Westray}, {\em Curve following in illiquid markets},
  Mathematics and Financial Economics, 4 (2011), pp.~299--335.

\bibitem{oksendal2005}
{\sc B.~Oksendal and A.~Sulem}, {\em Applied stochastic control of jump
  diffusions}, Springer, 2005.

\bibitem{pliska2004optimal}
{\sc S.~R. Pliska and K.~Suzuki}, {\em Optimal tracking for asset allocation
  with fixed and proportional transaction costs}, Quantitative Finance, 4
  (2004), pp.~233--243.

\bibitem{possamai2012homogenization}
{\sc D.~Possama\"{\i}, H.~M. Soner, and N.~Touzi}, {\em {Homogenization and
  asymptotics for small transaction costs: the multidimensional case}},
  Communications in Partial Differential Equations,  (2015), pp.~1--42.

\bibitem{Revuz1999}
{\sc D.~Revuz and M.~Yor}, {\em Continuous martingales and {B}rownian motion},
  Springer, Berlin, 1999.

\bibitem{Rogers2007}
{\sc L.~C. Rogers and S.~Singh}, {\em The cost of illiquidity and its effects
  on hedging}, Mathematical Finance, 20 (2010), pp.~597--615.

\bibitem{rosenbaum2011asymptotically}
{\sc M.~Rosenbaum and P.~Tankov}, {\em Asymptotically optimal discretization of
  hedging strategies with jumps}, The Annals of Applied Probability, 24 (2014),
  pp.~1002--1048.

\bibitem{Soner2012}
{\sc H.~M. Soner and N.~Touzi}, {\em Homogenization and asymptotics for small
  transaction costs}, SIAM Journal on Control and Optimization, 51 (2013),
  pp.~2893--2921.

\bibitem{Whalley1997}
{\sc A.~Whalley and P.~Wilmott}, {\em An asymptotic analysis of an optimal
  hedging model for option pricing with transaction costs}, Mathematical
  Finance, 7 (1997), pp.~307--324.

\end{thebibliography}

\end{document}